\documentclass[letterpaper,11pt]{amsart}
\usepackage{amsmath,amssymb,amsthm,pinlabel,tikz,hyperref,mathrsfs,color, thmtools, epigraph}
\usepackage{verbatim, fullpage, kotex}
\usepackage{float}
\usepackage{caption}
\usepackage{subcaption}
\usepackage{enumitem}
\newcommand{\nc}{\newcommand}
\nc{\dmo}{\DeclareMathOperator}
\dmo{\ra}{\rightarrow}
\dmo{\Prob}{\mathbb{P}}
\dmo{\E}{\mathbb{E}}
\dmo{\N}{\mathbb{N}}
\dmo{\Z}{\mathbb{Z}}
\dmo{\Q}{\mathbb{Q}}
\dmo{\R}{\mathbb{R}}
\dmo{\C}{\mathcal{C}}
\dmo{\X}{\mathcal{X}}
\dmo{\U}{\mathcal{U}}
\dmo{\T}{\mathcal{T}}
\dmo{\F}{\mathcal{F}}
\dmo{\AC}{\mathcal{AC}}
\dmo{\w}{\omega}
\dmo{\MIN}{\mathcal{MIN}}
\dmo{\Mod}{Mod}
\dmo{\PMod}{PMod}
\dmo{\PMF}{\mathcal{PMF}}
\dmo{\Mat}{Mat}
\dmo{\supp}{supp}
\dmo{\UE}{\mathcal{UE}}
\dmo{\vol}{vol}
\dmo{\B}{B}
\dmo{\PB}{PB}
\dmo{\PR}{PSL(2,\mathbb{R})}
\dmo{\GL}{GL(k, \mathbb{C})}
\dmo{\SL}{SL(2, \mathbb{Z})}
\dmo{\Isom}{Isom}
\dmo{\RP}{\mathbb{R} \mathrm{P}}
\dmo{\I}{\mathcal{I}}
\dmo{\el}{\ell_{\C}}
\dmo{\NN}{\mathcal{N}}
\dmo{\rk}{rank}
\dmo{\tr}{tr}
\dmo{\llangle}{\langle\langle}
\dmo{\rrangle}{\rangle\rangle}
\dmo{\Unif}{Unif}
\dmo{\Out}{Out}
\dmo{\Homeo}{Homeo}
\dmo{\Diff}{Diff}
\dmo{\C*}{\mathcal{C}^{\dagger}}
\dmo{\sumRho}{\mathcal{N}}
\dmo{\stopping}{\vartheta}
\dmo{\diam}{\operatorname{diam}}
\dmo{\inte}{\operatorname{int}}
\dmo{\Len}{Len}
\dmo{\Leb}{Leb}

\usetikzlibrary{decorations.markings, patterns}
\tikzset{->-/.style={decoration={
  markings,
  mark=at position #1 with {\arrow{>}}},postaction={decorate}}}
 
\nc{\nt}{\newtheorem}

\nt{theorem}{Theorem}

\newtheorem{thm}{{\bf Theorem}}[section]

\newtheorem{lem}[thm]{{\bf Lemma}}
\newtheorem{cor}[thm]{{\bf Corollary}}
\newtheorem{prop}[thm]{{\bf Proposition}}

\newtheorem{claim}[thm]{Claim} 
\newtheorem{remark}[thm]{Remark}

\newtheorem{dfn}[thm]{Definition}
\numberwithin{equation}{section}

\title[Pivoting technique for $\operatorname{Homeo}(S^{1})$]{Pivoting technique for the circle homeomorphism group}

\date{\today}
\author{Inhyeok Choi}
\email{%
        inhyeokchoi48@gmail.com
        }
\address{%
		School of Mathematics, KIAS\\
		85 Hoegiro Dongdaemun-gu, Seoul 02455, Republic of Korea\\
}

\begin{document}
\begin{abstract}
We adapt Gou{\"e}zel's pivoting technique to the circle homeomorphism group. As an application, we give different proofs of Gilabert Vio's probabilistic Tits alternative and Malicet's exponential synchronization.

\noindent{\bf Keywords.} 

\noindent{\bf MSC classes:} 20F67, 30F60, 57K20, 57M60, 60G50
\end{abstract}

\maketitle

\section{Introduction}

The celebrated Tits alternative asserts that every finitely generated linear group either contains a free subgroup of rank 2 or is virtually solvable \cite{tits1972free}. An analogous statement for finitely generated subgroups of the homeomorphism group $\Homeo(S^{1})$ of the circle $S^{1}$ is not true in general. Indeed, $\Homeo(S^{1})$ contains Thompson's group $F$ as a non-virtually solvable finitely generated subgroup, whose every pair of elements has a nontrivial relation (see \cite{ghys1987sur-un-groupe} for a $C^{\infty}$ example). We are thus led to a weaker, measure-theoretical version of the Tits alternative:

\begin{quote}
Does every subgroup of $\Homeo(S^{1})$ either preserve a probability measure on $S^{1}$ or contain a free subgroup?
\end{quote}

Conjectured by {\'E}tienne Ghys, this question was answered by Gregory Margulis \cite{margulis2000free} (see also \cite{beklaryan2002on-analogues}); Ghys gave another proof in \cite{ghys2001groups}. Let us introduce the notion of Schottky pairs to motivate the ping-pong lemma, a common ingredient in both approaches. We say that subsets $U, V \subseteq S^{1}$ are \emph{essentially disjoint} if their closures are disjoint, i.e., $\bar{U} \cap \bar{V} = \emptyset$.

\begin{dfn}\label{dfn:Schottky}
Let $f_{1}$ and $f_{2}$ be homeomorphisms of $S^{1}$. If there exist essentially disjoint subsets $U_{1}, U_{2}, V_{1}, V_{2}$ of $S^{1}$, each being a finite union of intervals, such that \[
f_{i}(S^{1} \setminus U_{i}) \subseteq V_{i}\,\,\textrm{and}\,\,f_{i}^{-1} (S^{1} \setminus V_{i})  \subseteq U_{i} \quad (i=1, 2),
\]
then we call $(f_{1}, f_{2})$ a \emph{Schottky pair associated with $(U_{1}, U_{2}, V_{1}, V_{2})$}, or simply a \emph{Schottky pair}. 

If each of $U_{1}, U_{2}, V_{1}, V_{2}$ is an interval, we say that $(f_{1}, f_{2})$ is a Schottky pair associated with intervals.
\end{dfn}

Given a subsemigroup $G$ of $\Homeo(S^{1})$, we say that the action of $G$ on $S^{1}$ is \emph{proximal} if $\inf_{g \in G} d(gx, gy) = 0$ for every pair of points $x, y \in S^{1}$. 

The ping-pong lemma, which dates back to Felix Klein's work in the 19th century, asserts that a Schottky pair generates a free group. This is also how Margulis and Ghys established the weak Tits alternative: \begin{thm}[{\cite[Theorem 2]{margulis2000free}, \cite[Section 5.2]{ghys2001groups}}]
Let $G$ be a subgroup of $\Homeo(S^{1})$ that does not admit any invariant probability measure on $S^{1}$. Then $G$ contains a Schottky pair associated with finite unions of intervals. If, moreover, the action of $G$ on $S^{1}$ is proximal, then $G$ contains a Schottky pair associated with intervals.
\end{thm}

In this paper, we consider a generalization of this theorem to subsemigroups of $\Homeo(S^{1})$:

\begin{theorem}[{\cite{malicet2017random}}]\label{thm:main0}
Let $G$ be a subsemigroup of $\Homeo(S^{1})$ that does not admit any invariant probability measure on $S^{1}$. Then $G$ contains a Schottky pair associated with finite unions of intervals. If, moreover, the action of $G$ on $S^{1}$ is proximal, then $G$ contains a Schottky pair associated with intervals.
\end{theorem}

This theorem follows from Dominique Malicet's work on the synchronization \cite[Theorems D and E]{malicet2017random}. We give a direct proof of Theorem \ref{thm:main0} motivated by Margulis and Ghys.

Once we know that there exists a free sub(semi)group of a given sub(semi)group $G$ of $\Homeo(S^{1})$, we can ask if a \emph{random} sub(semi)group of $G$ is free. This is formulated in terms of random walks on $G$. For this, let us consider a Borel probability measure $\mu$ on $\Homeo(S^{1})$. The \emph{support} of $\mu$, denoted by $\supp \mu$, is defined as the complement of the largest $\mu$-null open subset of $\Homeo(S^{1})$. The subsemigroup of $\Homeo(S^{1})$ generated by $\supp \mu$ is denoted by $\llangle \supp \mu \rrangle$. 

In this direction, Mart{\'in} Gilabert Vio recently  proved the following theorem:

\begin{thm}[{\cite[Theorem A]{gilabert-vio2024probabilistic}}]\label{thm:gilabertVio}
Let $\mu_{1}$ and $\mu_{2}$ be discrete probability measures on $\Diff_{+}^{1}(S^{1})$ such that $\llangle \supp \mu_{1} \rrangle$ and $\llangle \supp \mu_{2} \rrangle$ are countable subgroups with proximal actions on $S^{1}$ without global fixed points, and such that the integral \[
\int_{\llangle \supp \mu_{i} \rrangle} \max \left\{ |g|_{Lip}, |g^{-1}|_{Lip} \right\}^{\delta} \,d\mu_{i}(g)
\]
is finite for some $\delta>0$ for $i=1, 2$.

Let $(Z_{n})_{n>0}$ and $(Z_{n}')_{n>0}$ be independent random walks generated by $\mu_{1}$ and $\mu_{2}$, respectively. Then there exists $q \in (0, 1)$ such that \[
\Prob \left( \textrm{$Z_{n}$ and $Z_{n}'$ comprise a Schottky pair associated with intervals} \right) \ge 1-q^{n}
\]
for all $n \in \Z_{>0}$.
\end{thm}

As a consequence, Gilabert Vio proved that independent random walks eventually generate free subgroups almost surely. 

Theorem \ref{thm:gilabertVio} is concerned with random diffeomorphisms in a subgroup with proximal action. A companion result for more general homeomorphisms is as follows.

\begin{thm}[{\cite[Theorem C]{gilabert-vio2024probabilistic}}]\label{thm:gilabertVio2}
Let $\mu_{1}$ and $\mu_{2}$ be discrete probability measures on $\Homeo_{+}^{1}(S^{1})$ such that $\llangle \supp \mu_{1} \rrangle$ and $\llangle \supp \mu_{2} \rrangle$ are countable subgroups without invariant probability measure. Let $(Z_{n})_{n>0}$ and $(Z_{n}')_{n>0}$ be independent (left) random walks generated by $\mu_{1}$ and $\mu_{2}$, respectively. Then the following holds almost surely:  \[
\lim_{N \rightarrow \infty} \frac{1}{N} \# \{0 \le n \le N \, | \, \textrm{$Z_{n}$ and $Z_{n}'$ comprise a Schottky pair} \} =1.
\]
\end{thm}

We now present a strengthening of the above result.

\begin{theorem}\label{thm:main1}
Let $\mu_{1}$ and $\mu_{2}$ be  probability measures  on $\Homeo(S^{1})$ such that the semigroups $ \llangle \supp \mu_{1} \rrangle$ and $ \llangle \supp \mu_{2} \rrangle$ do not admit invariant probability measures on $S^{1}$. Let $(Z_{n})_{n>0}$ and $(Z_{n}')_{n>0}$ be independent random walks generated by $\mu_{1}$ and $\mu_{2}$, respectively. Then there exists $\kappa >0$ such that \begin{equation}\label{eqn:main1}
\Prob \left( \textrm{$Z_{n}$ and $Z_{n}'$ comprise a Schottky pair} \right) \ge 1-\frac{1}{\kappa} e^{-\kappa n}
\end{equation}
for all $n \in \Z_{>0}$.

Furthermore, the constant $\kappa$ is stable under perturbation in the following sense: there exist neighborhoods $\mathcal{U}_{1}$ of $\mu_{1}$ and $\mathcal{U}_{2}$ of $\mu_{2}$ in the space of probability measures on $\Homeo(S^{1})$ (with the weak-$\ast$ topology), respectively, so that Inequality \ref{eqn:main1} holds for a uniform $\kappa>0$ whenever $(Z_{n})_{n>0}$ is driven by a probability measure in $\mathcal{U}_{1}$ and $(Z_{n}')_{n>0}$ is driven by a probability measure in $\mathcal{U}_{2}$.
\end{theorem}

We next study the synchronization of random homeomorphisms of $S^{1}$. Let $f_{1}, f_{2}, \ldots, f_{m}$ be elements of $\Homeo_{+}(S^{1})$. Given $n>0$, each sequence $(\theta(1), \ldots, \theta(n)) \in \{1, \ldots, m\}^{n}$ gives rise to homeomorphism $f_{\theta(n)} \circ \cdots \circ f_{\theta(1)} \in \Homeo(S^{1})$. We can ask if the orbits of a random product of $f_{1}, \ldots, f_{m}$ are \emph{synchronized}, i.e., given any $x, y \in S^{1}$, if $(f_{\theta(n)} \circ \cdots \circ f_{\theta(1)}) (x)$ and $(f_{\theta(n)} \circ \cdots \circ f_{\theta(1)}) (y)$ get closer as $n$ grows for ``most'' choices of $(\theta(1), \theta(2), \ldots)$.

In this context, the semigroup $G_{+}$ generated by $f_{1}, \ldots, f_{m}$ need not be a subgroup of $\Homeo_{+}(S^{1})$. Indeed, the semigroup $G_{+}$ generated by $f_{1}, \ldots, f_{m}$ and the semigroup $G_{-}$ generated by $f_{1}^{-1}, \ldots, f_{m}^{-1}$ can exhibit widely different dynamics (e.g. having distinct minimal sets). This motivates our Theorem \ref{thm:main0} that concerns semigroups.

Assuming that the semigroup generated by $f_{1}, \ldots, f_{m}$ and the semigroup generated by $f_{1}^{-1}, \ldots, f_{m}^{-1}$ both act minimally on $S^{1}$, V. A. Antonov established the following alternatives \cite{MR756386}: either \begin{enumerate}
\item there exists a probability measure on $S^{1}$ preserved by each of $f_{1}, \ldots, f_{m}$ (and it follows that $f_{1}, \ldots, f_{m}$ are simultaneously conjugated to rotations), or  
\item there exists $g \in \Homeo(S^{1})$ of finite order commuting with each of $f_{1}, \ldots, f_{m}$, or 
\item for any i.i.d.s $\theta(1), \theta(2), \ldots$ whose supports are $\{1, \ldots, m\}$, for every pair of points $x, y \in S^{1}$ and for almost every infinite sequence $(\theta(1), \theta(2), \ldots)$, the distance between the trajectories $f_{\theta(n)} \cdots f_{\theta(1)} (x)$ and $f_{\theta(n)} \cdots f_{\theta(1)} (y)$ goes to 0 as $n$ tends to infinity.
\end{enumerate}
In the first case, the $f_{i}$'s simultaneously preserve a metric on $S^{1}$ and distinct points are kept distant. In the second case, a global synchronization cannot be expected but a local synchronization can happen. In the third case, synchronization happens almost surely. An independent work was done by V. A. Kleptsyn and M. B. Nalskii under the additional assumption that $f_{1},\ldots, f_{m}$ generate a hyperbolic map \cite{kleptsyn2004convergence}.

Synchronization in random dynamical systems was also studied by Bertrand Deroin, Victor Kleptsyn and Andr{\'e}s Navas \cite[Proposition 5.2]{deroin2007sur-la-dynamique}, where the authors gave an upper bound on the probability of non-synchronization. The authors then related synchronization in $\mu$-random walk with properties of the $\mu$-stationary measure such as non-atomness and uniqueness. See also \cite{navas2011groups} for further context and overview.

Synchronization was promoted into exponential synchronization by Dominique Malicet in the following form. Note that the minimality assumption is lifted.

\begin{thm}[{\cite[Theorem A]{malicet2017random}}]\label{thm:malicet}
Let $\mu$ be a probability measure on $\Homeo(S^{1})$ such that the semigroup $ \llangle \supp \mu \rrangle$ does not admit any invariant probability measure on $S^{1}$. Let $(Z_{n})_{n>0}$  be the (left) random walk generated by $\mu$. Then there exists $q \in (0, 1)$ such that for each $x \in S^{1}$ and for almost every random path $(Z_{n}(\w))_{n>0}$, there exists a neighborhood $I_{x, \w}$ of $x$ such that \[
\diam \big( Z_{n}(\w) (I_{x, \w}) \big) \le q^{n}
\]
for all $n \in \Z_{>0}$.
\end{thm}

We strengthen this result by providing an exponential bound for exponential synchronization. Below, $\Leb$ denotes the Lebesgue measure on $S^{1}$. \begin{theorem}\label{thm:main2}
Let $\mu$ be a probability measure on $\Homeo(S^{1})$ such that the semigroup $\llangle \supp \mu \rrangle$ does not admit any invariant probability measure on $S^{1}$. Let $(Z_{n})_{n>0}$  be the (left) random walk generated by $\mu$. Then there exist $q \in (0, 1)$ and $\kappa > 0$ such that for each $x \in S^{1}$, \begin{equation}\label{eqn:main2}
\Prob \left( \w: \begin{array}{c} \textrm{there exists an open set $I_{x, \w}$ containing $x$ such that} \\\textrm{ $\Leb(I_{x, \w}) \ge 1- q^{n}$ and $\Leb\big( Z_{k}(\w) (I_{x, \w}) \big) \le q^{k}$ for each $k \ge n$}\end{array}\right) \ge 1-\frac{1}{\kappa} e^{-\kappa n}
\end{equation}
for all $n \in \Z_{>0}$.

Furthermore, the constants $q$ and $\kappa$ are stable under perturbation. That means, there exists a neighborhood $\mathcal{U}$ of $\mu$ in the space of probability measures on $\Homeo(S^{1})$ so that Inequality \ref{eqn:main2} holds for uniform $q$ and $\kappa$ whenever $(Z_{n})_{n>0}$ is driven by some probability measure in $\mathcal{U}$.
\end{theorem}

In the most general case, $I_{x, \w}$ in the theorem above will consist of $N$ open intervals, where $N$ depends on the measure $\mu$. When the action of $\llangle \supp \mu \rrangle$ is proximal, we can take $I_{x, \w}$ as a single open interval.

\begin{theorem}\label{thm:main3}
Let $\mu$ be a probability measure on $\Homeo(S^{1})$ such that $\llangle \supp \mu \rrangle$ does not fix any point in $S^{1}$ and acts on $S^{1}$ proximally. Let $(Z_{n})_{n>0}$  be the (left) random walk generated by $\mu$. Then there exist $q \in (0, 1)$ and  $\kappa > 0$ such that for each $x \in S^{1}$, \begin{equation}\label{eqn:main3}
\Prob \left( \w: \begin{array}{c} \textrm{there exists an interval $I_{x, \w}$ containing $x$ such that} \\\textrm{$\diam(I_{x, \w}) \ge 1 - q^{n}$ and $\diam \big( Z_{k}(\w) (I_{x, \w}) \big) \le q^{k}$ for each $k \ge n$}\end{array}\right) \ge 1-\frac{1}{\kappa} e^{-\kappa n}
\end{equation}
for all $n \in \Z_{>0}$.

Furthermore, the constants $q$ and  $\kappa$ are stable under perturbation. That means, there exists a neighborhood $\mathcal{U}$ of $\mu$ in the space of probability measures on $\Homeo(S^{1})$ so that Inequality \ref{eqn:main3} holds for uniform $q$ and $\kappa$ whenever $(Z_{n})_{n>0}$ is driven by some probability measure in $\mathcal{U}$.
\end{theorem}

The statements in Theorem \ref{thm:main1}, \ref{thm:main2}, \ref{thm:main3} still hold even if the step distributions for the random walk are independent but non-identical, as long as they are distributed according to measures chosen from $\mathcal{U}$ or from $\mathcal{U}_{1}$ and $\mathcal{U}_{2}$, respectively. 

We also have an exponential bound for global synchronization for proximal actions, which strengthens \cite[Theorem E]{malicet2017random}.
\begin{theorem}\label{thm:main4}
Let $\mu$ be a probability measure on $\Homeo(S^{1})$ such that $\llangle \supp \mu \rrangle$ does not fix any point in $S^{1}$ and acts on $S^{1}$ proximally. Let $(Z_{n})_{n>0}$  be the random walk generated by $\mu$. Then there exists $\kappa > 0$ such that for each $x, y \in S^{1}$, \begin{equation}\label{eqn:main4}
\Prob\left( d(Z_{n} x, Z_{n}y) < e^{-\kappa n} \right) \ge 1-\frac{1}{\kappa} e^{-\kappa n}
\end{equation}
for all $n \in \Z_{>0}$.

Furthermore, the constant $\kappa$ is stable under perturbation. That means, there exists a neighborhood $\mathcal{U}$ of $\mu$ in the space of probability measures on $\Homeo(S^{1})$ so that Inequality \ref{eqn:main4} holds for a uniform $\kappa>0$ whenever $(Z_{n})_{n>0}$ is driven by an arbitrary measure in $\mathcal{U}$.
\end{theorem}

In Theorem \ref{thm:main1} or \ref{thm:main4}, it is not important if the random walk is a right random walk or left random walk. Indeed, the estimate is a snapshot at step $n$. Note also that, as in \cite[Theorem A]{malicet2017random}, our results are concerned with homeomorphism groups and do not require higher regularity of the homeomorphisms. The only property of homeomorphisms of $S^{1}$ that we use is the following: if $g \in \Homeo(S^{1})$ and if $I$ and $J$ are nested intervals of $S^{1}$, then $gI$ and $gJ$ are also nested.

Our method is based on S{\'e}bastien Gou{\"e}zel's pivoting technique, which was introduced in \cite{gouezel2022exponential} and led to a remarkable exponential estimate for random walks on Gromov hyperbolic spaces. There have been several attempts to generalize Gou{\"e}zel's technique to a broader setting (see \cite{choi2022random1}, \cite{chawla2022the-poisson}, \cite{peneau2025limit} for example), and this paper is in line with those efforts. We use Schottky dynamics exhibited by Schottky pairs of homeomorphisms to implement Gou{\"e}zel's pivoting time construction. It turns out that the 1-dimensionality of the ambient space is somehow crucial, but a more crucial thing is the nesting of the Schottky regions. Indeed, the particular choice of Lebesgue measure when measuring the diameter of intervals is not important. We have:

\begin{thm}\label{thm:main5}
The statements in Theorem \ref{thm:main2} and  \ref{thm:main3} hold even if the diameter $\diam(\cdot)$ is replaced with $\nu(\cdot)$ for an arbitrary probability measure $\nu$ on $S^{1}$. 
\end{thm}

Above, $\nu$ need not be absolutely continuous with respect to $\Leb$; it could be e.g., a measure concentrated on a Cantor set.

\begin{remark}\label{rem:Gromov}
Since the pivoting technique was originally developed for groups acting on Gromov hyperbolic spaces, the analogue of Theorem \ref{thm:main4} for Gromov hyperbolic spaces also holds. We state it for the record; we will not prove it here, but it can be proven using the pivoting technique. Below, $(\cdot | \cdot)_{o}$ denotes the Gromov product based at $o$.

\begin{prop}\label{prop:main6}
Let $X$ be a Gromov hyperbolic space with basepoint $o$, let $G$ be a group of isometries of $X$, and let $\mu$ be a probability measure on $G$ such that $\llangle \supp \mu \rrangle$ contains two independent loxodromic isometries. Let $(Z_{n})_{n>0}$ be the (left) random walk generated by $\mu$. Then there exists $\kappa>0$ such that  \[
\Prob \left( \w : \begin{array}{c} \textrm{there exists $\xi = \xi(\w, n) \in \partial X$ such that $\big(Z_{k}(\w) \xi' \big| Z_{k}(\w)o\big)_{o} > \kappa k$}\\
\textrm{  for every $k \ge n$ and for every $\xi' \in \partial X$ such that $\big(\xi' \big| \xi\big)_{o} < \kappa n$} \end{array}\right) \ge 1 - \frac{1}{\kappa} e^{-\kappa n}
\]
for each $n > 0$.  Furthermore, the constant $\kappa$ is stable under perturbation of the measure.
\end{prop}
\end{remark}

This can be considered as a counterpart to the main result of \cite{gouezel2022exponential}, which asserts that sample paths of a right random walk escapes to infinity with a linear speed outside a set of exponentially small probability.

\subsection{Sharpness of the results}\label{subsection:sharp}

Theorem \ref{thm:main0} is concerned with subsemigroup $G$ of $\Homeo(S^{1})$. If $G$ does not have any invariant probability measure, then $G$ contains a Schottky pair associated with finite unions of intervals. Conversely, if $G$ preserves a probability measure on $S^{1}$, then $G$ cannot contain a Schottky pair. If the action of $G$ on $S^{1}$ is not proximal, then $G$ cannot contain a Schottky pair associated with intervals (see Theorem \ref{thm:titsAltGen} and \ref{thm:proximalEquiv}).

Theorem \ref{thm:main1} is concerned with probability measures $\mu_{1}, \mu_{2}$ whose supports do not have any invariant probability measure. The desired statement fails if, for example, $\supp \mu_{1}$ is allowed to preserve a probability measure. Indeed, if we consider any $\mu_{1}$ supported on a finite subgroup of $\Homeo(S^{1})$, then $Z_{n} = id$ holds infinitely often almost surely. In that case, $Z_{n}$ and $Z_{n}'$ do not comprise a Schottky pair and do not generate a free subgroup of rank 2.

In Theorem \ref{thm:main2}, we assume that the support of $\mu$ does not have any invariant probability measure. The desired statement fails if $\supp \mu$ is allowed to preserve a probability measure. One counterexample is $\mu$ supported on a finite subgroup of $\Homeo(S^{1})$. Another counterexample is as follows. Fix a point $x \in S^{1}$ and identify $S^{1} \setminus \{x\}$ with $\mathbb{R}$ by a homeomorphism. Then the unit translation $\tilde{\tau} : t \mapsto t + 1$ on $\mathbb{R}$ induces a homeomorphism $\tau : S^{1} \rightarrow S^{1}$ fixing $x$, and generates a cyclic subgroup $\langle \tau \rangle$ of $\Homeo(S^{1})$. If we consider a symmetric nearest-neighbor random walk on $\langle \tau \rangle$, the random walk visits the identity element infinitely often almost surely. Hence, the desired eventual exponential contraction cannot happen almost surely. Note also that the action of $\langle \tau \rangle$ on $S^{1}$ is proximal. Hence, this also serves as a counterexample to Theorem \ref{thm:main3} and Theorem \ref{thm:main4} when there is no assumption about fixed point.

In Theorem \ref{thm:main2} and \ref{thm:main3}, it is important that the choice of $I_{x, \w}$ depends on the sample $\w$. It is easy to construct a random walk (say, a nearest-neighbor random walk on a surface group acting on $S^{1} = \partial \mathbb{H}^{2}$) such that for any nonempty open set $O$, there exists $\epsilon > 0$ such that \[
\Prob \big(\diam (Z_{n} \cdot O) > 1/2\big) > \epsilon
\]
for every sufficiently large $n \in \Z_{>0}$.

\subsection*{Organization of the paper.} In Section \ref{section:prelim}, we review the basics and fix the notation for intervals, semigroups and probability measures on $\Homeo(S^{1})$. In Section \ref{section:weakTits} we prove Theorem \ref{thm:main0}, the weak Tits alternative for semigroups, and its converse. For this, we carefully study $G$-contractible intervals, $G$-repellers and their relationship with $G$-invariant measures.

Using Schottky pairs for random walks furnished by Theorem \ref{thm:main0}, we prove in Section \ref{section:SchottkyRandom} the probabilistic Tits alternative and synchronization statements, assuming the pivoting proposition. Section \ref{section:pivoting} proves that proposition by adapting Gou{\"e}zel's pivotal time construction to nested intervals.

\subsection*{Acknowledgments}
The author thanks Hyungryul Baik and Sang-hyun Kim for their valuable comments on the Korean version of this paper. The author thanks Dominique Malicet and Mart{\'i}n Gilabert Vio for explaining the basics of circle homeomorphisms and clarifying the author's many confusions. Malicet and Gilabert Vio encouraged the author to consider non-degenerate random walks on subsemigroups (which are not necessarily subgroups) of $\Homeo(S^{1})$, which led to the current statements. The author is grateful to their suggestions. The author is grateful to Gilabert Vio for many corrections and suggestions for the first draft of this paper, and to {\'E}tienne Ghys, Victor Kleptsyn and Andr{\'e}s Navas for their comments on the first draft of this paper and reference suggestions. Lastly, the paper was greatly improved based on the referees' comments and suggestions, and the author is grateful to them. 

The author was partially supported by Mid-Career Researcher Program (RS-2023-00278510) through the National Research Foundation funded by the government of Korea. The first draft of the paper was written while the author was visiting the HCMC of KIAS.

\section{Preliminaries}\label{section:prelim}

In this paper, $S^{1}$ denotes the circle. We regard $S^{1}$ as the quotient of $\mathbb{R}$ by the translation $t \mapsto t+1$. This implicitly endows the Lebesgue measure $\Leb$ and the orientation on $S^{1}$. More explicitly, we say that distinct points $x_{1}, x_{2}, \ldots, x_{N} \in S^{1}$ are \emph{oriented counterclockwise} if there exist lifts $\tilde{x}_{1}, \ldots, \tilde{x}_{N}$ on $\mathbb{R}$ of $x_{1}, \ldots, x_{N}$, respectively, such that \[
\tilde{x}_{1} < \tilde{x}_{2} < \ldots < \tilde{x}_{N}, \quad \tilde{x}_{N} - \tilde{x}_{1} < 1.
\]
For distinct points $x, y \in S^{1}$, let $\tilde{x}$ and $\tilde{y}$ be lifts on $\mathbb{R}$ of $x$ and $y$, respectively, such that $\tilde{x} < \tilde{y} < \tilde{x} + 1$. We define the open interval $(x, y)$ in $S^{1}$ by the image of $(\tilde{x}, \tilde{y}) \subseteq \mathbb{R}$ on $S^{1}$ by the quotient map. Equivalently, \[
(x, y) := \big\{ z \in S^{1} : x, z, y\,\,\textrm{is oriented counterclockwise} \big\}.
\]
We similarly define the closed interval $[x, y]$ and the half-open interval $(x, y]$. We use $\sqcup$ to denote the disjoint union.

We denote by $\Homeo(S^{1})$ the group of homeomorphisms from $S^{1}$ to itself. A subset $G$ of $\Homeo(S^{1})$ is called a \emph{subsemigroup} if the composition is closed in $G$, i.e., $f\circ g \in G$ for each $f, g \in G$. We call $G$ a \emph{subgroup} of $\Homeo(S^{1})$ if it moreover satisfies that $f^{-1} \in G$ for each $f \in G$.

Let $A$ and $B$ be subsets of $S^{1}$. We say that $A$ and $B$ are \emph{essentially disjoint} if their closures $\bar{A}$ and $\bar{B}$ are disjoint. We say that $A$ \emph{essentially contains} $B$ if $\bar{B} \subseteq \inte A$.

Throughout the article, a \emph{neighborhood} of $A$ always refers to an open one, i.e., an open set containing $A$. We define the \emph{$\epsilon$-neighborhood of $A$} by \[
N_{\epsilon}(A) := \big\{ x \in S^{1} : \exists a \in A \,\,\textrm{such that} \,\,d(x, a) < \epsilon\big\}.
\]

Given a subset $A \subseteq S^{1}$, we denote by  $\zeta(A)$ the number of connected components of $A$ (possibly $+\infty$). Hence, $\zeta(A) = 1$ precisely when $A$ is connected, i.e., is a point, interval or $S^{1}$.

Let $G$ be a subset of $\Homeo(S^{1})$. We say that the action of $G$ on $S^{1}$ is \emph{proximal} if every pair of points on $S^{1}$ can be brought  close to each other by the action of $G$, i.e., \[
\inf_{g \in G} d(gx, gy) = 0 \,\,\textrm{for every $x, y \in S^{1}$}.
\]
We say that an interval $I \subseteq S^{1}$ is \emph{$G$-contractible} if it can be arbitrarily contracted by the action of $G$, i.e., $\inf_{g \in G} \diam (gI) = 0$.

We endow $\Homeo(S^{1})$ with the $C^{0}$-topology. Given subsets $A_{1}, \ldots, A_{n} \subseteq \Homeo(S^{1})$, we define their \emph{composition}: \[
A_{1} \cdots A_{n} := \{ g_{1} \cdots g_{n} : g_{i} \in A_{i} \,\,\textrm{for each}\,\, i = 1, \ldots, n\}.
\]
We abbreviate the $A\cdots A$, the $n$-self-convolution of $A$, by $A^{n}$.

In this paper, every probability measure on $\Homeo(S^{1})$ that we consider is a Borel probability measure. Given a probability measure $\mu$ on $\Homeo(S^{1})$, we define $\supp \mu$ to be the smallest $\mu$-conull closed set on $\Homeo(S^{1})$. The composition map on $\Homeo(S^{1})$ induces the convolution operation on probability measures. In general, for probability measures $\mu$ and $\nu$, we have \[
(\supp \mu) \cdot (\supp \nu) \subsetneq \supp (\mu \ast \nu).
\]

We consider the weak-$\ast$ topology on the space of probability measures on $\Homeo(S^{1})$. In this topology, a sequence $\mu_{i}$ of probability measures on $\Homeo(S^{1})$ converges to another measure $\mu$ if and only if $\lim_{i \rightarrow +\infty} \E_{\mu_{i}} (f) = \E_{\mu} (f)$ for each bounded continuous function $f: \Homeo(S^{1}) \rightarrow \mathbb{R}$.

\section{Weak Tits alternative for semigroups}\label{section:weakTits}

In this section, we prove Theorem \ref{thm:main0}. Let us divide the theorem into the following two parts.

\begin{thm}\label{thm:titsAltGen}
Let $G$ be a subsemigroup of $\Homeo(S^{1})$  that does not admit an invariant probability measure on $S^{1}$. Then $G$ contains a Schottky pair associated with finite unions of intervals. That means, there exists $f_{1}, f_{2} \in G$ and essentially disjoint open sets $U_{1}, U_{2}, V_{1}, V_{2}$ of $S^{1}$ that satisfy the following: \begin{enumerate}
\item each of $U_{1}, U_{2}, V_{1}, V_{2}$ has finitely many components, i.e., is a finite union of open intervals;
\item $f_{1}(S^{1} \setminus U_{1}) \subseteq V_{1}$ and $f_{1}^{-1} (S^{1} \setminus V_{1}) \subseteq U_{1}$, and 
\item $f_{2}(S^{1} \setminus U_{2}) \subseteq V_{2}$ and $f_{2}^{-1} (S^{1} \setminus V_{2}) \subseteq U_{2}$.
\end{enumerate}
\end{thm}

The converse of Theorem \ref{thm:titsAltGen} also holds. Indeed, if a subsemigroup $G$ contains a Schottky pair $(f_{1}, f_{2})$ associated with $(U_{1}, U_{2}, V_{1}, V_{2})$ and admits an invariant probability measure $\nu$, then \[
\big\{ f_{i}^{k} \big(S^{1} \setminus (\bar{U}_{i} \cup \bar{V}_{i})\big) \big\}_{k \in \Z}
\]
are disjoint sets with the same $\nu$-value for $i=1, 2$. Their $\nu$-values are hence null. This implies that \[
\nu(S^{1}) \le \nu \big(S^{1} \setminus (\bar{U}_{1} \cup \bar{V}_{1})\big) +\nu \big(S^{1} \setminus (\bar{U}_{2} \cup \bar{V}_{2})\big) =0.
\]

\begin{thm}\label{thm:proximalEquiv}
Let $G$ be a subsemigroup of $\Homeo(S^{1})$. Then the following are equivalent: 
\begin{enumerate}
\item $G$ acts on $S^{1}$ proximally and does not have a global fixed point on $S^{1}$.
\item $G$ acts on $S^{1}$ proximally and every $G$-orbit is infinite, i.e., $\# Gx = +\infty$ for each $x \in S^{1}$.
\item $G$ acts on $S^{1}$ proximally and does not preserve a probability measure on $S^{1}$.
\end{enumerate}
The above equivalent conditions imply that:
\begin{enumerate}\setcounter{enumi}{3}
\item $G$ has a Schottky pair associated with intervals, i.e., there exists $f_{1}, f_{2} \in G$ and essentially disjoint open intervals $I_{1}, I_{2}, J_{1}, J_{2}$ in $S^{1}$ such that \begin{equation}\label{eqn:SchottkyPair}
f_{i}(S^{1} \setminus I_{i}) \subseteq J_{i}, \,\,f_{i}^{-1} (S^{1} \setminus J_{i})  \subseteq I_{i} \quad (i=1, 2).
\end{equation}
\end{enumerate}
\end{thm}

Condition (4) above is actually equivalent to Condition (1)--(3). We will prove this in Section \ref{section:SchottkyRandom}; see Theorem \ref{thm:proximalEquivConverse}.

\begin{remark}\label{rem:weakTitsOpen}
Theorem \ref{thm:titsAltGen} implies that having no invariant probability measure is an open condition on the space of subsemigroups of $\Homeo(S^{1})$ with the Chabauty topology. Theorem \ref{thm:proximalEquiv} and Theorem \ref{thm:proximalEquivConverse} also tell us that proximality without global fixed points is an open condition. 
\end{remark}

We first prove that each point of $S^{1}$ has a $G$-contractible neighborhood unless $G$ preserves a probability measure.

\begin{thm}\label{thm:semigpInvProb}
Let $G$ be a subsemigroup of $\Homeo(S^{1})$ and let $x \in S^{1}$. Suppose that $x$ does not have a $G$-contractible neighborhood, and suppose that each $G$-orbit is infinite. Then $G$ has an invariant probability measure.
\end{thm}

\begin{proof}
From the assumption, we observe: \begin{claim}\label{claim:semigpInvFirst}
For each open interval $I$ containing  $x$, there exists $\delta= \delta(I)>0$ such that   $\diam(gI) > \delta$ for all $g \in G$.
\end{claim}
This is just a rephrasing of the fact that $I$ is not $G$-contractible.

\begin{claim}\label{claim:semigpInvSecond}
For each $y \in S^{1}$ and $N>0$, there exists  $\epsilon =\epsilon(N, y)>0$ such that the following holds. For each neighborhood $I$ of $y$ whose diameter is smaller than $\epsilon$, there exist $N$ elements $g_{1}, \ldots, g_{N}$ of $G$ such that $g_{1} I, \ldots, g_{N} I$ are disjoint.
\end{claim}

\begin{proof}
Since the $G$-orbit of $y$ is infinite, we can take $N$ elements $g_{1}, \ldots, g_{N} \in G$ such that $g_{1} y, \ldots, g_{N}y$ are distinct. Then their $\eta$-neighborhoods are also disjoint for some small enough $\eta$. Since each $g_{i}$ is continuous, $\cap g_{i}^{-1} N_{\eta}(g_{i}y)$ is an open neighborhood of $y$. Let $\epsilon$ be such that $N_{\epsilon}(y)$ is contained in this open neighborhood. Now, let $I$ be a neighborhood of  $y$ whose diameter smaller than $\epsilon$. Then $I$ is contained in $N_{\epsilon}(y)$, and it is clear that $g_{1} I, \ldots, g_{N} I$ are disjoint. 
\end{proof}

Meanwhile, consider an interval neighborhood $I$ of  $x$. Claim \ref{claim:semigpInvFirst} provides us with some $\delta = \delta( I) > 0$ such that $\diam(g I) \ge \delta$ for any $g \in G$. Hence, the maximum number of mutually disjoint $G$-translates of $I$ is at most $1/\delta$. For an interval $I \subseteq S^{1}$ and a subset $A \subseteq S^{1}$, we define \begin{equation}\label{eqn:disjtTrans}
N(A; I) := \sup \left\{ \# S : \begin{array}{c}\textrm{$S\subseteq G$, $gI \subseteq A$ for each $g \in S$}, \\
\textrm{ $g I \cap g' I = \emptyset$ for distinct elements $g, g' \in S$}\end{array}\right\}.
\end{equation}
The following is immediate:
\begin{claim}\label{claim:semigpInvThird}
For each open interval $I$ containing $x$ and for each subset $A$ of $S^{1}$, $N(A; I)$ defined above is finite.
\end{claim}

Let us now take an open neighborhood basis $\{I_{n}\}_{n>0}$ of $x$, i.e.,  \[
I_{1} \supseteq I_{2} \supseteq \ldots \searrow \{x\}.
\]
We then consider the set of binary rationals \[
\mathcal{D} := \{ 2^{-n} k : n>0, k=1, \ldots, 2^{n}\}
\]
and the collection of dyadic half-open intervals: \[
\mathcal{E} := \{ (a, b] : a, b \in \mathcal{D} \}.
\]
(For convenience, we include $\emptyset, S^{1} \in \mathcal{E}$.) For each dyadic half-open interval $A \in \mathcal{E}$, the sequence \[
\left( \frac{N(A; I_{n})}{N(S^{1}; I_{n})} \right)_{n > 0} = \left( \frac{N(A; I_{1})}{N(S^{1}; I_{1})} , \, \frac{N(A; I_{2})}{N(S^{1}; I_{2})},  \,\frac{N(A; I_{3})}{N(S^{1}; I_{3})}, \ldots \right)
\]
is well-defined thanks to Claim \ref{claim:semigpInvThird} and is bounded between 0 and 1. Hence, up to replacing $(I_{n})_{n>0}$ with some subsequence, $(N(A; I_{n})/N(S^{1};I_{n}))$ converges. Since $\mathcal{E}$ is countable, we can subsequently take convergent subsequences for each $A \in \mathcal{E}$. As a result, for a suitable interval neighborhood basis  $(I_{n})_{n>0}$ of $x$, we can guarantee that 
\[
\lim_{n \rightarrow +\infty} \frac{N(A; I_{n})}{N(S^{1}; I_{n})} =: \nu(A)\,\,\textrm{exists for each $A \in \mathcal{E}$.}
\]

Here, $\nu$ is not yet to be called a Borel measure. Still, we observe\begin{enumerate}
\item monotonicity: If $A, B \in \mathcal{E}$ satisfies $A \subseteq B$, then $\nu(A) \le \nu(B)$ also holds.
\end{enumerate}

To discuss finite additivity, let $A_{1}, A_{2} \in \mathcal{E}$ be disjoint elements of $\mathcal{E}$ and suppose that $A := A_{1} \sqcup A_{2}$ belongs to $\mathcal{E}$. Up to relabelling $A_{1}$ and $A_{2}$, we can write as $A_{1} = (a, c]$ and $A_{2} = (c, b]$ for some $a, b \in \mathcal{D}$ and $c \in A \cap \mathcal{D}$. (Here, we allow the possibility that $a = b$ and $A = S^{1}$.) 

We now claim for each $n$ that  \[
N(A_{1}; I_{n}) + N(A_{2}; I_{n}) \le N(A; I_{n}) \le N(A_{1}; I_{n}) + N(A_{2}; I_{n}) + 2.
\]
First, if $g_{1} I_{n}, \ldots, g_{k} I_{n}$ are disjoint translates of $I_{n}$ in $A_{1}$ and $h_{1} I_{n}, \ldots, h_{l} I_{n}$ are disjoint translates of $I_{n}$ in $A_{2}$, then $g_{1} I_{n}, \ldots, g_{k} I_{n}, h_{1} I_{n}, \ldots, h_{l} I_{n}$ are disjoint translates of $I_{n}$ in $A$. This explains the first inequality. Next, let $u_{1} I_{n}, \ldots, u_{m} I_{n}$ be disjoint translates of $I_{n}$ in $A$. Since these translates are disjoint, at most one can contain $c$. Next, at most one can contain $a$ or $b$, which is possible only when $a=b$ and $A = S^{1}$. Except these at most two translates, all the other translates are contained in either $A_{1}$ or $A_{2}$. This explains the second inequality.

Meanwhile, we observed in Claim \ref{claim:semigpInvSecond} that $N(S^{1}; I_{n})$ grows indefinitely as $n$ tends to infinity. Hence, we conclude \[
0 \le \big|\nu(A_{1}) + \nu(A_{2}) - \nu(A) \big| = \lim_{n \rightarrow +\infty}\left| \frac{N(A_{1}; I_{n}) + N(A_{2}; I_{n}) - N(A; I_{n}) }{N(S^{1}; I_{n})} \right| \le \lim_{n \rightarrow +\infty}\left| \frac{2}{N(S^{1}; I_{n})} \right| =0.
\]

Inducting on the number of summands, we conclude
\begin{enumerate}\setcounter{enumi}{1}
\item finite additivity: if some finitely many elements $A, A_{1}, \ldots, A_{n}$ of $\mathcal{E}$ satisfy \[
A = A_{1} \sqcup \ldots \sqcup A_{n},
\]
then  $\nu(A) = \nu(A_{1}) + \ldots + \nu(A_{n})$ holds.
\end{enumerate}

Note also that $\mathcal{E}$ is a semiring of sets. From the monotonicity and the finite additivity, we deduce
\begin{enumerate}\setcounter{enumi}{2}
\item finite subadditivity: if some finitely many elements $A, A_{1}, \ldots, A_{n}$ of $\mathcal{E}$ satisfy \[
A \subseteq  A_{1} \cup \ldots \cup A_{n},
\]
then $\nu(A) \le \nu(A_{1}) + \ldots + \nu(A_{n})$ holds.
\end{enumerate}

We now discuss some sort of absolute continuity. (Note: this is not the absolute continuity as measures. See Remark \ref{rem:absCtsNo}.)

\begin{claim}\label{claim:absCts}
For each $\eta>0$ there exists $\epsilon>0$ such that for every $\epsilon$-short interval $I \subseteq S^{1}$  and for every $k \in \Z_{>0}$, we have $N(I; I_{k})/ N(S^{1}; I_{k}) < \eta$. In particular, if $I$ is a dyadic half-open interval whose length is at most $\epsilon$, then $\nu(I)$ is smaller than $\eta$.
\end{claim}

\begin{proof}
Let $\eta>0$ and pick an integer $N$ greater than $1/\eta$. For each $y \in S^{1}$ there exists $\epsilon(y, N)$ as in Claim \ref{claim:semigpInvSecond}. We then pick a $\epsilon(y, N)$-short dyadic interval $J_{y}$ whose interior contains $y$. Then $\{\inte J_{y}\}_{y \in S^{1}}$ becomes an open cover of $S^{1}$. By the Lebesgue covering lemma, there exists $\epsilon$ such that any  $\epsilon$-short interval is contained in some $J_{y}$. It now suffices to check the claim for $I= J_{y}$ for each $y$.

To this end, let us fix $y \in S^{1}$. By Claim \ref{claim:semigpInvSecond}, there exists $g_{1}, \ldots, g_{N} \in G$ such that $g_{1} J_{y}, g_{2} J_{y}, \ldots, g_{N} J_{y}$ are disjoint. Now let $k$ be an integer, and pick $h_{1}, \ldots, h_{N(J_{y}; I_{k})}$ such that $h_{1} I_{k}, \ldots, h_{N(J_{y}; I_{k})} I_{k}$ are disjoint subsets of $J_{y}$. Then  \[
\{ g_{l} h_{i} I_{k}: l=1, \ldots, N, i = 1, \ldots, N(J_{y}; I_{k}) \}
\]
become $N \cdot N(J_{y}; I_{k})$ disjoint $G$-translates of $I_{k}$ in $S^{1}$. This implies that  $N(S^{1}; I_{k}) \ge N \cdot N(J_{y}; I_{k})$ and $N(J_{y}; I_{k}) / N(S^{1}; I_{k}) \le 1/N$. By sending $k$ to infinity, we conclude that $\nu(J_{y}) \le 1/N < \eta$.
\end{proof}

Let us now prove the following countable additivity of $\nu$: \begin{claim}\label{claim:countableAdd}
If some countably many elements $A, A_{1}, A_{2}, \ldots$ of $\mathcal{E}$ satisfy \[
A = \sqcup_{i=1}^{\infty} A_{i},
\]
then $\nu(A) = \sum_{i=1}^{\infty} \nu(A_{i})$ holds.
\end{claim}

\begin{proof}
One direction is straightforward: by finite additivity and monotonicity we have\[
\sum_{i=1}^{N} \nu(A_{i}) = \nu\big( A_{1} \sqcup \ldots \sqcup A_{N}\big) \le \nu(A),
\]
and we obtain $\sum_{i=1}^{\infty} \nu(A_{i}) \le \nu(A)$ by sending $N$ to infinity. For the reverse direction, we will prove   \begin{equation}\label{eqn:epsilonCountableAdd}
\sum_{i=1}^{\infty} \nu(A_{i}) + \epsilon \ge \nu(A)
\end{equation}
for arbitrary $\epsilon>0$. By Claim \ref{claim:absCts}, there exists $\epsilon_{i}>0$ for each  $i \in \Z_{>0}$ such that \[
\textrm{$\nu(I) \le \epsilon/3^{i}$ for every $\epsilon_{i}$-short dyadic half-open interval $I$.}
\]
Now let $A = (a, b]$. If $A$ is shorter than $\epsilon_{1}$, then its $\nu$-value is smaller than  $\epsilon$ by the above. Hence Inequality \ref{eqn:epsilonCountableAdd} is immediate. If $A$ is not shorter than $\epsilon_{1}$, then we take $c \in (a, b] \cap \mathcal{D}$ such that $(a, c]$ has length at most $\epsilon_{1}$.

Each $A_{i}$ is of the form $(a_{i}, b_{i}]$. We now take $c_{i} \in \mathcal{D}$ such that $[b_{i}, c_{i}]$ has length at most $\epsilon_{i}$. Then we have\[
[c, b] \subseteq (a, b]=A \subseteq \sqcup_{i=1}^{\infty} A_{i} \subseteq \cup_{i=1}^{\infty} (a_{i}, c_{i}).
\]
Since $[c, b]$ is compact, there exists an integer $N$ such that  \[
[c, b] \subseteq \cup_{i=1}^{N} (a_{i}, c_{i}) \subseteq \cup_{i=1}^{N} (a_{i}, c_{i}].
\]
From this, we observe \[
(a,b ] \subseteq (a, c] \cup \big(\sqcup_{i=1}^{N} (a_{i}, b_{i}] \big) \cup \big(\cup_{i=1}^{N} (b_{i}, c_{i}]\big).
\]
Finite subadditivity then tells us that \[
\nu\big((a, b] \big)\le \nu\big( (a, c]\big) + \sum_{i=1}^{N} \nu(A_{i}) + \sum_{i=1}^{N} \nu\big( (b_{i}, c_{i}]\big).
\]
Recall that $(a, c]$ has length at most $\epsilon_{1}$ and $(b_{i}, c_{i}]$ has length  at most $\epsilon_{i}$; their $\nu$-values are at most $\epsilon/3$ and at most $\epsilon/3^{i}$, respectively. \[
\nu(A)\le \epsilon/3 + \sum_{i=1}^{N} \nu(A_{i}) + \sum_{i=1}^{N} \epsilon/3^{i} \le \sum_{i=1}^{\infty} \nu(A_{i}) + \epsilon
\]
This ends the proof.
\end{proof}

The Carath{\'e}odory extension theorem tells us that $\nu$ is uniquely extended to a (countably additive) measure on the $\sigma$-algebra generated by $\mathcal{E}$. Here, the uniqueness relies on the fact that $\nu$ is a finite measure. So $\nu$ is now a Borel measure.

It remains to prove the $G$-invariance of $\nu$. We claim that $\nu(J) = \nu(gJ)$ holds for each $g \in G$ and for each dyadic half-open interval $J$. If this is true, then $\nu_{g} := g^{\ast} \nu$ coincides with $\nu$ on $\mathcal{E}$ so the uniqueness part of the Carath{\'e}odory extension theorem will imply that $\nu = \nu_{g}$ on the Borel $\sigma$-algebra. 

To prove the claim, let $J = (a, b]$ for some distinct elements $a, b \in \mathcal{D}$. Fix $\eta>0$, and let $\epsilon>0$ be the one for $\eta$ as in Claim \ref{claim:absCts}.

A subtlety is that $ga$ or $gb$ may well be outside $\mathcal{D}$. To cope with this, we take  small open intervals $U_{a}\ni a$, $U_{b}\ni b$ such that $gU_{a}$ and $gU_{b}$ are disjoint and are shorter than $\eta$. Here, by shrinking $U_{a}$ and $U_{b}$ a little bit, we can force $gU_{a}$ and $gU_{b}$ to be elements of $\mathcal{E}$, i.e., binary half-open intervals.

We now define $J_{1} := J \setminus (U_{a} \cup U_{b})$, $J_{2} := (S^{1} \setminus J) \setminus (U_{a} \cup U_{b})$. Then $J_{1}$ and $J_{2}$ are slightly smaller intervals than $J$ and $S^{1} \setminus J$, respectively, and $gJ_{1}, gJ_{2}, gU_{a}, gU_{b}$ are elements of $\mathcal{E}$ that partition $S^{1}$.

Let $k \in \Z_{>0}$. In the remaining, we abbreviate $N(S^{1} ; I_{k})$ into $N_{k}$ for convenience. By the definition of $N(\cdot \,; I_{k})$, we can take $g_{i}, g_{j}' \in G$ such that   \[
\sqcup_{i=1}^{N(J; I_{k})} g_{i} I_{k} \subseteq J,  \quad \sqcup_{j=1}^{N(S^{1} \setminus J; I_{k})} g'_{j}I_{k} \subseteq (S^{1} \setminus J).
\] We can then see that  \[
\sqcup_{i=1}^{N(J; I_{k})} gg_{i} I_{k} \subseteq gJ.
\]
Among $\{gg_{i} I_{k}\}_{i}$, there are at most $\eta N_{k}$ intervals that are included in $gU_{a}$ because of Claim \ref{claim:absCts}. Some of $\{g g_{i} I_{k}\}_{i}$ may intersect with $gU_{a}$ while not being included in $gU_{a}$, but such intervals are at most 2 (the ones containing an endpoint of $gU_{a}$). This uses the fact that $\{g g_{i} I_{k}\}_{i}$ are disjoint. Hence, at most $\eta N_{k} +2$ intervals among $\{g g_{i}I_{k}\}_{i}$ intersect with $gU_{a}$.

Similarly, there are at most $\eta N_{k}+2$ intervals among  $\{gg_{i} I_{k}\}_{i}$ that intersect with $gU_{b}$. Hence, there are at least $N(J; I_{k}) - 2\eta N_{k}-4$ intervals among $\{g g_{i} I_{k}\}_{i}$ that are included in $gJ_{1}$. A similar argument applies to  $\{gg'_{j}I_{k}\}_{j}$, and we conclude: \[
\begin{aligned}
N\big(gJ_{1} ; I_{k} \big) &\ge N(J; I_{k}) - 2\eta N_{k}-4, \\
N\big( gJ_{2}; I_{k} \big) &\ge N(S^{1} \setminus J; I_{k}) - 2 \eta N_{k}-4.
\end{aligned}
\]
Towards contradiction, let us assume that $N(gJ_{1}; I_{k})$ is greater than  $N(J; I_{k}) + 2\eta N_{k} +6$, i.e., $N(gJ_{1}; I_{k}) \ge N(J; I_{k}) +2\eta N_{k} +7$. Because $gJ_{1}$ and $gJ_{2}$ are disjoint subsets of $S^{1}$, we have \begin{equation}\label{eqn:contraNk}
N(S^{1} ; I_{k}) \ge N\big(gJ_{1}; I_{k} \big) + N\big( gJ_{2}; I_{k} \big) \ge N(J; I_{k}) + N(S^{1} \setminus J; I_{k})+3.
\end{equation}
Let us now pin down the translates of $I_{k}$'s realizing this number, i.e., we take $u_{1}, \ldots, u_{N_{k}} \in G$ such that $u_{1}I_{k}, \ldots, u_{N_{k}}I_{k}$ are disjoint. Here, $J$ and $S^{1} \setminus J$ partition $S^{1}$ and their boundary $\partial J$ consists of 2 points. In other words, except for at most 2 that contain some points of $\partial J$, the other  $u_{i} I_{k}$'s must be contained in $J$ or $S^{1} \setminus J$. Clearly the $u_{i} I_{k}$'s are mutually disjoint. This leads to \[
N(J; I_{k}) + N(S^{1} \setminus J; I_{k}) \ge N(S^{1}; I_{k}) - 2,
\]
which contradicts with Inequality \ref{eqn:contraNk}.

Therefore, we conclude that $N(gJ_{1}; I_{k})$ and $N(J; I_{k})$ differ by at most  $2\eta N_{k} +6$. For the same reason, $N(gJ_{2}; I_{k})$ and  $N(S^{1} \setminus J; I_{k})$ differ by at most  $2\eta N_{k} +6$. We now increase $k$ and observe\[\begin{aligned}
\nu(J) - 2\eta &= \lim_{k \rightarrow +\infty} \frac{N(J; I_{k})-2\eta N_{k} - 6}{N_{k}} \\
&\le   \lim_{k \rightarrow +\infty} \frac{N(gJ_{1}; I_{k})}{N_{k}} = \nu (gJ_{1}) \\
&\le \lim_{k \rightarrow +\infty} \frac{N(J; I_{k})+2\eta N_{k} + 6}{N_{k}} = \nu(J) + 2\eta.
\end{aligned}
\]
For the similar reason, $\nu(S^{1}\setminus J)$ and $\nu(gJ_{2})$ differ by at most  $2\eta$. We now conclude  \[
\nu(J) - 2 \eta \le \nu(gJ_{1}) \le \nu(gJ) = 1 - \nu(S^{1} \setminus gJ) \le 1 - \nu(gJ_{2}) \le 1 - \nu(S^{1} \setminus J) + 2\eta = \nu(J) + 2\eta
\]
Since this inequality holds for arbitrary $\eta>0$, we conclude $\nu(J) = \nu(gJ)$.\end{proof}

\begin{remark}\label{rem:absCtsNo}
Note that the measure $\nu$ constructed in this proof is not always absolutely continuous with respect to the Lebesgue measure. For example, if one considers the group $G$ of rotations on the circle and then blow up dyadic rationals, we obtain a new uniquely ergodic action of $G$ on $S^{1}$ with the Cantor set as the minimal set. The new action preserves the pullback of the Lebesgue measure through the semiconjugacy map. The pullback measure $\nu$ is indeed the $\nu$ constructed in the proof above. It is true that short \emph{intervals} have small $\nu$-value, but it is not true that measurable sets with small Lebesgue measure have small $\nu$-value. Indeed, $\nu$ can be supported on the Lebesgue measure-zero Cantor set.

When $G$ is assumed to be a subgroup (rather than a subsemigroup), Margulis and Ghys reduces the general case to minimal actions by considering the semiconjugacy to the minimal Cantor set. It is tricky to apply this method  if $G$ is a subsemigroup, as $G$ does not canonically act on the minimal set by homeomorphisms. 
\end{remark}

\begin{cor}\label{cor:semigpNoInv}
Let $G$ be a subsemigroup of $\Homeo(S^{1})$ without invariant probability measure. Then there exists $\epsilon>0$ such that every $\epsilon$-short intervals are $G$-contractible.
\end{cor}

\begin{proof}
Since there is no invariant probability measure, each $G$-orbit must be infinite. Now Theorem \ref{thm:semigpInvProb} tells us that each point on $S^{1}$ has a $G$-contractible neighborhood. The Lebesgue covering lemma then gives the desired $\epsilon$.
\end{proof}

\begin{cor}\label{cor:contractibleOpen}

Let $G$ be a subsemigroup of $\Homeo(S^{1})$ without invariant probability measure. Then every $G$-contractible closed interval of $S^{1}$ is contained in another $G$-contractible open interval.
\end{cor}

\begin{proof}
We first take $\epsilon>0$ to be the constant as in Corollary \ref{cor:semigpNoInv}. If a closed interval $[a, b]$ is $G$-contractible, then there exists $g \in G$ such that $\diam(g [a, b]) < \epsilon/3$. Because $g$ is continuous, we have $\diam( g[\alpha, \beta])< \epsilon/2$ for $\alpha, \beta \in S^{1} \setminus [a, b]$ close enough to $a$ and $b$, respectively. Then Corollary \ref{cor:semigpNoInv} tells us that  $g[\alpha, \beta]$ is $G$-contractible, i.e., there exists a sequence $\{g_{n}\}_{n>0}$ in $G$ such that $\lim_{n} \diam(g_{n} \cdot g[\alpha, \beta]) =0$. Now $\{g_{n} \cdot g\}_{n>0}$ is also a sequence in $G$ and contracts $[\alpha, \beta]$ arbitrarily small. Hence  $[\alpha, \beta]$ is $G$-contractible.
\end{proof}

\subsection{Proof of Theorem \ref{thm:titsAltGen}}\label{subsection:titsAltGen}

We are now ready to prove Theorem \ref{thm:titsAltGen}, following Margulis' argument with some paraphrasing.

The following lemma (and its proof) is a classical fact in group theory due to B. H. Neumann \cite[Lemma 4.1]{neumann1954groups}. It originally states that no group can be written as a finite union of left cosets of  infinite-index subgroups. We need a semigroup version of this fact, whose proof is very similar to the original one. We record it for the readers' convenience.

\begin{lem}\label{lem:disjtStab}
Let $G$ be a subsemigroup of $\Homeo(S^{1})$ such that every $G$-orbit is infinite. Then for every pair of finite sets $A$ and $B$, there exists $g \in G$ such that $gA$ and $B$ are disjoint.
\end{lem}

\begin{proof}
For each $x, y \in S^{1}$ we define $Fix(x, y):= \{g \in G : g(x) = y\}$. Let $G^{-1} := \{id\} \cup \{g^{-1} : g \in G\}$, which is again a semigroup. Our goal is to prove that:\begin{claim}\label{claim:disjtStab}
For any  $n \in \mathbb{Z}_{>0}$, for any $n$ finite subsets $S_{1}, \ldots, S_{n} \subseteq G^{-1}$, and for any (not necessarily distinct) $2n$ points $x_{1}, \ldots, x_{n}, y_{1}, \ldots, y_{n}$, \[
G \subseteq \cup_{i=1}^{n} S_{i} \cdot  Fix(x_{i}, y_{i})
\]
cannot hold.
\end{claim}
Let us observe this for $n=1$. Given a finite set $S = \{g_{1}, \ldots, g_{k}\} \subseteq \Homeo(S^{1})$ and points $x, y \in S^{1}$, the elements of $S \cdot Fix(x, y)$ send $x$ to one of finitely many candidates $g_{1} y, \ldots, g_{k} y$. Meanwhile, because $x$ has infinite $G$-orbit, there exists $g \in G$ that sends $x$ to something other than $g_{1} y \ldots, g_{k} y$. This shows that $G \not\subseteq S \cdot Fix(x, y)$.

In order to induct on $n$, let us assume\[
G \subseteq \cup_{i=1}^{n} S_{i} \cdot  Fix(x_{i}, y_{i})
\]
for some $S_{1}, \ldots, S_{n} \subseteq G^{-1}$ and $x_{1}, y_{1}, \ldots, x_{n}, y_{n} \in S^{1}$. Since $y_{1}$ has infinite $G$-orbit, there exists $g \in G$ such that $gy_{1} \notin S_{1} y_{1}$. Then $g \cdot Fix(x_{1}, y_{1})$ cannot intersect with $S_{1} \cdot Fix(x_{1}, y_{1})$. It follows that  \[
G \cap \big(g \cdot Fix(x_{1}, y_{1})\big) \subseteq \cup_{i=2}^{n} S_{i} \cdot Fix(x_{i}, y_{i}).
\]
This implies \[
g^{-1} \cdot G \cap Fix(x_{1}, y_{1}) \subseteq \cup_{i=2}^{n} g^{-1} S_{i} \cdot Fix(x_{i}, y_{i}).
\]
Hence, for each $s \in S_{1}$, we have  \[
G \cap s Fix(x_{1}, y_{1}) \subseteq sg^{-1} G \cap sFix(x_{1}, y_{1}) \subseteq \cup_{i=2}^{n} s g^{-1} S_{i} \cdot Fix(x_{i}, y_{i}).
\]
Here, the first inclusion is due to the fact that  $G \subseteq sg^{-1} \cdot (g s^{-1} G) \subseteq sg^{-1} G$. Using this, we deduce \[\begin{aligned}
G &\subseteq \Big(\cup_{s \in S_{1}} (G \cap s Fix(x_{1}, y_{1}) ) \Big) \cup \Big( \cup_{i=2}^{n} S_{i} Fix(x_{i}, y_{i}) \Big) \\
&\subseteq \cup_{i=2}^{n} \big(\{s g^{-1} s' : s \in S_{1}, s' \in S_{i} \} \cup S_{i} \big) \cdot Fix(x_{i}, y_{i}).
\end{aligned}
\]
Clearly $\{s g^{-1} s' : s \in S_{1}, s' \in S_{i} \} \cup S_{i}$ is a finite subset of $G^{-1}$ for each $i$. Hence, a counterexample for $n$ leads to a counterexample for $n-1$. Since we have ruled out such a counterexample for $n=1$, the claim follows from the induction.

Coming back to the lemma, we need to prove that \[
G \subseteq \cup_{a \in A, b \in B} Fix(a, b)
\]does not hold. This follows immediately from Claim  \ref{claim:disjtStab}.
\end{proof}

In the rest of this section, we will call finite unions of intervals \emph{elementary sets}. We begin by recording an elementary lemma.

\begin{lem}\label{lem:eltSet}
Let $\epsilon>0$, let $E \subseteq S^{1}$ be an elementary set and let $\{F_{k}\}_{k>0}$ be a sequence of intervals. If $\Leb(F_{k} \setminus E)$ is greater than $\epsilon$ for each $k$, then there exists a subsequence $\{F_{k(l)}\}_{l>0}$ of $\{F_{k}\}_{k}$ and an elementary set $K$ outside $E$  such that $\Leb(K) =\epsilon/2$ and $K \subseteq F_{k(l)}$ for all $l$.
\end{lem}

\begin{proof}[Proof of Theorem \ref{thm:titsAltGen}]
Let $G$ be a subsemigroup of $\Homeo(S^{1})$ without any invariant probability measure. We can pick $\epsilon>0$ for $G$ as in Corollary \ref{cor:semigpNoInv}. For convenience, assume that $N:= 1/\epsilon$ is a positive integer. We then take $N$ points equidistributed on $S^{1}$; they become endpoints of $N$ half-open intervals $I_{1}, \ldots, I_{N}$ of length $1/N = \epsilon$ that partition $S^{1}$.

We construct, for each $n=1, 2, \ldots$, a finite set $S_{n} \subseteq S^{1}$, an interval $V_{n}$,  an elementary set $W_{n}$  and a  sequence  $\{g_{k; n}\}_{k>0}$ in $G$. We will also define $W_{0}$ as the base case. We claim that they satisfy:
\begin{enumerate}
\item $V_{n} \in \{I_{1}, \ldots, I_{N}\}$.
\item $W_{0}, W_{1}, \ldots$ are disjoint and $\Leb(W_{n}) = \frac{1}{2N}(1-\frac{1}{2N})^{n}$. (Hence, $\Leb(W_{0} \sqcup \ldots \sqcup W_{n}) = 1- (1-\frac{1}{2N})^{n+1}$.)
\item $g_{k; n}(W_{n}) \subseteq V_{n}$ for each $k$.
\item $g_{k; n} (W_{0} \sqcup \ldots \sqcup W_{n-1})$ converges to $S_{n}$ as $k$ tends to infinity. That means, for every $\eta>0$,  $g_{k; n} (W_{0} \sqcup \ldots \sqcup W_{n-1}) \subseteq N_{\eta}(S_{n})$ holds for all sufficiently large $k$.
\end{enumerate}

We discuss the base case $n=1$. Pick an arbitrary interval of length $1/2N$ and denote it by $W_{0}$. Because $W_{0}$ is $\epsilon$-short, it is $G$-contractible. Hence there exists a sequence $\{g_{k}\}_{k>0}$ in $G$ such that $\diam(g_{k} W_{0}) \searrow 0$ as $k \rightarrow +\infty$. By taking a suitable subsequence,  we may assume that $\{g_{k} W_{0}\}_{k>0}$ converges to some point  $x_{1} =: S_{1}$. Meanwhile, for each $k$, \[
g_{k}^{-1} (I_{1}) \setminus W_{0},\,g_{k}^{-1} (I_{2}) \setminus W_{0},\,  \ldots, \,g_{k}^{-1} (I_{N})\setminus W_{0}
\]
are disjoint elementary sets partitioning $S^{1} \setminus W_{0}$. Let the $l(k)$-th one be the one with greatest Lebesgue measure value; that value should be at least  $\frac{1}{N} (1 - \Leb(W_{0}))$. Note that $l(k)$ is picked from $\{1, \ldots, N\}$, a finite set. Hence, up to a subsequence, we may assume that $l(1), l(2), \ldots$ are identical; we define  $V_{1}$ to be $I_{l(1)} = I_{l(2)} = \ldots$. By appealing to Lemma \ref{lem:eltSet}, we may assume the following up to a subsequence: there exists an elementary subset $W_{1} \subseteq S^{1} \setminus W_{0}$ with Lebesgue measure $\frac{1}{2N}(1-\Leb(W_{0}))$ so that  $g_{k}^{-1} (V_{1}) \setminus W_{0}$ contains $W_{1}$ for each $k$. Now we adopt the resulting $\{g_{k}\}_{k>0}$ as $\{g_{k;1}\}_{k>0}$. The desired properties for $n=1$ are easily checked.

Next, given the objects for $n-1$, we define the ones for $n$. First, since $V_{n-1}$ is $\epsilon$-short, it is $G$-contractible; there exists $\{h_{k}\}_{k>0} \subseteq G$ such that   $\diam(h_{k} V_{n-1}) \searrow 0$ as $k \rightarrow +\infty$. Up to a subsequence, we may assume that $\{h_{k} V_{n-1}\}_{k>0}$ converges to a point $x_{n}$.  Up to a further subsequence, we may assume that $\{h_{k}S_{n-1}\}_{k>0}$, a sequence of finite sets of cardinality $\#S_{n-1}$, converges to a finite set $S'$. We then define $S_{n} := S' \cup \{x_{n}\}$. For each  $k$, we can take large enough  $i(k)$ such that \[
 h_{k} g_{i(k); n-1} (W_{0} \sqcup \ldots \sqcup W_{n-2}) \subseteq N_{1/k} (h_{k} S_{n-1}),\quad h_{k} g_{i(k); n-1} W_{n-1} \subseteq h_{k} V_{n-1}.
 \] Then $\{g_{k} := h_{k} g_{i(k);n-1}\}_{k>0}$ becomes a sequence in $G$ such that  $g_{k} (W_{0} \sqcup \ldots \sqcup W_{n-1}) \rightarrow S_{n}$ as $k\rightarrow +\infty$.

For each $k$, we consider \[
g_{k}^{-1} (I_{1}) \setminus (W_{0} \sqcup \ldots \sqcup W_{n-1}), \,
g_{k}^{-1} (I_{2}) \setminus (W_{0} \sqcup \ldots \sqcup W_{n-1}),\,\ldots, \,g_{k}^{-1} (I_{N})\setminus (W_{0} \sqcup \ldots \sqcup W_{n-1}).
\]
These are disjoint elementary sets partitioning $S^{1} \setminus (W_{0} \sqcup \ldots \sqcup W_{n-1})$. Let the $l(k)$-th one be the one with the greatest Lebesgue measure value. That measure value should be at least  $\frac{1}{N} (1 - \Leb(W_{0} \sqcup \ldots \sqcup W_{n-1}))$. Note that the $l(k)$'s take values in a finite set $\{1, \ldots, N\}$. Up to a subsequence we may assume that $l(1), l(2), \ldots$ are identical; then we define $V_{n}$ to be $I_{l(1)} = I_{l(2)} = \cdots$. Thanks to  Lemma \ref{lem:eltSet}, the following holds up to a subsequence: there exists an elementary set $W_{n} \subseteq S^{1} \setminus (W_{0} \sqcup \ldots \sqcup W_{n-1})$ with Lebesgue measure $\frac{1}{2N}\big(1-\Leb(W_{0} \sqcup \ldots \sqcup W_{n-1})\big)$ such that $g_{k}^{-1} (V_{n}) \setminus (W_{0} \sqcup \ldots \sqcup W_{n-1})$ contains  $W_{n}$ for each $k$. We now take the resulting $\{g_{k}\}_{k>0}$ as  $\{g_{k;n}\}_{k>0}$. The desired properties for $n$ follow.

We have now constructed $S_{n}, V_{n}, W_{n}, \{g_{k; n}\}_{k>0}$ satisfying Properties (1)--(4). 
In other words, we constructed for each $n$ an elementary set $K_{n} (:=W_{0} \sqcup \ldots \sqcup W_{n-1})$, a finite set $S_{n}$ and a sequence $\{g_{k;n}\}_{k>0}$ in $G$ such that $g_{k;n} K_{n} \rightarrow S_{n}$ as $k \rightarrow +\infty$, and such that $\Leb(K_{n}) = 1 - (1-\frac{1}{2N})^{n}$. (We can now forget about $V_{n}$.)

We have not yet put restrictions on the size of $S_{n}$, but by modifying the choice of $\{g_{k; n}\}_{k>0}$ (while keeping $K_{n}:=W_{0} \sqcup \ldots \sqcup W_{n-1}$) we can make $\#S_{n} \le N$. To see this, suppose that $\#S_{n} > N$. Then one of $I_{1}, \ldots, I_{N}$ contains more than 2 points of $S_{n}$. Without loss of generality, suppose that $I_{1}$ does so. Because $I_{1}$ has length $\epsilon$ and is $G$-contractible, there exist $\{h_{k}\}_{k>0}$ in $G$ such that  $\diam(h_{k} I_{1}) \searrow 0$ as $k \rightarrow +\infty$. Up to a subsequence, we may assume that $\{h_{k} I_{1}\}_{k>0}$ converges to a point $x$. Moreover,  $\{h_{k} (S_{n} \setminus I_{1})\}_{k>0}$ is a sequence of sets of cardinality at most $\#S_{n} - 2$. Up to a subsequence, we may assume that $\{h_{k} (S_{n} \setminus I_{1})\}_{k>0}$ converges to a finite set $F$ of cardinality at most $\#S_{n}-2$. We now set $S_{n}' := F \cup \{x\}$; then $h_{k} S_{n} \rightarrow S_{n}'$ as $k\rightarrow +\infty$. For each $k$, we choose large enough $i(k)$ such that \[
h_{k} g_{i(k); n} (K_{n}) \subseteq N_{1/k} (h_{k} S_{n}).
\] Then $\{h_{k} g_{i(k);n}\}_{k>0}$ is a sequence in $G$ such that $h_{k} g_{i(k); n} (K_{n}) \rightarrow S_{n}'$ as $k \rightarrow +\infty$. Note that $\#S_{n}' \le \#S_{n} -1$. By applying this procedure inductively, we can make $\#S_{n}$ smaller than or equal to $N$.
Hence, we may assume that $\#S_{n} \le N$ for each $n$.

We thus have a sequence $(K_{n})_{n>0}$ of elementary sets, a sequence $(S_{n})_{n>0}$ of sets with cardinality at most $N$, and a sequence $(\{g_{k; n}\}_{k>0})_{n>0}$ of sequences in $G$ that satisfy the following for each $n$: \begin{enumerate}
\item $g_{k; n} K_{n} \rightarrow S_{n}$ as $k$ tends to infinity, and
\item $\Leb(K_{n}) \ge 1 - (1-1/2N)^{n}$.
\end{enumerate}
Now, by replacing  $(K_{n}, S_{n}, \{g_{k; n}\}_{k>0} )_{n>0}$ with a suitable subsequence, we may assume that $S_{n}$ converges to a finite set $S$. (This is where the uniform bound on $\#S_{n}$ is needed.) We then pick a small enough  $\eta>0$ and consider the $\eta$-neighborhood $N_{\eta}(S)$ of $S$. Then $S_{n} \subseteq N_{\eta/2}(S)$ for suitably large $n$, and $g_{k(n); n} (K_{n}) \subseteq N_{\eta/2} (S_{n}) \subseteq N_{\eta}(S)$ for suitably large $k(n)$. Hence, $L_{n} := g_{k(n); n}^{-1} (S^{1} \setminus N_{\eta}(S))$ becomes a set of Lebesgue measure at most $1- \Leb(K_{n})$. In fact, this set is a union of $\#S$ closed intervals, each with length at most $1- \Leb(K_{n})$. Let $C_{n}$ be the set of centers of these intervals. Up to a subsequence, we may assume that $C_{n}$ converges to a finite subset $C$ as $n$ tends to infinity. Because $L_{n}$ lies in the $(1-\Leb(K_{n}))$-neighborhood of $C_{n}$, $L_{n} \subseteq N_{\eta}(C)$ for suitably large $n$. 

In summary, we have constructed two finite sets $C$ and $S$; for any small enough $\eta>0$, there exists a large enough $n$ and $k(n)$ such that $g_{k(n); n}$ sends $S^{1} \setminus N_{\eta}(C)$ into $N_{\eta}(S)$, and vice versa for $g_{k;n}^{-1}$.

Now using Lemma  \ref{lem:disjtStab}, we can pick $f \in G$ such that $C$ and $f(S)$ to be disjoint. We can take $g \in G$ such that  $g(C \cup f(S))$ and $C$ are disjoint, and $h \in G$ such that $h \cdot f(S)$ and $C \cup f(S) \cup g^{-1} C$ are disjoint. Then $C$, $f(S)$, $g^{-1} C$, $h f(S)$ are pairwise disjoint finite sets, so we can take a small $\eta>0$ such that $N_{\eta}(C), fN_{\eta}(S), g^{-1} N_{\eta}(C)$, $hfN_{\eta} (S)$ are mutually essentially disjoint. For this  $\eta$, we can take $g_{k(n); n}$ as described above. Then $(fg_{k(n); n}, h fg_{k(n); n} g)$ becomes a Schottky pair in $G$  associated with $N_{\eta}(C), fN_{\eta}(S), g^{-1} N_{\eta}(C)$, $hfN_{\eta} (S)$, each of which is a finite union of intervals. 
\end{proof}

\subsection{Proof of Theorem \ref{thm:proximalEquiv}}\label{subsection:proximalEquiv}

We now talk about proximality. Let us introduce a notion for the proof of Theorem \ref{thm:proximalEquiv}.

\begin{dfn}\label{dfn:repeller}
Let $G$ be a subsemigroup of $\Homeo(S^{1})$ and let $x \in S^{1}$. If every closed interval in $S^{1} \setminus \{x\}$ is $G$-contractible, we say that  $x$ is a $G$-repeller.
\end{dfn}

\begin{prop}\label{prop:repellerSchottky}
Let $G$ be a subsemigroup of $\Homeo(S^{1})$ such that every $G$-orbit is infinite. Suppose also that there exists a $G$-repeller $x \in S^{1}$. Then $G$ contains a Schottky pair associated with essentially disjoint intervals.
\end{prop}

\begin{proof}
Let $\{I_{n}\}_{n>0}$ be an interval neighborhood basis of $x$, i.e., $I_{n} \searrow \{x\}$ as $n \rightarrow +\infty$. For each $n$, $S^{1} \setminus I_{n}$ is a closed interval disjoint from $x$ and hence $G$-contractible. From this we can take a sequence $\{g_{n}\}_{n>0}$ in $G$ such that $\diam(S^{1} \setminus g_{n} I_{n}) \searrow 0$ as $n \rightarrow +\infty$. Let $J_{n} := S^{1} \setminus g_{n}I_{n}$. Because $S^{1}$ is compact, $\{J_{n}\}_{n>0}$ has an accumulation point $y$. Up to a subsequence, we may assume that $\diam(J_{n} \cup y) \searrow 0$ as $n \rightarrow +\infty$. At this moment, if $y=x$ happens to be the case, then we pick $g \in G$ such that $gx \neq x$ (using the infinitude of $Gx$). Then $g \cdot g_{n}$ sends $S^{1}\setminus I_{n}$ to $gJ_{n}$, and the sequence $\{gJ_{n}\}_{n>0}$ converges to $gx \neq x$ as $n \rightarrow +\infty$. Considering this, we may assume $y\neq x$.

By Lemma \ref{lem:disjtStab}, there exists $g \in G$ such that $g\cdot \{x, y\}$ does not intersect with $\{x\}$, i.e., $x, y, g^{-1} x$ are distinct points. Another round of Lemma \ref{lem:disjtStab} guarantees an element $h \in G$ such that $hy$ and $\{x, y, g^{-1} x\}$ are disjoint. 

Here, note that $hg_{n} g$ sends $S^{1} \setminus g^{-1} I_{n}$ to $hJ_{n}$, where the sequence $\{g^{-1} I_{n}\}_{n>0}$ converges to $g^{-1} x$ and $\{hJ_{n}\}_{n>0}$ converges to $hy$. Since $x, y, g^{-1} x, hy$ are all distinct, we can take large enough $n$ such that $\overline{I_{n}}, \overline{J_{n}}, \overline{g^{-1} I_{n}}, \overline{hJ_{n}}$ are each sufficiently close to $x, y, g^{-1} x, hy$ and are mutually disjoint. For that $n$, $(g_{n}, hg_{n} g)$ becomes a Schottky pair associated with essentially disjoint intervals $(I_{n}, J_{n}, g^{-1} I_{n}, hJ_{n})$. Clearly $g_{n}$ and $hg_{n}g$ both belongs to $G$. This finishes the proof.
\end{proof}

For a subsemigroup $G$ of $\Homeo(S^{1})$, if the supremum of lengths of $G$-contractible intervals is 1, then there must be $G$-repeller. Indeed, suppose that   $\lim_{n} \diam( S^{1} \setminus I_{n}) =0$ for some sequence $\{I_{n}\}_{n>0}$ of $G$-contractible intervals. By taking a subsequence if needed, we may assume that $\{S^{1} \setminus I_{n}\}_{n>0}$ converges to some point $x \in S^{1}$. Then any closed interval not containing $x$ should be contained in some $I_{n}$ for some large $n$. Such an interval should be $G$-contractible, and $x$ is a $G$-repeller.

The contrapositive of the above is as follows: if a subsemigroup $G$ of $\Homeo(S^{1})$ does not have a $G$-repeller, then the supremum of lengths of $G$-contractible intervals is smaller than 1.  We now introduce a notion:

\begin{dfn}\label{dfn:nextContract}
Let $G$ be a subsemigroup of $\Homeo(S^{1})$. We say that a closed interval $J = [a, b]$ (that is not the entire circle) is \emph{$G$-firm} if it satisfies the following: \begin{enumerate}
\item $[a, c]$ is $G$-contractible for each $c \in (a, b)$, and
\item $J=[a, b]$ is not $G$-contractible.
\end{enumerate}
\end{dfn}

Let $G$ be a subsemigroup of $\Homeo(S^{1})$ without invariant probability measure, let $x \in S^{1}$ be an arbitrary point, and let \[
I := \{y \in S^{1} : \textrm{$[x, y]$ is $G$-contractible}\}
\]
The following is immediate: for each element $y \in I \setminus \{x\}$, every element of $[x, y]$ also belongs to $I$. Hence, $I$ is either an interval with $x$ as the left endpoint, or the entire circle. Corollary \ref{cor:contractibleOpen} tells us that $I$ is a half-open interval, with the left end closed and the right end open, unless $I = S^{1}$. Hence, if $I \neq S^{1}$, then the closure $\bar{I}$ of $I$ is a $G$-firm interval. 

Note also that $x$ is a $G$-repeller if $I = S^{1}$. In other words, if $G$ is moreover a subsemigroup without $G$-repeller, then each $x\in S^{1}$ becomes a left endpoint of some $G$-firm interval. We denote this $G$-firm interval by $Firm(x)$. Note that the supremum of length of $G$-firm intervals is smaller than 1.

\begin{lem}\label{lem:nextContractNoProx}
Let $G$ be a subsemigroup of $\Homeo(S^{1})$ without $G$-repeller. If $x, y \in S^{1}$ satisfy that $Firm(x) \subseteq [x, y]$ and $Firm(y) \subseteq [y, x]$, then $x$ and $y$ cannot be brought close to each other by the $G$-action, i.e., $\inf_{g \in G} d(gx, gy) > 0$ holds.
\end{lem}

\begin{proof}
Towards contradiction, suppose that there exists a sequence $\{g_{n}\}_{n>0}$ in $G$ such that 
$\lim_{n} d(g_{n}x, g_{n}y) = 0$. Then either $\liminf_{n} \diam(g_{n}[x, y]) =0$ or $\liminf_{n} \diam(g_{n}[y, x]) =0$ holds. The former implies that $Firm(x) \subseteq [x, y]$ is $G$-contractible, which is a contradiction. The latter implies that $Firm(y) \subseteq [y, x]$ is $G$-contractible, which is again a contradiction.\end{proof}

Let $G$ be a subsemigroup of $\Homeo(S^{1})$ whose every $G$-orbit is infinite, and suppose that there exists a $G$-repeller $p$. We then claim that $G$ is proximal, i.e., $\inf_{g \in G} d(gx, gy) = 0$ for each $x, y \in S^{1}$. First, when $x, y \neq p$ this is clear from the definition of $G$-repellers. If $x = p$, then we can pick some $h \in G$ such that $x \notin \{hx, hy\}$ using Lemma \ref{lem:disjtStab}. It is easy to check that $h^{-1} x$ is also a $G$-repeller, which is distinct from $x$ and $y$. It is then clear that $\inf_{g \in G} d(gx, gy) = 0$.

Interestingly, the converse also holds. We first observe:

\begin{lem}\label{lem:noFirmIntersect}
Let $G$ be a subsemigroup of $\Homeo(S^{1})$ acting proximally on $S^{1}$ without  any invariant probability measure.  Then there exists a $G$-repeller on $S^{1}$.
\end{lem}

\begin{proof}
If there exist $x, y \in S^{1}$ such that  \[
Firm(x) \subseteq [x, y] \,\,\textrm{and}\,\, Firm(y) \subseteq [y, x]
\]
holds, then the action of $G$ on $S^{1}$ is not proximal by Lemma \ref{lem:nextContractNoProx}, a contradiction. Hence, \[
Firm(x) \not\subseteq [x, y] \,\,\textrm{or}\,\, Firm(y) \not\subseteq [y, x]
\]
holds for each $x, y \in S^{1}$. We then claim the following; see Figure \ref{fig:contractDichom}.

\begin{claim}\label{claim:contractibleNbdDich}
For each $x \in S^{1}$ there exists an open neighborhood $U$ of $x$ such that, for every open interval $I \subseteq U$ and for every $q \in S^{1} \setminus I$, one of the following holds:\begin{enumerate}
\item $[p, q]$ is $G$-contractible for all $p \in I$, or;
\item $[q, p]$ is $G$-contractible for all $p \in I$.
\end{enumerate}
\end{claim}

\begin{proof}
Given $x \in S^{1}$, we take $z$ such that $Firm(x) = [x, z]$. Then by the assumption, $Firm(z)$ is not contained in $[z, x]$; there exists $a \in (x, z)$ such that $[z, a]$ is $G$-contractible. (Think of $a$ as a point ``just to the right of $x$"). By Corollary \ref{cor:contractibleOpen}, there exists $z' \in (a, z)$ such that $[z', a]$ is also $G$-contractible. (Think of $z'$ as a point ``just to the left of  $z$".) We now take $z'' \in (z', z)$. Then $z'' \in S^{1} \setminus [z, a] \subseteq (x, z)$. Hence, $[x, z'']$ is $G$-contractible. This means $[x', z'']$ is $G$-contractible for some  $x' \in (z'', x) \subseteq (z, x)$ (Again, $x'$ is ``just left to $x$").

To show that $U := (x', a)$ satisfies the desired property, let $I = (c, d)$ be an interval contained in $U$. Every $q \in S^{1} \setminus I$ either belongs to $[d, z'']$ or $[z'', c]$. In the former case, $[p, q] \subseteq [x', z'']$ is $G$-contractible for every $p \in I$. In the latter case, $[q, p] \subseteq [z', a]$ is $G$-contractible for each $p \in I$. 
\end{proof}

\begin{figure}
\begin{tikzpicture}[scale=0.75]
\draw (0, 0) circle (3);
\draw[very thick, rotate=180] (3.3, 0) arc (0:180:3.3);
\draw[very thick] (3.6, 0) arc (0:210:3.6);
\draw[dashed] (3, 0) -- (4, 0);
\draw (4.15, 0) node {$z$};
\draw[dashed] (-3, 0) -- (-4, 0);
\draw (-4.15, 0) node {$x$};
\draw[dashed, rotate=203] (3, 0) -- (4.55, 0);
\draw[rotate=203] (4.65, 0) node {$a$};
\draw[dashed, rotate=-15] (3, 0) -- (4, 0);
\draw[rotate=-15] (4.18, 0) node {$z'$};
\draw[dashed, rotate=-7.5] (3, 0) -- (4.45, 0);
\draw[rotate=-7.5] (4.7, 0) node {$z''$};
\draw[dashed, rotate=162] (3, 0) -- (4.45, 0);
\draw[rotate=162] (4.65, 0) node {$x'$};
\draw (0, -3.7) node {$Firm(x)$};
\draw (0, 3.95) node {$Firm(z)$};
\draw[dashed, very thick, rotate=167] (4.65, 0) arc (0:33:4.65);
\draw[rotate=-6] (-4.92, 0) node {$U$};
\fill[rotate=3] (-4.65, 0) circle (0.1);
\fill[rotate=15] (-4.65, 0) circle (0.1);
\draw[rotate=3] (-4.95, 0) node {$c$};
\draw[rotate=15] (-4.95, 0) node {$d$};
\draw[ very thick, rotate=183] (4.65, 0) arc (0:12:4.65);

\end{tikzpicture}
\caption{Schematics for Claim \ref{claim:contractibleNbdDich}.}
\label{fig:contractDichom}
\end{figure}
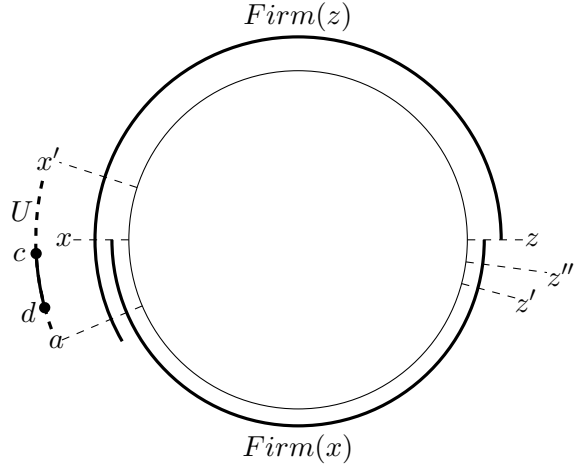

Pick an arbitrary $x \in S^{1}$ and consider an increasing sequence of intervals $I_{n} = [x, z_{n}]$ that fills up the interior of $Firm(x)$. That means, writing  $Firm(x)$ as $[x, z]$, we require that $[x, z_{n}] \subsetneq [x, z]$ and $\lim_{n} z_{n} =z$. Note that $z \notin [x, z_{n}]$. Each $I_{n}$ is $G$-contractible and there exists $g_{n} \in G$ such that $g_{n} I_{n}$ is $1/n$-short. By taking a subsequence, we may assume that $\{g_{n}I_{n}\}_{n>0}$ converges to a point $y$.

Let $U=U(y)$ be the open neighborhood for $y$ as in Claim \ref{claim:contractibleNbdDich}. For every sufficiently large $n$, $g_{n}I_{n}$ is contained in  $U(y)$ and $g_{n}z$ is outside $g_{n}I_{n}$. By Claim \ref{claim:contractibleNbdDich}, one of the following should hold:\begin{enumerate}
\item $g_{n}[p, z]$ is $G$-contractible for every $p \in [x, z_{n}] \subseteq g_{n}^{-1}U$; or
\item $g_{n}[z, p]$ is $G$-contractible for every $p \in  [x, z_{n}] \subseteq g_{n}^{-1}U$.
\end{enumerate}
In the former case, $[x, z] = Firm(x)$ becomes $G$-contractible, which is a contradiction. Hence, only the latter can be true: $[z, z_{n}]$ is $G$-contractible. Because this holds for arbitrary $n$, and because the $[z, z_{n}]$'s  exhaust $S^{1} \setminus \{z\}$, we conclude that $z$ is a $G$-repeller.
\end{proof}

We are now ready to prove Theorem \ref{thm:proximalEquiv}.

\begin{proof}[Proof of Theorem \ref{thm:proximalEquiv}]
Let $G$ be a subsemigroup of $\Homeo(S^{1})$. First note that a $G$-fixed point is by definition a singleton $G$-orbit. Next, if there exists a finite $G$-orbit, i.e., $\# Gx < +\infty$ for some $x \in S^{1}$, then the uniform measure on $Gx$ becomes $G$-invariant. Indeed, for any $g \in G$, we have $g \cdot Gx \subseteq Gx$ by the semigroup property of $G$. Since $Gx$ is a finite set, this forces that $gGx = Gx$. In other words, each $g \in G$ permutes points in $Gx$ and preserves the uniform measure on $Gx$. We conclude the implication (3) $\Rightarrow$ (2) $\Rightarrow$ (1) in the statement.

We now show that (1) implies (3) by the method of contradiction. To this end, let $G$ be a subsemigroup of $\Homeo(S^{1})$ that acts on $S^{1}$ proximally without a fixed point, but with an invariant measure $\nu$. Suppose first that $G$ has a finite orbit, i.e., there exists $y \in S^{1}$ such that $Gy = \{y_{1}, \ldots, y_{N}\}$ is a finite set. Since $G$ does not have a global fixed point, $N$ must be at least 2. Then for any homeomorphism $g \in G$, $gy_{1}$ and $gy_{2}$ are distinct points in $Gy$. In other words, \[
\inf_{g \in G} d(gy_{1}, gy_{2}) \ge \min \big\{ d(y_{i}, y_{j}) : i, j \in \{1, \ldots, N\}, i \neq j \big\} > 0.
\]
and the action of $G$ is not proximal, which is a contradiction. Hence, $G$ does not have a finite orbit.

It is now immediate that $\nu$ is atom-less. Indeed, if $\nu(x) = \epsilon > 0$ for some $x \in S^{1}$, then $\nu(Gx) = \epsilon \cdot \#Gx = +\infty$, which cannot happen for a probability measure on $S^{1}$. We can then find an interval $I = [a, b]$ such that both $I$ and $S^{1} \setminus I$ has positive $\nu$-value. Let $\epsilon>0$ be such that $\nu([a, b]), \nu([b, a]) \ge \epsilon$. (In fact, one can realize $\epsilon = 1/2$ by the Intermediate Value Theorem.)

Since $G$ acts proximally on $S^{1}$, there exists a sequence $(g_{n})_{n>0}$ in $G$ such that either $\diam(g_{n}[a, b]) \searrow 0$ or $\diam(g_{n}[b, a]) \searrow 0$ as $n \rightarrow +\infty$. By switching $a$ and $b$ and taking a subsequence of $(g_{n})_{n>0}$ if necessary, we may assume that $\{g_{n} [a, b]\}_{n>0}$ converges to a point $y \in S^{1}$. Then for any small neighborhood $U$ of $y$, $\nu(U) \ge \nu(g_{n}[a, b]) = \nu([a, b]) \ge \epsilon$ holds for some large enough $n$. It follows that $\nu(\{y\}) = \inf_{\textrm{open $U \ni y$}} \nu(U)\ge \epsilon$, contradicting to the non-atomness of $\nu$. Hence, such a $G$-invariant measure $\nu$ cannot exist.

Let us now show that (3) implies (4). For this, let $G$ be a subsemigroup of $\Homeo(S^{1})$ that acts proximally without an invariant probability measure. Then Lemma \ref{lem:noFirmIntersect} guarantees a $G$-repeller. Note also that every $G$-orbit is infinite. By Proposition \ref{prop:repellerSchottky}, $G$ contains a Schottky pair associated with essentially disjoint intervals.
\end{proof}

\section{Schottky sets and random walks}\label{section:SchottkyRandom}

In this section, we use the properties of Schottky pairs constructed in Section \ref{section:weakTits} to study random walks on $\Homeo(S^{1})$. We first introduce Schottky sets, which are finite collections of Schottky regions with controlled multiplicity and resolution. The main technical input of the section is Proposition \ref{prop:pivotingPrep}, proved later in Section \ref{section:pivoting}; once this input is available, the results below prove exponential synchronization and the probabilistic Tits alternative.

Let us first define Schottky sets. Most of the time, the Borel measure $\mu$ on $\Homeo(S^{1})$ generating the random walk is not purely atomic. In this case, the semigroup $\llangle \supp \mu \rrangle$ might charge no atom to Schottky pairs. To accommodate this, we will consider a continuous family of Schottky pairs associated with common intervals/open sets.

\begin{dfn}\label{dfn:hyperbolicElt}
Let $f$ be a circle homeomorphism and let $U_{1}, U_{2}$ be disjoint subsets of $S^{1}$. If \[
f\big(S^{1} \setminus U_{1}\big) \subseteq U_{2}, \,\,f^{-1}\big(S^{1} \setminus U_{2}\big) \subseteq U_{1}
\]
holds, we say that $f$ is a \emph{$(U_{1}, U_{2})$-hyperbolic map}. We denote the collection of $(U_{1}, U_{2})$-hyperbolic maps by $\mathfrak{S}(U_{1}, U_{2})$.
\end{dfn}

If $U_{1}, U_{2}$ are open sets (closed sets, resp.), then $\mathfrak{S}(U_{1}, U_{2})$ is also open (closed, resp.) with respect to the $C^{0}$-topology.

\begin{dfn}\label{dfn:SchottkyEss}
For essentially disjoint subsets $I_{1}, \ldots, I_{N}, J_{1}, \ldots, J_{N}$ of $S^{1}$, we call \[
S := \mathfrak{S}(I_{1}, J_{1}) \sqcup \ldots \sqcup \mathfrak{S}(I_{N}, J_{N}) \subseteq \Homeo(S^{1})
\]
the \emph{Schottky set} associated with $I_{1}, \ldots, I_{N}, J_{1}, \ldots, J_{N}$. We call $N$ the \emph{resolution} of $S$. For each $s \in S$ there exists unique $i$ such that $s \in \mathfrak{S}(I_{i}, J_{i})$; for such an $i$, we write $I(s) := I_{i}$ and $J(s) := J_{i}$.

We define the \emph{multiplicity} $\zeta(S)$ of $S$ by \[\begin{aligned}
\zeta(S) &:= \sup \big\{ \zeta(I_{1}), \ldots,  \zeta(I_{N}), \zeta(J_{1}),\ldots,  \zeta(J_{N}), \big\} \\
&=  \sup \big\{ \# (\textrm{connected components of $U$}) : U = I_{1}, \ldots, I_{N}, J_{1}, \ldots, J_{N}\big\}.
\end{aligned}
\]

If there is a subset $\mathcal{I}$ that is essentially disjoint from  $I_{1} \cup \ldots \cup I_{N}$ and essentially contains $J_{1} \cup \ldots \cup J_{N} \subseteq \inte \mathcal{I}$, then we call it a \emph{median} for $S$.
\end{dfn}

When $\zeta(S)$ is finite, we say that $S$ is a Schottky set associated with finite unions of intervals. If $\zeta(S) = 1$, we say that $S$ is a Schottky set associated with intervals. In practice, we will almost always use the Schottky sets with finite multiplicity only.

Let $S$ be a Schottky set associated with essentially disjoint sets $I_{1}, \ldots, I_{N}, J_{1}, \ldots, J_{N}$. Then for a suitably small $\epsilon>0$,  $\mathcal{I} = N_{\epsilon}( J_{1} \cup \cdots \cup J_{N})$ serves as a median. In a special case that $I_{1}, \ldots, I_{N}$ and $J_{1}, \ldots J_{N}$ are separated in opposite semicircles of $S^{1}$, one can take an \emph{interval} median. The existence of an interval median for a Schottky set for a given probability $\mu$ turns out to be an essential ingredient for the exponential synchronization. In fact, if there exists a Schottky set in a semigroup $G$ with an interval median, then $G$ has proximal action and there exists another Schottky set \emph{associated with intervals} in $G$. 

\begin{dfn}\label{dfn:SchottkyUniform}
Let $\epsilon>0$ and let $S$ be a Schottky set associated with essentially disjoint sets   $I_{1}, \ldots, I_{N}, J_{1}, \ldots, J_{N}$. If a (Borel) measure $\mu$ on $\Homeo(S^{1})$ satisfies  \[
\mu\big(\mathfrak{S}(I_{i}, J_{i})\big) >\epsilon/N \,\, \textrm{for each}\,\, i=1, \ldots, N, 
\]
then we say that $\mu$ is an \emph{$(S, \epsilon)$-admissible measure}; if $\mu$ satisfies \[
\mu\big(\mathfrak{S}(I_{i}, J_{i})\big) =1/N \,\, \textrm{for each}\,\, i=1, \ldots, N, 
\]
then we say that $\mu$ is \emph{Schottky-uniform} on $S$.
\end{dfn}

Note that there can be several Schottky-uniform measures on a single Schottky set (because $\mathfrak{S}(I, J)$ is not a singleton for most $I$ and $J$). Let us now rephrase the weak Tits alternative for proximal actions discussed in Section \ref{section:weakTits}.

\begin{cor}\label{cor:weakTitsProximal}
Let $\mu$ be a probability measure on $\Homeo(S^{1})$ such that the semigroup $\llangle \supp \mu \rrangle$ acts proximally on $S^{1}$ without a global fixed point. Then there exists $N$ such that $(\supp \mu)^{\ast N}$ contains a Schottky pair associated with intervals.
\end{cor}

\begin{proof}
By Theorem \ref{thm:proximalEquiv}, there exists $k, l \in \Z_{>0}$, $f \in (\supp \mu)^{\ast k}$ and $g \in (\supp \mu)^{\ast l}$ such that $(f, g)$ forms a Schottky pair associated with essentially disjoint intervals $I_{1}, I_{2}, J_{1}, J_{2}$. Then $(f^{l}, g^{k})$ is also a Schottky pair associated with  $I_{1}, I_{2}, J_{1}, J_{2}$, and $f^{l}, g^{k}$ belongs to $(\supp \mu)^{\ast kl}$. 
\end{proof}

It is not hard to ``amplify'' a Schottky pair into a larger Schottky set.

\begin{lem}\label{lem:amplify}
Let $\mu$ be a probability measure on $\Homeo(S^{1})$ such that $\supp \mu$ contains a Schottky pair associated with essentially disjoint subsets $I_{1}, I_{2}, J_{1}, J_{2}$ of $S^{1}$, and let $\zeta = \sup \big\{\zeta(I_{1}), \zeta(I_{2}), \zeta(J_{1}), \zeta(J_{2})\big\}$. Then for each $N \in \Z_{>0}$ there exists $m \in \Z_{>0}$, $\epsilon>0$ and a Schottky set $S$ with resolution $N$ and with multiplicity $\le 2\zeta$ such that $\mu^{\ast m}$ is $(S, \epsilon)$-admissible.

If $\zeta=1$, then $S$ can be taken to have an interval median and have multiplicity $1$.
\end{lem}

\begin{proof}
Let $(f_{1}, f_{2})$ be a Schottky pair in $\supp \mu$ associated with essentially disjoint sets $I_{1}, I_{2}, J_{1}, J_{2}$ on $S^{1}$. Let $\zeta = \sup\big\{\zeta(I_{1}), \zeta(I_{2}), \zeta(J_{1}), \zeta(J_{2})\big\}$.

First, let us explain why we can guarantee a Schottky pair with an interval median in the case $\zeta = 1$, i.e., in the case that $I_{i}, J_{i}$ are intervals. If there exists an interval $\mathcal{I}$ such that $I_{1} \cup I_{2} \subseteq \mathcal{I}$ and $J_{1} \cup J_{2} \subseteq S^{1} \setminus \mathcal{I}$, we keep it. If there exists no such interval, it means that $I_{1}, J_{1}, I_{2}, J_{2}$ are arranged clockwise or counterclockwise along $S^{1}$. In either case, we can take an interval $\mathcal{I}$ that essentially contains $J_{2}$ but is essentially disjoint from  $I_{1}, J_{1}$ and $I_{2}$. This $\mathcal{I}$ is not a median for $(f_{1}, f_{2})$ but is a median for $\{f_{2}^{2}, f_{2} f_{1}\}$. Indeed, by setting \[
f_{1}' := f_{2} f_{1}, \,\,f_{2}' := f_{2}^{2}, \,\,I_{1}' := I_{1}, \,\,J_{1}' :=f_{2}J_{1}, \,\,I_{2}' := I_{2}, \,\,J_{2}' := f_{2} J_{2},
\]
we observe $f_{i}' ( S^{1} \setminus I_{i}') \subseteq J_{i}'$ and $f_{i}'^{-1}(S^{1} \setminus J_{i}') \subseteq I_{i}'$ for $i=1,2$, and \[
\overline{J_{1}'} \cap \overline{J_{2}'} = f_{2} (\overline{J_{1}} \cap \overline{J_{2}}) = \emptyset, \,\, (\overline{I_{1}'} \cup \overline{I_{2}'}) \cap (\overline{J_{1}'} \cup \overline{J_{2}'}) \subseteq  (\overline{I_{1}} \cup \overline{I_{2}}) \cap f_{2} (S^{1} \setminus \inte I_{2}) \subseteq (\overline{I_{1}\cup I_{2}})  \cap \overline{J_{2}} = \emptyset.
\] 
Furthermore, $\mathcal{I}$ is essentially disjoint from $I_{1}$ and $I_{2}$ but its interior contains $\bar{J_{2}}$, which in turn contains $\bar{J_{1}'}$ and $\bar{J_{2}'}$. Hence, $\mathcal{I}$ is a median for  $\{f_{1}', f_{2}'\}$.

Hence, in the case $\zeta = 1$, up to replacing $(f_{1}, f_{2})$ with $(f_{2}f_{1}, f_{2}^{2})$ and up to replacing $\mu$ with $\mu^{\ast 2}$, we may assume that the Schottky pair has interval median $\mathcal{I}$. If $\zeta \neq 1$, we take $\mathcal{I} = N_{\epsilon}(J_{1} \cup J_{2})$ for a small enough $\epsilon$ as a median; note that $\mathcal{I}$ has at most $2\zeta$ components.

Let us now define $2^{N}$ homeomorphisms parametrized by  $\{1, 2\}^{N}$. Given  $\sigma \in \{1, 2\}^{N}$, we construct
\[
f_{\sigma} := f_{\sigma(1)} f_{\sigma(2)} \cdots f_{\sigma(N)}, \,\, I_{\sigma} := f_{\sigma}^{-1} \left(\overline{S^{1} \setminus \mathcal{I}} \right), \,\, J_{\sigma} := f_{\sigma} \mathcal{I}.
\]
Note that $f_{\sigma}^{2}$ sends $S^{1} \setminus I_{\sigma}$ into $J_{\sigma}$ and $f_{\sigma}^{-2}$ sends $S^{1} \setminus J_{\sigma}$ into $I_{\sigma}$. Furthermore, we observe\begin{equation}\label{eqn:JMedian}
\begin{aligned}
\overline{J_{\sigma}} &=  f_{\sigma(1)}\cdots f_{\sigma(N)} \overline{\mathcal{I}} \subseteq  f_{\sigma(1)} \cdots f_{\sigma(N-1)}J_{\sigma(N)} \\
&\subseteq f_{\sigma(1)} \cdots f_{\sigma(N-1)} (\inte \mathcal{I}) = f_{\sigma(1)} \cdots f_{\sigma(N-2)}(J_{\sigma(N-1)}) \\
 & \subseteq \ldots \subseteq J_{\sigma(1)} \subseteq \inte \mathcal{I}.
 \end{aligned}
\end{equation}
For a similar reason we have $\overline{I_{\sigma}} \subseteq \inte (S^{1} \setminus \mathcal{I})$. In short, $\overline{I_{\sigma}}$ and  $\overline{J_{\sigma'}}$ do not overlap with each other for any $\sigma, \sigma' \in \{1, 2\}^{N}$. Now let us take distinct elements $\sigma$ and $\sigma'$ of $\{1, 2\}^{N}$. Then there exists $i$ such that $\sigma(i) \neq \sigma'(i)$, and we take a minimal one. Then \[
\overline{J_{\sigma}} \subseteq f_{\sigma(1)} \cdots f_{\sigma(i-1)} \overline{J_{\sigma(i)}}, \,\,\overline{J_{\sigma'}} \subseteq f_{\sigma'(1)} \cdots f_{\sigma'(i-1)} \overline{J_{\sigma'(i)}}
\]
should not intersect. For a similar reason, $\overline{I_{\sigma}}$ and $\overline{I_{\sigma'}}$ are disjoint. To sum up, the $2 \cdot 2^{N}$ sets $\{ I_{\sigma}, J_{\sigma} : \sigma \in \{1, 2\}^{N}\}$ are all pairwise essentially disjoint. It is clear that  $(\supp \mu^{\ast 2N})$ intersects with each of  $\mathcal{S}(I_{\sigma}, J_{\sigma})$. Furthermore, Display \ref{eqn:JMedian} (and its counterpart for the $\overline{I_{\sigma}}$'s) tells us that $\mathcal{I}$ works as a median. Finally, note that $\mathcal{I}$ and $S^{1} \setminus \mathcal{I}$ have the same multiplicity, which bounds the number of components of $I_{\sigma}$ and $J_{\sigma}$ for each $\sigma$. (When $\zeta = 1$, the $I_{\sigma}, J_{\sigma}$'s are intervals.)

We now take very small $\eta>0$ and let $I_{\sigma}' := N_{\eta}(I_{\sigma}), J_{\sigma}' := N_{\eta}(J_{\sigma})$. If $\eta$ is small enough, $\{I_{\sigma}', J_{\sigma}' : \sigma \in \{1, 2\}^{N}\}$ are mutually essentially disjoint, $\cup_{\sigma} I_{\sigma}'$ is essentially disjoint from $\mathcal{I}$ and $\cup_{\sigma} J_{\sigma}'$ is essentially contained in $\mathcal{I}$. Also, the maximum number of components of $\{I_{\sigma}', J_{\sigma}'\}_{\sigma}$ is no bigger than the maximum for $\{I_{\sigma}, J_{\sigma}\}_{\sigma}$.

Now for each $\sigma \in \{1, 2\}^{N}$, $\mathcal{S}(I_{\sigma}', J_{\sigma}')$ is an \emph{open} subset (of $\Homeo(S^{1})$) that intersects with  $\supp \mu^{\ast 2N}$, as it contains $f_{\sigma}^{2}$. Hence, it attains a strictly positive $\mu^{\ast 2N}$-value. This implies that $S' = \cup_{\sigma \in \{1, 2\}^{N}} \mathcal{S}(I_{\sigma}', J_{\sigma'})$ is a Schottky set with a median $\mathcal{I}$, with resolution $2^{N} \ge N$ and with multiplicity at most $2\zeta$. Moreover, $\mu^{\ast 2N}$ is $(S', \epsilon)$-admissible for some $\epsilon>0$. One can now consider $S = \cup_{\sigma \in A} \mathcal{S}(I_{\sigma}', J_{\sigma}')$ for some subset $A$ of $\{1, 2\}^{N}$ with cardinality $N$ to conclude a similar property.
\end{proof}

We can now prove the converse of Theorem \ref{thm:proximalEquiv}.

\begin{thm}\label{thm:proximalEquivConverse}
Let $G$ be a subsemigroup of $\Homeo(S^{1})$ that contains a Schottky pair associated with intervals. Then $G$ acts proximally and does not have a global fixed point on $S^{1}$.
\end{thm}

\begin{proof}
Let $(f_{1}, f_{2})$ be the Schottky pair in $G$ associated with essentially disjoint intervals $I_{1}, I_{2}, J_{1}, J_{2}$. Let $x \in S^{1}$ be an arbitrary point. Since the intervals are disjoint, we have either $x \notin I_{1}\cup J_{1}$ or $x \notin I_{2} \cup J_{2}$. In the first case, we have $f_{1}(x) \in J_{1}$ and hence $f_{1}(x) \neq x$; in the second case, we have $f_{2}(x) \in J_{2}$ and hence $f_{2}(x) \neq x$. In both cases, $x$ is not a common fixed point of $f_{1}$ and $f_{2}$.

Next, let $x, y \in S^{1}$ be arbitrary two points and let $N > 10$. To show the proximality of $G$, consider a probability measure $\mu$ with $\supp \mu = \{f_{1}, f_{2}\}$ ($\mu(f_{1}) = \mu(f_{2}) = 1/2$ will do). Then by Lemma \ref{lem:amplify}, there exists $m, \epsilon >0$ and  a Schottky set $S$ with resolution $N$ and with multiplicity $1$ such that $\mu^{\ast m}$ is $(S, \epsilon)$-admissible. Let $I_{1}, \ldots, I_{N}, J_{1}, \ldots, J_{N}$ be the (essentially disjoint) intervals that $S$ is associated with. Then each $\mathfrak{S}(I_{i}, J_{i})$ intersects with $\supp \mu^{\ast m} \subseteq G$.

Since the $I_{i}$'s are disjoint, we have \[
\# \big( \mathcal{A} := \{i : \textrm{$I_{i}$ contains $x$ or $y$}\}\big) \le 2.
\]
Furthermore, since the $J_{i}$'s are disjoint, \[
\frac{1}{\sqrt{N}} \cdot \# \left(\mathcal{B} := \big\{ i : \Leb(J_{i}) > 1/\sqrt{N} \big\} \right) \le \Leb(S^{1}) = 1.
\]
Hence, $\mathcal{A} \cup \mathcal{B}$ has at most $\sqrt{N} + 2 < N$ elements, and we can pick an index $i$ outside it. For that $i$, we conclude that $d(gx, gy) < \Leb(J_{i}) \le 1/\sqrt{N}$ for each $g \in \mathfrak{S}(I_{i}, J_{i})$. Since $G$ intersects with $\mathfrak{S}(I_{i}, J_{i})$, we conclude that $\inf_{g \in G} d(gx, gy) < 1/\sqrt{N}$. By sending $N$ to infinity, we conclude that the action of $G$ is proximal.
\end{proof}

\begin{cor}\label{cor:measureConvTarget}
Let $\mu$ be a probability measure on $\Homeo(S^{1})$ such that the semigroup $\llangle\supp \mu \rrangle$ does not admit any invariant probability measure on $S^{1}$. Then there exist an open neighborhood $\mathcal{U}$ of $\mu$ in the space of probability measures on $\Homeo(S^{1})$, $m, \zeta \in \Z_{>0}$, $\epsilon>0$ and a Schottky set $S$ with multiplicity $\zeta$ and with resolution $N \ge 9\cdot 10^{6} \zeta^{2}$ such that $\mu'^{\ast m}$ is $(S, \epsilon)$-admissible for each $\mu' \in \mathcal{U}$.

If the action of $\llangle\supp \mu\rrangle$ is proximal without a global fixed point, we can moreover require that $S$ has an interval median.
\end{cor}

\begin{proof}
Let us first assume that $\llangle \supp \mu\rrangle$ does not have an invariant probability measure. By Theorem \ref{thm:titsAltGen}, there exists $m_{1}, \zeta' \in \Z_{>0}$ such that $(\supp \mu)^{m_{1}} \subseteq \supp \mu^{\ast m_{1}}$ contains a Schottky pair associated with essentially disjoint sets $I_{1}, I_{2}, J_{1}, J_{2}$ satisfying \[
\sup \{\zeta(I_{1}),\zeta(I_{2}),\zeta(J_{1}), \zeta(J_{2})\} \le \zeta'.
\]
We now apply Lemma \ref{lem:amplify} (with $N = 36\cdot 10^{6}\zeta'^{2}$) to obtain $m_{2} \in \Z_{>0}$, $\epsilon>0$ and a Schottky set $S$ with resolution $N = 36\cdot 10^{6} \zeta'^{2}$ and with multiplicity at most $\zeta := 2\zeta'$ such that $\mu^{\ast m_{1} m_{2}}$ is $(S, \epsilon)$-admissible. Here, let us write $S = \cup_{i=1}^{N} \mathfrak{S}(I_{i}, J_{i})$. 
Moreover, the set $\mathcal{V}$ of $(S, \epsilon)$-admissible probability measures is an open set, as $I_{i}$, $J_{i}$'s are open. Since the convolution operator on the space of probability measures on $\Homeo(S^{1})$ is continuous \cite[Proposition 3.1]{MR413217}, there exists a neighborhood $\mathcal{U}$ of $\mu$ such that $\mathcal{U}^{\ast m_{1}m_{2}}$ is contained in $\mathcal{V}$.

The proximal case can be handled by Corollary \ref{cor:weakTitsProximal} and Lemma \ref{lem:amplify}.
\end{proof}

\subsection{Exponential Synchronization}\label{subsect:expSync}

In this subsection, we prove \ref{thm:main4}, the exponential bound for proximal random walks on $\Homeo(S^{1})$. This will follow from elementary lemmas about nesting of intervals, assuming one technical proposition from the pivoting technique.

We now present a central proposition that follows from the pivoting technique. We postpone its proof to the next section.

\begin{prop}\label{prop:pivotingPrep}
For each $\epsilon>0$, $m \in \Z_{>0}$, there exists $\kappa = \kappa(\epsilon, m) > 0$ that satisfies the following.

Let $S$ be a Schottky set with multiplicity $\zeta$, with resolution $N \ge 9\cdot 10^{6}\zeta^{2}$ and with a median $\mathcal{I}$. Let $\mu$ be a probability measure on $\Homeo(S^{1})$ such that $\mu^{\ast m}$ is an $(S, \epsilon)$-admissible measure. 

Then for each $n \in \Z_{>0}$, there exists a probability space $(\Omega_{n}, \mathbb{P})$, a measurable subset $A_{n} \subseteq \Omega_{n}$, a measurable partition $\mathcal{P}_{n} = \{\mathcal{E}_{\alpha}\}_{\alpha}$ of the set $A_{n}$, and $\Homeo(S^{1})$-valued random variables \[
Z_{n}, \{w_{i}\}_{i=0, \ldots, \lfloor \kappa n \rfloor}, \{s_{i}\}_{i =1, \ldots, \lfloor \kappa n \rfloor }
\] defined on $\Omega_{n}$ such that the following holds: \begin{enumerate}
\item $\Prob(A_{n}) \ge 1 - \frac{1}{\kappa}e^{-\kappa n}$.
\item Restricted on each equivalence class $\mathcal{E} \in \mathcal{P}_{n}$, $w_{0}, \ldots, w_{\lfloor \kappa n \rfloor}$ are \emph{constant} homeomorphisms and $s_{i}$ are independently distributed according to a Schottky-uniform measures on $S$.
\item On $A_{n}$, $w_{i} \mathcal{I} \subseteq \mathcal{I}$ holds for each $i=1, \ldots, \lfloor \kappa n \rfloor - 1$. (Note: we do not enforce $w_{i} \mathcal{I} \subseteq \mathcal{I}$ for $i=0, \lfloor \kappa n \rfloor$.)
\item $Z_{n}$ is distributed according to $\mu^{\ast n}$ and $Z_{n}  = w_{0} s_{1} w_{1} \cdots s_{\lfloor \kappa n \rfloor } w_{\lfloor\kappa n \rfloor}$ holds on $A_{n}$.
\end{enumerate}
\end{prop}

We will prove the exponential synchronization assuming this proposition. \emph{From now on, we fix a measure $\Len$ on $S^{1}$. For Theorem \ref{thm:main1}, \ref{thm:main3}, \ref{thm:main4}, these can be taken as the Lebesgue measure. For Theorem \ref{thm:main5}, one can plug in an arbitrary measure}.

\begin{lem}\label{lem:expShrink1}
Let $w \in \Homeo(S^{1})$, let $S$ be a Schottky set with median $\mathcal{I}$ and with resolution $N$, and let $\mu$ be a Schottky-uniform measure on $S$. Then we have \[
\Prob_{s \sim \mu} \left( \Len(w s\mathcal{I}) \le \frac{1}{\sqrt{N}} \Len(w\mathcal{I}) \right) \ge 1 - \frac{1}{\sqrt{N}}.
\]
\end{lem}

\begin{proof}
First, let us write $S = \mathfrak{S}(I_{1}, J_{1}) \cup \ldots \cup \mathfrak{S}(I_{N}, J_{N})$ for some essentially disjoint sets   $I_{1}, \ldots, I_{N}, J_{1}, \ldots, J_{N}$. Recall that elements of  $\mathfrak{S}(I_{i}, J_{i})$ send $\mathcal{I}$ into  $J_{i}$. $(\ast)$ Note that $wJ_{1}, \ldots, wJ_{N}$ are disjointly contained in  $w\mathcal{I}$. Hence, the sum of their ``size" is no greater than that of  $w\mathcal{I}$, which implies that \[
Ind :=  \left\{ i : \Len (wJ_{i}) \ge \frac{1}{\sqrt{N}} \Len(w\mathcal{I}) \right\}
\]has at most $\sqrt{N}$ elements. For each $i \notin Ind$, ($\ast$) tells us that  $ \Len(w s\mathcal{I}) \le \frac{1}{\sqrt{N}} \Len(w\mathcal{I})$ for each $s \in \mathfrak{S}(I_{i}, J_{i})$. Summing up, we observe \[
\begin{aligned}
\Prob_{s \sim \mu}  \left( \Len(w s\mathcal{I}) \le \frac{1}{\sqrt{N}} \Len(w\mathcal{I}) \right)&\ge \Prob_{s \sim \mu} \big( s \in \mathfrak{S}(I_{i}, J_{i}) : i \notin Ind \big) \\
&\ge \frac{1}{N} (N-\sqrt{N}) = 1 - \frac{1}{\sqrt{N}}.\qedhere
\end{aligned}
\]
\end{proof}

\begin{lem}\label{lem:expShrink2}
Let $S$ be a Schottky set with median $\mathcal{I}$ and with resolution $N \ge 100$. Fix homeomorphisms $w_{0}, \ldots, w_{n} \in \Homeo(S^{1})$ that satisfy  $w_{i} \mathcal{I} \subseteq \mathcal{I}$ for $i=1, \ldots, n$. Then for random variables $s_{1}, \ldots, s_{n}$ independently distributed according to Schottky-uniform measures on $S$, we have \[
\Prob\left( \Len\big(w_{0} s_{1} w_{1} \cdots s_{n} w_{n} \cdot \mathcal{I}\big) \le \frac{1}{N^{n/4}} \Len\big(w_{0} \mathcal{I}\big) \right) \ge 1 - e^{-n/4}.
\]
\end{lem}

\begin{proof}
Again, we start by writing $S = \mathfrak{S}(I_{1}, J_{1}) \cup \ldots \cup \mathfrak{S}(I_{N}, J_{N})$. Note that for each $i$, each element $s$ of $\mathfrak{S}(I_{i}, J_{i})$ sends $\mathcal{I}$ into $J_{i}$ and satisfies $s\mathcal{I} \subseteq J_{i} \subseteq \mathcal{I}$. In other words, the inclusion\[
W_{0} \mathcal{I}\supseteq W_{0} s_{1} \mathcal{I} \supseteq W_{1} \mathcal{I} \supseteq W_{1} s_{2} \mathcal{I} \supseteq \ldots \supseteq W_{n} \mathcal{I}  \quad \big(W_{k} = W_{k}(s_{1}, \ldots, s_{k}) := w_{0} s_{1} w_{1} \ldots s_{k} w_{k} \big)
\]
holds regardless of the values of $s_{i}$'s.

Now fixing $0 \le k \le n-1$ and the choices of $\{s_{i} : 1 \le i \le k\}$, we observe that\[
\Prob_{s_{k+1}\sim \textrm{Schottky-uniform on $S$}} \left( \Len (W_{k} s_{k+1} \mathcal{I}) \le \frac{1}{\sqrt{N}} \Len(W_{k}\mathcal{I}) \right) \ge 1-\frac{1}{\sqrt{N}}
\]
thanks to Lemma \ref{lem:expShrink1}. In other words, for \[
E_{k} := \left\{(s_{1}, \ldots, s_{k}) : \Len (W_{k-1} s_{k} \mathcal{I}) \le \frac{1}{\sqrt{N}} \Len(W_{k-1} \mathcal{I}) \right\},
\]
we have $\Prob(E_{k+1} | s_{1}, \ldots, s_{k}) \ge 1 - 1/\sqrt{N}$ regardless of the values of $s_{1}, \ldots, s_{k}$. Summing up these conditional probabilities, we obtain \begin{equation}\label{eqn:ekSumEst}
\Prob \left( \sum_{k=1}^{n} 1_{E_{k}} \ge n/2\right) \ge \Prob\left( B(n, 1-1/\sqrt{N}) \ge n/2 \right).
\end{equation}
Here, $B(n, 1-1/\sqrt{N})$ denotes the binomial random variable, the sum of $N$ independent Bernoulli random variables with expectation $1-1/\sqrt{N}$. We use Markov's inequality to estimate the latter:\[
\begin{aligned}
e^{-n/2} \cdot \Prob\left( B(n, 1-1/\sqrt{N}) \le n/2 \right) \le \E\left[e^{- B(n, 1-1/\sqrt{N})}\right] \le \left( \frac{1}{\sqrt{N}} + e^{-1} \right)^{n}.
\end{aligned}
\]
Here, the assumption $\sqrt{N} \ge 10$ implies $1/\sqrt{N}+ e^{-1} \le e^{-3/4}$. We thus obtain $\Prob \left(B(n, 1-1/\sqrt{N}) \le n/2 \right) \le e^{-n/4}$. Combining this with Inequality \ref{eqn:ekSumEst}, we can conclude the proof.
\end{proof}

We have other analogous computations.

\begin{lem}\label{lem:expShrink3}
Let  $x, y \in S^{1}$, let $w \in \Homeo(S^{1})$, let $S$ be a Schottky set with median $\mathcal{I}$ and with resolution $N$, and let $\mu$ be a Schottky-uniform measure on $S$. Then we have
\[
\Prob_{s \sim \mu} \big( \{x, y\} \cap ws \mathcal{I}= \emptyset  \big) \ge 1 - 2/N
\]\end{lem}

\begin{proof}
Let $S = \mathfrak{S}(I_{1}, J_{1}) \cup \ldots\cup \mathfrak{S}(I_{N}, J_{N})$. Then for each $i$, every element of $\mathfrak{S}(I_{i}, J_{i}) \in S$ sends $\mathcal{I}$ into $J_{i}$. Since $J_{1}, \ldots, J_{N}$ are disjoint, $Ind := \{ i : \{w^{-1}x, w^{-1}y\} \cap J_{i} \neq \emptyset\}$ has cardinality at most 2. This implies  \[
\begin{aligned}
\Prob_{s \sim \mu} \left( \{x, y\} \cap ws\mathcal{I} = \emptyset   \right) \ge \Prob_{s \sim \mu} \big( s \in \mathfrak{S}(I_{i}, J_{i}) : i \notin Ind \big) \ge \frac{1}{N} (N-2).\qedhere
\end{aligned}
\]
\end{proof}

\begin{lem}\label{lem:expShrink4}
Let  $x, y \in S^{1}$, let $w_{0} \in \Homeo(S^{1})$ and  let $S$ be a Schottky set with median $\mathcal{I}$ and with resolution $N \ge 6$. Fix homeomorphisms $w_{1}, \ldots, w_{n} \in \Homeo(S^{1})$ such that $w_{i} \mathcal{I} \subseteq \mathcal{I}$ for $i=1, \ldots, n$. Then for random variables $s_{1}, \ldots, s_{n}$ independently distributed according to Schottky-uniform measures on $S$, we have \[
\Prob\big( \{x, y\} \cap w_{0} s_{1} w_{1} \cdots s_{n} w_{n}  \mathcal{I}= \emptyset \big) \ge 1 - e^{-n}
\]
\end{lem}

\begin{proof}
As in the proof of Lemma \ref{lem:expShrink2}, \[
W_{0} \mathcal{I} \supseteq W_{0} s_{1} \mathcal{I} \supseteq W_{1} \mathcal{I} \supseteq W_{1} s_{2} \mathcal{I}\supseteq \ldots \supseteq W_{n} \mathcal{I}  \quad \big(W_{k} = W_{k}(s_{1}, \ldots, s_{k}) := w_{0} s_{1} w_{1} \ldots s_{k} w_{k} \big)
\]
holds regardless of the choices of the $s_{i}$'s. Furthermore, when $0 \le k \le n-1$ and $\{s_{i} : 1 \le i \le k\}$ are given, \[
\Prob_{s_{k+1} \sim \textrm{Schottky-uniform on $S$} } \left( s_{k+1}\mathcal{I} \cap \{W_{k}^{-1}x, W_{k}^{-1}y\} =\emptyset \right) \ge 1- 2/N
\]
holds by Lemma \ref{lem:expShrink3}. In other words, if we define \[
E_{k} := \big\{ (s_{1}, \ldots, s_{k}) : \{x, y\} \cap W_{k-1} s_{k} \mathcal{I} = \emptyset \big\},
\]
then we have $\Prob(E_{k+1} | s_{1}, \ldots, s_{k}) \ge 1- 2/N$ for all choices of $s_{1}, \ldots, s_{k}$. This leads to \[
\Prob\big( \{x, y\} \cap W_{n} \mathcal{I} = \emptyset \big) \ge \Prob \left( E_{1} \cup \ldots \cup E_{n} \right) \ge 1 - (2/N)^{n} \ge 1-e^{-n}.\qedhere
\]
\end{proof}

We can interpret the above lemma in the following way. Let $S=\cup_{i} \mathfrak{S}(I_{i}, J_{i})$ be a Schottky set with median $\mathcal{I}$ and with resolution $N \ge 6$. Then  $\check{S} := \cup_{i} \mathfrak{S}(J_{i}, I_{i})$ becomes another Schottky set with median $S^{1} \setminus \mathcal{I}$. Now, given a Schottky-uniform measure $\mu$ on $S$, the measure $\check{\mu}$ defined by $\check{\mu}(\cdot) := \mu(\cdot^{-1})$ becomes a Schottky-uniform measure on $\check{S}$. Finally, consider some homeomorphisms $w_{0}, \ldots, w_{n}$ that satisfy the following equivalent condition: \[
w_{i} \mathcal{I} \subseteq \mathcal{I} \,\,\textrm{for}\,\, i= 0, \ldots, n-1 \Leftrightarrow w_{i}^{-1} (S^{1} \setminus \mathcal{I}) \subseteq (S^{1} \setminus \mathcal{I})\,\,\textrm{for}\,\, i= 0, \ldots, n-1 .
\]
Now by applying Lemma \ref{lem:expShrink4}, we observe for arbitrary $x, y \in S^{1}$ that \[
\Prob_{\textrm{$s_{i}^{-1}$ independently Schottky-uniform  on $\check{S}$}}  \big( \{x, y\} \cap w_{n}^{-1} s_{n}^{-1} w_{n-1}^{-1} \cdots s_{1}^{-1} w_{0}^{-1} (S^{1} \setminus \mathcal{I}) =\emptyset \big) \ge 1-e^{-n}.
\]
Equivalently, we can say  \[
\Prob_{\textrm{$s_{i}$ independently Schottky-uniform on  $S$}}  \big(w_{0} s_{1} w_{1} \cdots s_{n} w_{n} \cdot \{x, y\}  \subseteq \mathcal{I} \big) \ge 1-e^{-n}.
\]
We record this as a separate lemma:

\begin{lem}\label{lem:expShrink4.5}
Let  $x, y \in S^{1}$, let $w \in \Homeo(S^{1})$ and  let $S$ be a Schottky set with median $\mathcal{I}$ and with resolution $N \ge 6$. Fix homeomorphisms $w_{0}, \ldots, w_{n} \in \Homeo(S^{1})$ such that $w_{i} \mathcal{I} \subseteq \mathcal{I}$ for $i=0, \ldots, n-1$. Then for random variables $s_{1}, \ldots, s_{n}$ independently distributed according to Schottky-uniform measures on $S$, we have \[
\Prob  \big(w_{0} s_{1} w_{1} \cdots s_{n} w_{n} \cdot \{x, y\}  \subseteq \mathcal{I} \big) \ge 1-e^{-n}.
\]
\end{lem}

We can now prove Theorem \ref{thm:main4}.

 \begin{thm}\label{thm:expShrink}
 
For each $\epsilon>0$ and $m \in \Z_{>0}$, there exists  $\kappa_{1} = \kappa_{1}(\epsilon, m) > 0$ such that the following holds.

Let $S$ be a Schottky set with resolution $N\ge 9\cdot 10^{6}$, with multiplicity 1 and with an interval median $\mathcal{I}$. Let $\mu$ be a probability measure on $\Homeo(S^{1})$ such that $\mu^{\ast m}$ is $(S, \epsilon)$-admissible. Then for every $x, y \in S^{1}$ and for every $n \in \Z_{>0}$ we have \[
\Prob_{Z_{n} \sim \mu^{\ast n}} \big( d(Z_{n}x, Z_{n}y) \le e^{-\kappa_{1} n} \big) \ge 1 -\frac{1}{\kappa_{1}} e^{-\kappa_{1} n}.\]
\end{thm}

\begin{proof}
For this proof we will employ the Lebesgue measure, i.e., $\Len = \Leb$.

Let $\kappa = \kappa(\epsilon, m)$ be as in Proposition \ref{prop:pivotingPrep}. Next, given a positive integer $n$, we fix the probability space $\Omega_{n}$, the measurable subset $A_{n}$, the measurable partition $\mathcal{P}_{n} = \{\mathcal{E}_{\alpha}\}_{\alpha}$ of $A_{n}$ and the random variables $Z_{n}$, $w_{0}, \ldots, w_{\lfloor \kappa n  \rfloor}, s_{1}, \ldots, s_{\lfloor \kappa n \rfloor}$ as in Proposition \ref{prop:pivotingPrep}.

Let $\mathcal{E} \in \mathcal{P}_{n}$ be an arbitrary equivalence class. Restricted on $\mathcal{E}$, $w_{0}, \ldots, w_{\lfloor \kappa n  \rfloor}$ are constant homeomorphisms and $s_{1}, \ldots, s_{\lfloor \kappa n \rfloor}$ are independently distributed according to Schottky-uniform measures on  $S$. Furthermore, each of $w_{1}, \ldots, w_{\lfloor \kappa n \rfloor-1}$ satisfy $w_{i}\mathcal{I} \subseteq \mathcal{I}$. This allows us to apply Lemma  \ref{lem:expShrink2} and \ref{lem:expShrink4.5}.

For convenience, let us define  $w_{0}' := w_{0} s_{1} w_{1} \cdots  s_{\lfloor 0.5\kappa n \rfloor } w_{\lfloor 0.5 \kappa n \rfloor}$. This homeomorphism depends on the choices of $s_{1}, \ldots,  s_{\lfloor 0.5\kappa n \rfloor }$. By Lemma \ref{lem:expShrink2}, we have 
 \[
\Prob\left( \Len (w_{0}' \mathcal{I} = w_{0} s_{1} w_{1} \cdots s_{\lfloor 0.5\kappa n \rfloor } w_{\lfloor 0.5\kappa n \rfloor } \cdot \mathcal{I}) \le \frac{1}{N^{\lfloor0.5 \kappa n \rfloor/4}} \cdot 1 \, \Big| \, \mathcal{E}\right) \ge 1 - e^{-\lfloor0.5 \kappa n \rfloor/4}.
\]
The event depicted here does not depend on $s_{\lfloor 0.5\kappa n \rfloor+1}, \ldots, s_{\lfloor\kappa n \rfloor }$ whatsoever. Moreover, by Lemma \ref{lem:expShrink4.5}, we observe the following regardless of the nature of $w_{0}'$: \[
\Prob \Big( w_{0}'  s_{\lfloor 0.5\kappa n \rfloor +1} w_{ \lfloor 0.5\kappa n \rfloor+1 } \cdots s_{\lfloor \kappa n \rfloor } w_{\lfloor \kappa n \rfloor} \cdot \{x, y\} \subseteq w_{0}'\mathcal{I}\, \Big|\, \mathcal{E}, w_{0}' \Big) \ge 1 - e^{-\lfloor 0.5\kappa n \rfloor}.
\]
Lastly, we have $\diam(w_{0}'\mathcal{I})  \le \Leb (w_{0}' \mathcal{I})$ precisely because $w_{0}'\mathcal{I}$ is an interval.

Combined together, we have  \[
\Prob \Big( d(Z_{n}x, Z_{n}y)\le \Len (w_{0} '\cdot \mathcal{I}) \le  \frac{1}{N^{\lfloor0.5 \kappa n \rfloor/4}} \, \Big| \, \mathcal{E} \Big) \ge (1- e^{-\lfloor0.5 \kappa n \rfloor/4})(1-e^{-\lfloor 0.5\kappa n \rfloor }) \ge 1 - 2 \cdot e^{-\lfloor0.5 \kappa n \rfloor/4}.
\]
Since we observe this lower bound on each of $\mathcal{E} \in \mathcal{P}_{n}$, we can sum up the conditional probability to deduce \[
\begin{aligned}
\Prob\big( d(Z_{n}x, Z_{n}y)\le N^{-\lfloor0.5 \kappa n \rfloor/4} \big) &\ge \sum_{\mathcal{E} \in \mathcal{P}_{n}} \Prob(\mathcal{E}) \Prob\big(d(Z_{n}x, Z_{n}y)\le N^{-\lfloor0.5 \kappa n \rfloor/4}\, \big|\,\mathcal{E}\big) \\
&\ge \sum_{\mathcal{E} \in \mathcal{P}_{n}} \Prob(\mathcal{E}) \cdot (1- 2e^{-\lfloor0.5 \kappa n \rfloor/4}) \\
&= (1- 2e^{-\lfloor0.5 \kappa n \rfloor/4}) \Prob(A_{n}) \ge (1-2 e^{-\lfloor0.5 \kappa n \rfloor/4}) \left(1 - \frac{1}{\kappa}e^{-\kappa n}\right). \qedhere
\end{aligned}
\]
\end{proof}

Theorem \ref{thm:main4} now follows from Theorem \ref{thm:expShrink} together with Corollary \ref{cor:measureConvTarget}.

\subsection{Probabilistic Tits alternative}\label{subsection:probTits}

We now turn to the proof of Theorem \ref{thm:main1}.

\begin{lem}\label{lem:1segment}
Let $N$ be an integer greater than 4. Let $I_{1}, \ldots, I_{N}, J_{1}, \ldots, J_{N}$ be intervals such that $I_{1}, \ldots, I_{N}$ are mutually disjoint and $J_{1}, \ldots, J_{N}$ are mutually disjoint. Then for any homeomorphism $g \in \Homeo(S^{1})$, we have \[
\# \big\{ (i, j) \in \{1, \ldots, N\}^{2} : \textrm{$I_{i}$ and $gJ_{j}$ intersect}\big\} \le 3N\sqrt{N}.
\]
\end{lem}
\begin{proof}
We first let \[
\mathcal{C}=\mathcal{C}(g) := \big\{ I_{i} : \# \{ j : I_{i} \cap gJ_{j} \neq \emptyset\} \ge \sqrt{N} \big\}
\]
Then each element of  $\mathcal{C}$ meets more than 2 out of $\{gJ_{1}, \ldots, gJ_{N}\}$, so it is not completely contained in a single $gJ_{j}$. Hence, each $gJ_{j}$ can meet at most 2 elements of $\mathcal{C}$ (otherwise $gJ_{j}$ will contain an element of $\mathcal{C}$). Hence, \[\begin{aligned}
2N = 2\# \{gJ_{j} : j=1, \ldots, N\} &\ge 2\# \{gJ_{j} : \textrm{$gJ_{j}$ meets some element of  $\mathcal{C}$}\}\\
&\ge \#\big\{ (I_{i}, gJ_{j}) : I_{i} \cap gJ_{j} \neq \emptyset, I_{i} \in \mathcal{C} \big\}\\
&\ge \#\mathcal{C} \cdot \min_{I_{i} \in \mathcal{C}} \#\{ J_{j} : I_{i} \cap gJ_{j} \neq \emptyset\} \\
&\ge  \# \mathcal{C} \sqrt{N}
\end{aligned}
\]
holds, which implies that $\mathcal{C}$ has at most $2\sqrt{N}$ elements.

Now fixing  $I_{i}\notin \mathcal{C}$, the number of $gJ_{j}$'s that meet $I_{i}$ is at most $\sqrt{N}$. Summing up, we have \[\begin{aligned}
\# \big\{ (i, j) \in \{1, \ldots, N\}^{2} : \textrm{$I_{i}$ and  $gJ_{j}$ intersect}\big\} &\le \# \mathcal{C} \cdot N + (N-\#\mathcal{C}) \cdot \sqrt{N} \\
&\le 2N\sqrt{N} + N\sqrt{N} = 3N\sqrt{N}. \qedhere
\end{aligned}
\] 
\end{proof}

The previous lemma generalizes as follows.

\begin{lem}\label{lem:1segmentZeta}
Let $N \in \Z_{>4}$ and $\zeta \in \Z_{>0}$. Let $I_{1}, \ldots, I_{N}, J_{1}, \ldots, J_{N}$ be sets with $\le \zeta$ connected components such that $I_{1}, \ldots, I_{N}$ are mutually disjoint and $J_{1}, \ldots, J_{N}$ are mutually disjoint. Then for any homeomorphism $g \in \Homeo(S^{1})$, we have \[
\# \big\{ (i, j) \in \{1, \ldots, N\}^{2} : \textrm{$I_{i}$ and $gJ_{j}$ intersect}\big\} \le 3\zeta N\sqrt{N}.
\]
\end{lem}
\begin{proof}
We can decompose each $I_{i}$, $J_{i}$ into $\zeta$ disjoint intervals: there exist $\zeta N$ disjoint intervals $\{ I_{i}^{(l)} : i=1, \ldots, N, l = 1, \ldots, \zeta\}$ and $\zeta N$ disjoint intervals $\{ J_{j}^{(m)} : j=1, \ldots, N, m = 1, \ldots, \zeta\}$ (some of which is possibly empty) such that \[
I_{i} = I_{i}^{(1)} \sqcup \ldots \sqcup I_{i}^{(\zeta)}, \quad 
J_{i} = J_{i}^{(1)} \sqcup \ldots \sqcup J_{i}^{(\zeta)} \quad (i=1, \ldots, N).
\]
Let us now define \[\begin{aligned}
\mathcal{C}= \mathcal{C}(g) &:= \Big\{ I_{i} : \#\{j : I_{i} \cap gJ_{j} \neq \emptyset\} \ge \zeta \sqrt{N} \Big\},\\
\mathcal{C}' = \mathcal{C}'(g) &:= \Big\{ (i, l) : \# \{(j, m) : I_{i}^{(l)} \cap J_{j}^{(m)} \neq \emptyset\} \ge \sqrt{N} \Big\}.
\end{aligned}
\]

By mimicking the argument for Lemma \ref{lem:1segment}, we can observe that \[
2\zeta N \ge \# \mathcal{C}' \cdot \sqrt{N}.
\]
In other words, $\#\mathcal{C}'$ is at most $2\zeta\sqrt{N}$. Note that, by the pigeonhole principle, for each $I_{i} \in \mathcal{C}$, there exists $l \in \{1, \ldots, \zeta\}$ such that $(i, l) \in \mathcal{C}'$. Consequently, we have $\#\mathcal{C} \le\# \mathcal{C}' \le 2\zeta \sqrt{N}$.

Fixing $I_{i} \notin \mathcal{C}$, the number of $gJ_{j}$ that meets $I_{i}$ is at most $\zeta\sqrt{N}$. Summing up, we have \[\begin{aligned}
\# \big\{ (i, j) \in \{1, \ldots, N\}^{2} : \textrm{$I_{i}$ and  $gJ_{j}$ intersect}\big\} &\le \# \mathcal{C} \cdot N + (N-\#\mathcal{C}) \cdot \zeta \sqrt{N} \\
&\le 2N\zeta\sqrt{N} + N\zeta\sqrt{N} = 3\zeta N\sqrt{N}. \qedhere
\end{aligned}
\]

\end{proof}

\begin{lem}\label{lem:1segmentSchottky}
Let $S$ and $S'$ be Schottky sets with multiplicity $\le \zeta$, with resolution $N \ge 10\zeta^{2}$ and with medians $\mathcal{I}$ and $\mathcal{I}'$, respectively. Let $s$ and $s'$ be independent random variables that are Schottky-uniform on $S$ and $S'$, respectively. Then for each $g \in \Homeo(S^{1})$ we have \[
\Prob \big(  s' g s \mathcal{I} \,\,\textrm{ is essentially contained in }\mathcal{I}'\big)\ge 1 - 3\zeta/\sqrt{N}.
\]
\end{lem}
\begin{proof}
Let $S=\mathfrak{S}(I_{1}, J_{1})\cup \ldots\cup \mathfrak{S}(I_{N}, J_{N})$ and $S' = \mathfrak{S}(I'_{1}, J'_{1})\cup\ldots\cup\mathfrak{S}(I'_{N}, J'_{N})$. Then for each $i$, the inverse $s'^{-1}$ of an arbitrary element $s'$ of $\mathfrak{S}(I'_{i}, J_{i}')$ sends $\overline{S^{1} \setminus \mathcal{I}'}$ into $I_{i}'$. Meanwhile, an arbitrary element $s$ of $\mathfrak{S}(I_{i}, J_{i})$ sends $\overline{\mathcal{I}}$ into $J_{i}$. Now Lemma \ref{lem:1segmentZeta} tells us that \[
Ind:= \big\{ (i, j) : \textrm{$I_{i}'$ and $gJ_{j}$ intersect}\big\}
\]
has at most  $3\zeta N\sqrt{N} $ elements. Moreover, given $(i, j) \notin Ind$,  for every $s \in \mathfrak{S}(I_{i}, J_{i})$ and  $s' \in \mathfrak{S}(I_{j}', J_{j}')$ we have\[
gs \overline{\mathcal{I}} \subseteq gJ_{j}\subseteq S^{1} \setminus I_{i}' \subseteq s'^{-1} \inte \mathcal{I}'
\]
Summing up, we conclude \[
\begin{aligned}
\Prob \left( s' g s \bar{\mathcal{I}} \subseteq \inte \mathcal{I}'  \right) \ge \Prob \big( s \in \mathfrak{S}(I_{i}, J_{i}), s' \in \mathfrak{S}(I_{j}', J_{j}') :  (i, j) \notin Ind \big) \ge 1-3\zeta/\sqrt{N}.\qedhere
\end{aligned}
\]
\end{proof}

\begin{lem}\label{lem:1segmentAction}
Let $S$ and $S'$ be Schottky sets with multiplicity $\le \zeta$, with resolution $N \ge 100\zeta^{2}$ and with medians $\mathcal{I}$ and $\mathcal{I}'$, respectively. For $i=1, \ldots, n$, let $s_{i}$ be a Schottky-uniform measure on $S$ and let $s_{-i}$ be a Schottky-uniform measure on $S'$. Suppose that the $\{s_{i} : 1 \le |i| \le n\}$ are all independent. Fix homeomorphisms $\{w_{i} : -n \le i \le n\}$ such that \[
w_{i} \mathcal{I} \subseteq \mathcal{I} , \,\, w_{-i} \mathcal{I}' \subseteq \mathcal{I}' \quad (1 \le i \le n).
\]

Then we have\[
\Prob \big( w_{-n} s_{-n}  \cdots w_{-1}s_{-1} \cdot w_{0} \cdot s_{1} w_{1} \cdots s_{n} w_{n} \cdot \bar{\mathcal{I}} \subseteq \inte \mathcal{I}' \big) \ge 1-e^{-n}
\]
\end{lem}
Note that we have not assumed any restriction on $w_{0}$ in Lemma \ref{lem:1segmentAction}.

\begin{proof}
We define $W_{0} := id$ and define $W_{k} := s_{1} w_{1} \cdots s_{k} w_{k}$, $W_{-k} := w_{-k} s_{-k} \cdots w_{-1} s_{-1}$. Then the following inclusion holds: \[
\begin{aligned}
W_{0}\mathcal{I} &\supseteq W_{0} s_{1} \mathcal{I} \supseteq  W_{1} \mathcal{I}\supseteq W_{1} s_{2} \mathcal{I} \supseteq \ldots \supseteq W_{n} \mathcal{I}, \\
W_{0}^{-1} \mathcal{I}' & \subseteq (s_{-1}W_{0})^{-1} \mathcal{I}' \subseteq W_{-1}^{-1} \mathcal{I}' \subseteq (s_{-2} W_{-1})^{-1} \mathcal{I}' \subseteq W_{-2}^{-1} \mathcal{I}' \subseteq \ldots \subseteq W_{-n}^{-1} \mathcal{I}',
\end{aligned}
\]
regardless of the choices of the $s_{i}$'s. We now define  \[
E_{k} := \big\{ (s_{-k}, \ldots,s_{-1}, s_{1}, \ldots, s_{k}) : w_{0}W_{k-1} s_{k} \bar{\mathcal{I}} \subseteq (s_{-k} W_{-(k-1)})^{-1}\inte \mathcal{I}' \big\}.
\]
Then Lemma \ref{lem:1segmentSchottky} tells us that  \[
\Prob\big( E_{k+1} \, \big| \, s_{-k}, \ldots, s_{k} \big) \ge 1-3\zeta/\sqrt{N} \ge 1-1/e
\]
holds regardless of the choices of $s_{-k}, \ldots, s_{k}$. Summing up the conditional probabilities, we conclude 
\[
\Prob\big(w_{0}W_{n} \bar{\mathcal{I}} \subseteq W_{-n}^{-1} \inte \mathcal{I}' \big) \ge \Prob \big( E_{1} \cup \ldots \cup E_{n} \big)\ge 1-e^{-n}.\qedhere
\]
\end{proof}

\begin{thm}\label{thm:freeSubgpEquiv}
Let $S$ and $S'$ be Schottky sets with multiplicity $\le \zeta$, with resolution $N \ge 100\zeta^{2}$ and with medians $\mathcal{I}$ and $\mathcal{I}'$, respectively. Let $\mu$  and $\mu'$ be Schottky-uniform measures on $S$ and $S'$, respectively. Fix homeomorphisms $w_{0}, v_{0}, w_{1}, v_{1}, \ldots, w_{2n}, v_{2n} \in \Homeo(S^{1})$ such that \[
w_{i} \mathcal{I} \subseteq \mathcal{I}, \,\,v_{i} \mathcal{I}' \subseteq \mathcal{I}' \quad (i=1, \ldots, 2n-1)
\]
Let $s_{1}, \ldots, s_{2n}$ ($t_{1}, \ldots, t_{2n}$, resp.) be random variables distributed according to a Schottky-uniform measure on $S$ ($S'$, resp.), all independent. Then we have  \[
\Prob \left(\begin{array}{c} \textrm{$w_{0} s_{1} w_{1} \cdots s_{2n} w_{2n}$ and $v_{0}t_{1} v_{1} \cdots t_{2n} v_{2n}$ comprise}\\ 
 \textrm{a Schottky pair and generate a free subgroup}\end{array} \right) \ge 1 - 6e^{-n}.\]
\end{thm}

\begin{proof}
We define the following events.\[\begin{aligned}
E_{1} &:= \big\{ s_{n+1} w_{n+1} \cdots s_{2n}w_{2n} \cdot w_{0} s_{1} w_{1} \cdots s_{n} w_{n} \bar{\mathcal{I}} \subseteq \inte \mathcal{I} \big\}, \\
E_{2} &:= \big\{ t_{n+1} v_{n+1} \cdots t_{2n}v_{2n} \cdot v_{0} t_{1} v_{1} \cdots t_{n} v_{n} \bar{\mathcal{I}'}\subseteq \inte \mathcal{I}' \big\}, \\
E_{3} &:= \big\{ s_{n+1} w_{n+1} \cdots s_{2n}w_{2n} \cdot v_{0} t_{1} v_{1} \cdots t_{n} v_{n} \bar{\mathcal{I}'} \subseteq \inte \mathcal{I} \big\}, \\
E_{4} &:= \big\{ t_{n+1} v_{n+1} \cdots t_{2n}v_{2n} \cdot w_{0} s_{1} w_{1} \cdots s_{n} w_{n} \bar{\mathcal{I}} \subseteq \inte \mathcal{I}' \big\}, \\
E_{5} &:= \big\{ s_{n + 1} w_{n + 1} \cdots s_{2n} w_{2n} \cdot v_{2n}^{-1} t_{2n}^{-1} \cdots v_{n+1}^{-1} t_{n+1}^{-1} \cdot \overline{S^{1} \setminus \mathcal{I}'} \subseteq \inte \mathcal{I}\big\}, \\
E_{6} &:= \big\{ v_{n}^{-1} t_{n}^{-1} \cdots v_{1}^{-1} t_{1}^{-1} v_{0}^{-1} \cdot w_{0} s_{1} w_{1} \cdots s_{n }w_{n} \cdot \bar{\mathcal{I}} \subseteq \inte (S^{1} \setminus \mathcal{I}') \big\}.
\end{aligned}
\]
Let us study the first event. Here, $s_{i}$'s are Schottky-uniformly and independently distributed on $S$, $\mathcal{I}$ is a median for $S$, and $w_{i}\mathcal{I} \subseteq \mathcal{I}$ holds for each $i \neq 0, 2n$. (Note that $w_{2n} \cdot w_{0}$ does not nest $\mathcal{I}$.) By Lemma \ref{lem:1segmentAction}, we conclude  $\Prob(E_{1}) \ge 1 - e^{-n}$. For a similar reason, we conclude that the probabilities of $E_{2}$, $E_{3}$ and $E_{4}$ are all at least  $1 - e^{-n}$.

Before studying the fifth event, let us first write  $S'=\mathfrak{S}(I_{1}', J_{1}')\cup  \ldots \cup \mathfrak{S}(I_{N}', J_{N}')$ and revert it: $\check{S}' := \mathfrak{S}(J_{1}', I_{1}')\cup  \ldots \cup \mathfrak{S}(J_{N}', I_{N}')$. Then the $s_{i}$'s are Schottky-uniformly and independently distributed on $S$, whereas the $t_{i}^{-1}$'s are Schottky-uniformly and independently distributed on $\check{S}'$. Moreover, $\mathcal{I}$ is a median for $S$ and $w_{i} \mathcal{I} \subseteq \mathcal{I}$ holds for each $i$, whereas $S^{1} \setminus \mathcal{I}'$ is a median for $\check{S}'$ and $v_{i}^{-1} (S^{1} \setminus \mathcal{I}') \subseteq (S^{1} \setminus \mathcal{I}')$ holds for each $i$, Now, Lemma \ref{lem:1segmentAction} tells us that  $\Prob(E_{5}) \ge 1 - e^{-n}$. A similar argument tells us that $\Prob( E_{6}) \ge 1-e^{-n}$.

Now in the event $E_{1} \cap E_{2} \cap E_{3} \cap E_{4} \cap E_{5} \cap E_{6}$, we will analyze the relative positions of the sets\[\begin{aligned}
I^{(1)} &:= (s_{n+1} w_{n+1} \cdots s_{2n} w_{2n})^{-1}  (S^{1} \setminus \inte \mathcal{I}), \\
I^{(2)}& := (t_{n+1} v_{n+1} \cdots t_{2n} v_{2n})^{-1} \cdot (S^{1} \setminus  \inte \mathcal{I}'), \\
J^{(1)} &:= w_{0} s_{1} w_{1} \cdots s_{n} w_{n} \overline{\mathcal{I}}, \\
J^{(2)} &:= v_{0}t_{1} v_{1} \cdots t_{n} v_{n} \overline{\mathcal{I}'}.
\end{aligned}
\]
First, since we are in the event $E_{1}$, $I^{(1)}$ and $J^{(1)}$ do not overlap with each other. Similarly, the definition of $E_{2}$ tells us that  $I^{(2)}$ and  $J^{(2)}$ do not meet. The definition of $E_{3}$ ($E_{4}$, $E_{5}$ and $E_{6}$, resp.) tells us that $I^{(1)}$ and  $J^{(2)}$ ($I^{(2)}$ and $J^{(1)}$; $I^{(1)}$ and  $I^{(2)}$;  $J^{(1)}$ and  $J^{(2)}$, resp.) do not meet. In summary, all the 4 intervals are mutually disjoint in the event $\cap_{k=1}^{6} E_{k}$. Meanwhile, $w_{0} s_{1} w_{1} \cdots s_{2n} w_{2n}$ sends $S^{1} \setminus I^{(1)}$ into $\inte J^{(1)}$ and $v_{0} t_{1} v_{1} \cdots t_{2n} v_{2n}$ sends $S^{1} \setminus I^{(2)}$ into $\inte J^{(2)}$. 

In conclusion, $w_{0} s_{1} w_{1} \cdots s_{2n} w_{2n}$ and  $v_{0} t_{1} v_{1} \cdots t_{2n} v_{2n}$ comprise a Schottky pair associated with essentially disjoint sets $I^{(1)}, I^{(2)}, J^{(1)}, J^{(2)}$ and generate a (rank-2) free subgroup of $\Homeo(S^{1})$, when in the event $\cap_{k=1}^{6} E_{k}$. Since $\Prob(E_{k}^{c}) \le e^{-n}$ for each $k$, we conclude that  $\cap_{k=1}^{6} E_{k}$ has probability at least $1-6e^{-n}$.
\end{proof}

Now as in the proof of Theorem \ref{thm:expShrink}, we can derive the following theorem from Theorem  \ref{thm:freeSubgpEquiv} using the probability space and measurable partition guaranteed in Proposition  \ref{prop:pivotingPrep}.

\begin{thm}\label{thm:freeSubgp}
For each $\epsilon > 0$ and  $m \in \Z_{>0}$, there exists $\kappa_{2} = \kappa_{2}(\epsilon, m, N) > 0$ that satisfies the following.

Let $S$ and $S'$ be Schottky sets with multiplicity $\le \zeta$ and resolution $N \ge 9\cdot 10^{6}\zeta^{2}$. Let $\mu$ and $\mu'$ be probability measures on $\Homeo(S^{1})$ such that $\mu^{\ast m}$ is $(S, \epsilon)$-admissible and $\mu'^{\ast m}$ is $(S', \epsilon)$-admissible. Then for each $n \in \Z_{>0}$ we have  \[
\Prob_{(Z_{n}, Z_{n}') \sim \mu^{\ast n} \times \mu'^{\ast n}} \big(\textrm{$Z_{n}, Z_{n}'$ comprise a  Schottky pair and generate a free subgroup} \big) \ge 1 - \frac{1}{\kappa_{2}} e^{-\kappa_{2} n}.
\]
\end{thm}

Theorem \ref{thm:main1} now follows from Theorem \ref{thm:freeSubgp} together with Corollary \ref{cor:measureConvTarget}.

\subsection{Local contraction}\label{subsect:localCont}

Note that Theorem \ref{thm:main1} and \ref{thm:main4} are regarding ``snapshots" of a random walk at a certain step. Meanwhile, Theorem \ref{thm:main2} and Theorem \ref{thm:main3} ask for a specific choice of $I_{x, \omega}$, when the input $x \in S^{1}$ and a sample point $\w$ in the probability space is given. This does not only rely on the distribution $\mu^{\ast n}$ of $Z_{n}$ for each $n$ but their entire joint distribution. In fact, the same result will not hold for the right random walk.

\begin{proof}[Proofs of Theorem \ref{thm:main2} and \ref{thm:main3}]
To begin the proof, we first apply Corollary \ref{cor:measureConvTarget}. We obtain an open neighborhood $\mathcal{U}$ of $\mu$, integers $m, \zeta \in \Z_{>0}$, a positive number $\epsilon>0$ and  a Schottky set $S$ with multiplicity $\zeta$ and with resolution $N \ge 9\cdot 10^{6} \zeta^{2}$ such that $\mu'^{\ast m}$ is $(S, \epsilon)$-admissible for each $\mu' \in \mathcal{U}$. Let $\mathcal{I}$ be a median for $S$. Note that for Theorem \ref{thm:main3} (where $\llangle \supp \mu\rrangle$ acts proximally without global fixed points), we can choose $\mathcal{I}$ as an interval.

Let $\kappa$ be as in Proposition \ref{prop:pivotingPrep} for $\epsilon$ and $m$. To ease the notation, we will assume that $1/\kappa \in \N$. Then it suffices to prove the statement only for $n$ being multiples of $100/\kappa$.

Let us consider a large ambient space \[
\Omega := (G^{\Z_{>0}}, \mu^{\Z_{>0}})
\]
equipped with i.i.d.s $g_{i}$ distributed according to $\mu$. We adopt the left random walk convention in this proof, i.e., $Z_{i} := g_{i} \cdots g_{1}$.

We now regard $\Omega$ as a product space \[
\cdots \times \Omega_{3} \times \Omega_{2} \times  \Omega_{1} =: \Omega,
\]
where $\Omega_{k}$ is the space for the coordinates $(g_{n(2^{k} - 1)}, g_{n(2^{k} -1) - 1}, \ldots, g_{n(2^{k-1} - 1) + 1})$ for $k\ge1$. Note the relation \[
g_{n(2^{k}-1)} \cdots g_{n(2^{k-1}-1)+ l} = Z_{n(2^{k}-1)} \cdot Z_{n(2^{k-1}-1)+l-1}^{-1} \quad (l=1, \ldots, n2^{k-1}).
\]

We now apply Proposition \ref{prop:pivotingPrep} for each $\Omega_{k}$. Then $\Omega_{k}$ is now equipped with a measurable subset $A^{(k)}$, a measurable partition $\mathcal{P}^{(k)} = \{\mathcal{E}_{\alpha}^{(k)}\}_{\alpha}$ of $A^{(k)}$, and random variables \[
\{w_{i}^{(k)}\}_{i=0, \ldots,  \kappa n2^{k-1} }, \, \{s_{i}^{(k)}\}_{i=1, \ldots,  \kappa n2^{k-1} }
\]
such that: \begin{enumerate}
\item $\Prob(A^{(k)}) \ge 1 - \frac{1}{\kappa}e^{-\kappa n \cdot 2^{k-1}}$.
\item Restricted on each equivalence $\mathcal{E} \in \mathcal{P}_{k}$, $w_{0}^{(k)}, \ldots, w_{ \kappa n 2^{k-1} }^{(k)}$ are \emph{constant} homeomorphisms and $s_{i}^{(k)}$'s are independently distributed according to a Schottky-uniform measures on $S$.
\item On $A^{(k)}$, $w_{i}^{(k)} \mathcal{I} \subseteq \mathcal{I}$ holds for each $i=1, \ldots,  \kappa n 2^{k-1}  - 1$;
\item For each $\w \in A^{(k)}$ we have \begin{equation}\label{eqn:gnProd}
\begin{aligned}
w_{0}^{(k)}(\w) s_{1}^{(k)}(\w) \cdots s_{ \kappa n 2^{k-1}  }^{(k)}(\w) w_{\kappa n  2^{k-1} }^{(k)}(\w) &=  g_{n(2^{k}-1)}(\w) g_{n(2^{k}-1) - 1}(\w) \cdots g_{n(2^{k-1 }-1)+1}(\w) \\
&= Z_{n(2^{k}-1)}(\w) \cdot Z_{n(2^{k-1}-1)}^{-1}  (\w) .
\end{aligned}
\end{equation}
\end{enumerate}
Also, the partitions $\mathcal{P}^{(k)}$'s for distinct $k$'s are all independent.

Let us now define the event \[
F_{k} := \Big\{\w :  s_{0.9 \kappa n 2^{k} +1}^{(k+1)}  w_{0.9 \kappa n 2^{k} +1}^{(k+1)} \cdots s_{ \kappa n 2^{k} }^{(k+1)} w_{ \kappa n 2^{k} }^{(k+1)} \cdot w_{0}^{(k)} s_{1}^{(k)} w_{1}^{(k)} \cdots s_{0.1   \kappa n 2^{k-1} }^{(k)}w_{0.1  \kappa n  2^{k-1} }^{(k)} \mathcal{I} \subseteq \mathcal{I}\Big\}.
\]
For each $\mathcal{E}' \in \mathcal{P}^{(k+1)}$ and $\mathcal{E} \in \mathcal{P}^{(k)}$, the conditional probability of $F_{k}$ on $\mathcal{E}' \times \mathcal{E}$ is at least $1-e^{-0.1 \kappa n 2^{k-1}}$ by Lemma \ref{lem:1segmentAction}. Also, the probability of $A^{(k+1)} \times A^{(k)}$ is at least $1-\frac{2}{\kappa} e^{-\kappa n 2^{k-1}}$. Summing up the conditional probability, we conclude  \[
\Prob(F_{k}) \ge 1 - (1+2/\kappa) e^{-0.1 \kappa n 2^{k-1}}. \quad (k=1, 2, \ldots)
\]

Next, for each $k\ge 1$ and for each $n(2^{k}-1) < t \le n(2^{k+1} - 1)$, we define \[
End_{t} := \Big\{ \Len \big(Z_{t} Z_{n ( 2^{k}-1)}^{-1} \cdot w_{0}^{(k)} s_{1}^{(k)} w_{1}^{(k)} \cdots s_{0.5  \kappa n 2^{k-1} }^{(k)} w_{0.5 \kappa n  2^{k-1} }^{(k)} \cdot \mathcal{I} \big) \le \frac{1}{ N^{\kappa n 2^{(k-1)}/8}} \Big\}.
\]
For each choice of $(g_{n ( 2^{k+1} - 1)}, \ldots, g_{n ( 2^{k}-1) + 1}) \in \Omega_{k+1}$ and each $\mathcal{E} \in \mathcal{P}^{(k)}$, $Z_{t} Z_{n (2^{k}-1)}^{-1}$ is pinned down together with $w_{0}^{(k)}, w_{1}^{(k)}$, $\ldots$, whereas $s_{1}^{(k)}, s_{2}^{(k)}$, $\ldots$ are independently Schottky-uniformly distributed on $S$. Now Lemma \ref{lem:expShrink2} tells us that \[
\Prob(End_{t} \, | \,g_{n(2^{k+1} - 1)}, \ldots, g_{n(2^{k}-1) + 1}, \mathcal{E}) \ge 1 - e^{-\kappa n 2^{k-1}/8} \ge 1 - e^{-0.01\kappa t}.
\] Summing up the conditional probability across $\Omega_{k+1} \times A^{(k)}$, whose total probability is at least $1 - \frac{1}{\kappa} e^{-\kappa n  2^{k-1}}$, we conclude that \[
\Prob(End_{t}) \ge 1 - (1+1/\kappa) e^{-0.01 \kappa t}. \quad (t=n, n+1, n+2, \ldots).
\]

We now claim: \begin{claim}
Let $\w \in \big(\cap_{k =1}^{\infty} F_{k}\big) \cap \big(\cap_{t=n}^{\infty} End_{t}\big)$. Then for each $t \ge n$ and for each interval $I$ such that \[
I \subseteq \big(s_{0.5\kappa n  + 1}^{(1)} w_{0.5\kappa n  + 1}^{(1)} \cdots s_{\kappa n  }^{(1)} w_{\kappa n }^{(1)} \big)^{-1} \cdot \mathcal{I},
\]
we have $\Len (Z_{t} I) \le e^{-0.01 \kappa t}$.
\end{claim}

To prove the claim let $t \ge n$ and let $k\ge 1$ be such that $n (2^{k} - 1) \le t \le n(2^{k+1} - 1)$. If $k=1$, the claim follows from the definition that  \[
Z_{n ( 2^{k}-1)}^{-1} \cdot w_{0}^{(k)} s_{1}^{(k)} w_{1}^{(k)} \cdots s_{0.5  \kappa n 2^{k-1} }^{(k)} w_{0.5 \kappa n  2^{k-1} }^{(k)} = \big(s_{0.5\kappa n+ 1}^{(1)} w_{0.5\kappa n  + 1}^{(1)} \cdots s_{\kappa n  }^{(1)} w_{\kappa n }^{(1)} \big)^{-1}.
\]
and that $\w \in End_{t}$. When $k$ is larger than $1$, we note that \[\begin{aligned}
& \quad Z_{t} Z_{n ( 2^{k}-1)}^{-1} \cdot w_{0}^{(k)} s_{1}^{(k)} w_{1}^{(k)} \cdots s_{0.5  \kappa n 2^{k-1} }^{(k)} w_{0.5 \kappa n  2^{k-1} }^{(k)} \cdot \mathcal{I}\\
& \supseteq Z_{t} Z_{n ( 2^{k}-1)}^{-1} \cdot w_{0}^{(k)} s_{1}^{(k)} w_{1}^{(k)} \cdots s_{0.9  \kappa n 2^{k-1} -1}^{(k)} w_{0.9 \kappa n  2^{k-1} -1}^{(k)} \cdot \mathcal{I} \\
&\supseteq Z_{t} Z_{n ( 2^{k}-1)}^{-1} \cdot w_{0}^{(k)} s_{1}^{(k)} w_{1}^{(k)} \cdots s_{0.9  \kappa n 2^{k-1} -1}^{(k)} w_{0.9 \kappa n  2^{k-1} -1}^{(k)} \cdot 
\\
&\quad s_{0.9 \kappa n2^{k-1} }^{(k)}  w_{0.9 \kappa n2^{k-1} }^{(k)}  \cdots s_{ \kappa n 2^{k-1} }^{(k)} w_{ \kappa n 2^{k-1} }^{(k)} \cdot w_{0}^{(k-1)} s_{1}^{(k-1)} w_{1}^{(k-1)} \cdots s_{0.1   \kappa n 2^{k-2} }^{(k-1)}w_{0.1  \kappa n  2^{k-2} }^{(k-1)} \mathcal{I} \\
&= Z_{t} \cdot Z_{n(2^{k-1} - 1)}^{-1} \cdot w_{0}^{(k-1)} s_{1}^{(k-1)} w_{1}^{(k-1)} \cdots s_{0.1   \kappa n 2^{k-2} }^{(k-1)}w_{0.1  \kappa n  2^{k-2} }^{(k-1)} \mathcal{I}.
\end{aligned}
\]
Here, the first inclusion is due to the fact that $s \mathcal{I} \subseteq \mathcal{I}$ and $w_{i}^{(k)} \mathcal{I} \subseteq \mathcal{I}$ for any $s \in S$ and any $w_{i}^{(k)}$ for each $i \neq 0, \kappa n 2^{k-1}$. The second inclusion is because of $\w \in F_{k-1}$, and the third equality is using Equation \ref{eqn:gnProd}.

We can keep going like this and arrive at the inclusion \[
Z_{t} Z_{n ( 2^{k}-1)}^{-1} \cdot w_{0}^{(k)} s_{1}^{(k)} w_{1}^{(k)} \cdots s_{0.5  \kappa n 2^{k-1} }^{(k)} w_{0.5 \kappa n  2^{k-1} }^{(k)} \cdot \mathcal{I} \supseteq Z_{t} \cdot Z_{n(2^{1}-1)}^{-1} \cdot w_{0}^{(1)} s_{1}^{(1)} w_{1}^{(1)} \cdots s_{l}^{(1)} w_{l}^{(1)} \mathcal{I}
\]
for any $l$ between $0.1 \kappa n$ and $\kappa n -1$ (thanks to the fact that $s\mathcal{I} \subseteq \mathcal{I}$ and $w_{i}^{(1)} \mathcal{I} \subseteq \mathcal{I}$). By using the relation for $l=0.5 \kappa n$ and the fact that $\w\in End_t$,  we establish the claim.

Finally, let us estimate the probability of \[
Dec := \left\{\w :  \begin{array}{c}\Len \Big(\big(s_{0.5\kappa n  + 1}^{(1)} w_{0.5\kappa n  + 1}^{(1)} \cdots s_{\kappa n  }^{(1)} w_{\kappa n }^{(1)} \big)^{-1} \cdot \mathcal{I}\Big) \\
= 1 - \Len \Big((w_{\kappa n}^{(1)})^{-1} (s_{\kappa n}^{(1)})^{-1} \cdots (w_{0.5\kappa n  + 1}^{(1)} )^{-1} ( s_{0.5\kappa n  + 1}^{(1)} )^{-1} \cdot (S^{1} \setminus \mathcal{I}) \Big)\\
 \ge 1-0.01^{n/4}, \\
\textrm{$\big(s_{0.5\kappa n  + 1}^{(1)} w_{0.5\kappa n  + 1}^{(1)} \cdots s_{\kappa n  }^{(1)} w_{\kappa n }^{(1)} \big)^{-1} \cdot \mathcal{I}$ contains $x$} \end{array} \right\}.
\]

Here, $S^{1} \setminus \mathcal{I}$ is a median for $\check{S}$, the reversed version of $S$ and $(s_{i}^{(1)})^{-1}$'s are independently Schottky-uniform on $\check{S}$. Moreover,  $(w_{i}^{(1)})^{-1} (S^{1} \setminus \mathcal{I}) \subseteq S^{1} \setminus \mathcal{I}$ holds for each $i \neq 0, \kappa n$. Hence, we can apply Lemma \ref{lem:expShrink2} and conclude that $(w_{\kappa n}^{(1)})^{-1} (s_{\kappa n}^{(1)})^{-1} \cdots (w_{0.5\kappa n  + 1}^{(1)} )^{-1} ( s_{0.5\kappa n  + 1}^{(1)} )^{-1} \cdot (S^{1} \setminus \mathcal{I})$ has length smaller than $0.01^{n/4}$ for probability $\ge 1-e^{-\kappa n/8}$. Furthermore, we can apply Lemma \ref{lem:expShrink4} to conclude that  $(w_{\kappa n}^{(1)})^{-1} (s_{\kappa n}^{(1)})^{-1} \cdots (w_{0.5\kappa n  + 1}^{(1)} )^{-1} ( s_{0.5\kappa n  + 1}^{(1)} )^{-1} \cdot (S^{1} \setminus \mathcal{I})$ does not contain $x$ for probability $1-e^{-\kappa n/8}$. Thus, $\Prob(Dec) \ge 1- 2e^{-\kappa n/8}$.

In conclusion, we have found a set $\big(\cap_{k =1}^{\infty} F_{k}\big) \cap \big(\cap_{t=n}^{\infty} End_{t}\big) \cap Dec$, whose complement has exponentially decaying probability in $n$, such that for each sample $\w$ in the set, there exists an open neighborhood of $x$ of length at least $1-2 e^{-\kappa n / 8}$ that gets exponentially contracted at every step $t \ge n$. This neighborhood is a homeomorphic image of $\mathcal{I}$, so is an interval if $\mathcal{I}$ is an interval. Furthermore, the constant $\kappa$ depends only on the choice of $S$, $\epsilon$ and $m$. Hence, the same estimates hold for $\mu'$-random walks for any $\mu' \in \mathcal{U}$ in place of $\mu$. This ends the proof of Theorem \ref{thm:main2} and \ref{thm:main3}.
\end{proof}

\section{Pivoting technique}\label{section:pivoting}

In this section, we explain Gou{\"e}zel's pivoting technique that was introduced in \cite{gouezel2022exponential}. It was later applied to a broader setting in \cite{choi2022random1}. 

As a warm-up, we observe the following.

\begin{lem}\label{lem:pivotingPrep1}
For each $\epsilon>0$ and $m \in \Z_{>0}$, there exists  $\kappa = \kappa(\epsilon, m)$ such that the following holds.

Let $S$ be a Schottky set and let $\mu$ be a probability measure on $\Homeo(S^{1})$ such that 
$\mu^{\ast m}$ is an $(S, \epsilon)$-admissible measure. Then for each  $n \in \Z_{>0}$, there exists a probability space $\Omega_{n}$, a measurable subset $A_{n} \subseteq \Omega_{n}$, a measurable partition $\mathcal{P}_{n} = \{\mathcal{E}_{\alpha}\}_{\alpha}$ of $A_{n}$, and $\Homeo(S^{1})$-valued random variables \[
Z_{n}, \{w_{i}\}_{i=0, \ldots, \lfloor \kappa n \rfloor}, \{r_{i}, s_{i}, t_{i}\}_{i=1, \ldots, \lfloor \kappa n\rfloor }
\] that satisfy the following. \begin{enumerate}
\item $\Prob(A_{n}) \ge 1 - \frac{1}{\kappa}e^{-\kappa n}$.
\item When restricted on each equivalence class $\mathcal{E} \in \mathcal{P}_{n}$,  $w_{0}, \ldots, w_{\lfloor \kappa n \rfloor}$ are each fixed constant maps and $r_{i}, s_{i}, t_{i}$ are independent random variables distributed according to a Schottky-uniform measure on $S$.
\item $Z_{n}$ is distributed according to $\mu^{\ast n}$ on $\Omega_{n}$, and \[
Z_{n}  = w_{0} r_{1}s_{1} t_{1}w_{1} \cdots r_{\lfloor \kappa n \rfloor }s_{\lfloor \kappa n \rfloor }t_{\lfloor \kappa n \rfloor } w_{\lfloor\kappa n \rfloor}
\] holds on $A_{n}$.
\end{enumerate}
\end{lem}

\begin{proof}
It suffices to prove this for $n$ being a multiple of $3m$, with $\kappa = \epsilon^{3}/10m$. Indeed, for $n=3mk+l$ ($1 \le l \le 3m-1$) we can treat as follows: we first take $\Omega_{3mk}$, $\mathcal{P}_{3mk}$, $A_{3mk}$, $(w_{i})_{i}, (r_{i}, s_{i}, t_{i})_{i}$ using the proposition and consider $(G^{l}, \mu^{l})$ (where $G = \Homeo(S^{1})$). And then we define 
 \[\begin{aligned}
\Omega_{3mk+l} &:= \Omega_{3mk} \times  G^{l},\\
\mathcal{P}_{3mk+l} &:= \mathcal{P}_{3mk} \times G^{l} = \{ \mathcal{E}_{\alpha} \times (g_{1}, \ldots, g_{l}) : \mathcal{E}_{\alpha} \in \mathcal{P}_{3mk}, (g_{1}, \ldots, g_{l}) \in G^{l} \}, \\
A_{3mk+l} &:= A_{3mk} \times G^{l}.
\end{aligned}
\]
We then keep $(w_{i})_{i}, (r_{i}, s_{i}, t_{i})_{i}$ but replace $w_{\lfloor \kappa n \rfloor}$ with  $w_{\lfloor \kappa n \rfloor} \cdot  g_{1} \cdots g_{l}$ to realize the conclusions for $n=3mk+l$.

We now begin our proof for $3m | n$. Since $\mu^{\ast m}$ is $(S, \epsilon)$-admissible, we can construct a probability measure $\mu_{S}$ that is Schottky-uniform on $S$ and another probability measure $\nu$ on $\Homeo(S^{1})$ such that \[
\mu^{\ast 3m} = \epsilon^{3} \mu_{S}^{\ast 3} + (1-\epsilon^{3}) \nu
\]
holds. Now, we construct Bernoulli RVs $(\rho_{i})_{i=0}^{\infty}$ with expectation $\epsilon^{3}$, RVs $(\eta_{i}^{(1)})_{i=1}^{\infty}$, $(\eta_{i}^{(2)})_{i=1}^{\infty}$ and $(\eta_{i}^{(3)})_{i=1}^{\infty}$ each distributed according to $\mu_{S}$, RVs $(\nu_{i})_{i=1}^{\infty}$ distributed according to $\nu$, all independently, and then define $g_{i}$'s by \[
\textrm{$g_{i} := \eta_{i}^{(1)}\cdot \eta_{i}^{(2)} \cdot\eta_{i}^{(3)}$ when $\rho_{i} = 1$,\,\, $g_{i} =\nu_{i}$ when $\rho_{i} = 0$}.
\]
This way, $g_{1}, g_{2}, \ldots$ become i.i.d.s distributed according to $\mu^{\ast 3m}$. We now collect the indices at which $\rho_{i}$ attains value 1: \[
\{ i(1)< i(2) < \ldots \} := \{ 1 \le i \le n/3m : \rho_{i} = 1 \},\,\, N := \# \{ 1 \le i \le n/3m : \rho_{i} = 1 \}.
\]
Then Markov's inequality implies \[
e^{-\epsilon^{3} n/10m} \cdot \Prob\big( N \le \epsilon^{3} n / 10m \big) \le \E \left[ e^{-B(n/3m, \epsilon^{3})}\right] =  \big(1 - \epsilon^{3}(1 - e^{-1})\big)^{n/3m} \le (1 - 0.6 \epsilon^{3})^{n/3m} \le e^{-0.6 \cdot \epsilon^{3} n/3m}.
\]
Hence, the probability of $N\le \epsilon^{3} n/10m$ is at most $e^{- \epsilon^{3} n/10m}$. Now on the event $\{N \ge \epsilon^{3} n/10m\}$ we construct  \[
\begin{aligned}
w_{0} &:= \prod_{i=1}^{i(1)-1} g_{i} = \nu_{1} \cdots \nu_{i(1)-1}, \\
w_{l} &:= \prod_{i=i(l)+1}^{i(l+1) - 1} g_{i} = \nu_{i(l)+1} \cdots \nu_{i(l+1) - 1} & (l=1, \ldots, \lfloor  \epsilon^{3} n/10m \rfloor-1)\\
w_{\lfloor \epsilon^{3} n/10m\rfloor} &:= \prod_{i=i(\lfloor \epsilon^{3} n/10m\rfloor)+1}^{n/3m} g_{i} = \nu_{i(\lfloor \epsilon^{3} n/10m\rfloor) + 1} \cdots \nu_{n/3m}
\end{aligned}
\]
and set $(r_{l}, s_{l}, t_{l}) := \big(\eta_{i(l)}^{(1)}, \eta_{i(l)}^{(2)}, \eta_{i(l)}^{(3)}\big)$ for each  $l=1, \ldots, \lfloor \epsilon^{3} n/10m\rfloor$. Then \[
Z_{n} := g_{1} g_{2} \cdots g_{n/3m}= w_{0} r_{1}s_{1} t_{1}w_{1} \cdots r_{\lfloor \epsilon^{3} n/10m\rfloor }s_{\lfloor \epsilon^{3} n/10m\rfloor}t_{\lfloor \epsilon^{3} n/10m\rfloor } w_{\lfloor \epsilon^{3} n/10m\rfloor}
\]
is distributed according to  $\mu^{\ast n}$. We can then finish the proof by declaring the equivalence relation based on the values of  $\{\rho_{l}, \nu_{l} : 1\le l \le n/3m\}$ and $\big\{\eta_{l}^{(1)}, \eta_{l}^{(2)}, \eta_{l}^{(3)} : i(\lfloor \epsilon^{3} n / 10 m \rfloor)<l\le n/3m\big\}$.
\end{proof}

Let us now recall the trick we used in Lemma \ref{lem:1segmentZeta}.

\begin{dfn}\label{dfn:RepellingChoice}
Let $S =  \cup_{i=1}^{N} \mathfrak{S}(I_{i}, J_{i})$ be a Schottky set with resolution $N$ and with multiplicity $\zeta$. For each $g \in \Homeo(S^{1})$, we define  \[
\mathcal{C}(g; S) := \big\{ I_{i}: \# \{j : \bar{I}_{i} \cap g\bar{J}_{j} \neq \emptyset\} \ge \zeta \sqrt{N} \big\}.
\]
Furthermore, for each interval $I \subseteq S^{1}$ we define  \[
\mathcal{R}(I; S) := \{ J_{i} : \bar{J}_{i} \cap \bar{I} \neq \emptyset \}.
\]
\end{dfn}

\begin{lem}\label{lem:RepellingChoice}
Let $S$ be a Schottky set with resolution $N$ and let $g \in \Homeo(S^{1})$ be a homeomorphism. Then the cardinality of $\mathcal{C}(g; S)$ is at most $2 \zeta\sqrt{N}$. Furthermore, for every $I_{i} \notin \mathcal{C}(g; S)$, the cardinality of $\mathcal{R}(g^{-1}I_{i}; S)$ is at most  $ \zeta \sqrt{N}$.
\end{lem}

Before proceeding to the definition of pivotal times, we recall the notation introduced earlier: when a Schottky set $S = \cup_{i=1}^{N}\mathfrak{S}(I_{i}, J_{i})$ is understood, each element $s$ of $S$ belongs to some $\mathfrak{S}(I_{i}, J_{i})$. In this situation, we write $I(s)$ for $I_{i}$ and $J(s)$ for $J_{i}$.

Informally, an index ￼$i$ counts as \emph{pivotal} if the random letters $r_{i}, s_{i}, t_{i}$ nests the relevant Schottky interval so that replacing $s_{i}$ by many other elements of $S$ ￼ does not disturb this nesting. This is the analogue for  circle homeomorphisms  of Gouëzel’s alignment condition for hyperbolic spaces. Let us now dive into the definition.

\begin{dfn}\label{dfn:pivot}
Let \[
S := \cup_{i=1}^{N} \mathfrak{S}(I_{i}, J_{i})
\]
be a Schottky set with resolution $N$, with multiplicity $\zeta$ and with a median $\mathcal{I}$. Fixing a sequence $\mathbf{w} = (w_{i})_{i=0}^{\infty}$ in $\Homeo(S^{1})$, we draw sequences $\mathbf{r}= (r_{i})_{i=1}^{\infty}, \mathbf{s} =(s_{i})_{i=1}^{\infty}, \mathbf{t}=(t_{i})_{i=1}^{\infty}$ from $S$. We use the following recursive notation: \[
W_{0} := w_{0}, \,\, W_{n} := W_{n-1} \cdot r_{n}s_{n}t_{n} \cdot w_{n} \,\,(n>0).
\]
For each $n \in \mathbb{Z}_{\ge 0}$, we define the \emph{pivotal subset} $L_{n}=L_{n}(\mathbf{r}, \mathbf{s}, \mathbf{t}; \mathbf{w}) \subseteq S^{1}$ and the \emph{set of pivotal times} $P_{n}(\mathbf{r}, \mathbf{s}, \mathbf{t}; \mathbf{w}) \subseteq \{1, \ldots, n\}$ in the following recursive manner:\begin{enumerate}
\item for $n=0$, we let  $L_{0} := \mathcal{I}, P_{0} := \emptyset$.
\item for each $n \ge 1$, we divide into the following two cases:\begin{enumerate}[label=(\Alph*)]
\item If both $J(r_{n}) \subseteq W_{n-1}^{-1} L_{n-1}$ AND $I(t_{n}) \notin \mathcal{C}(w_{n}; S)$ hold true, then we define \[
L_{n}:= W_{n-1} r_{n} s_{n}t_{n} \big(S^{1} \setminus I(t_{n}) \big),\,\, P_{n} := P_{n-1} \cup \{n\}.
\]
\item If either $J(r_{n}) \subseteq W_{n-1}^{-1} L_{n-1}$ OR $I(t_{n}) \notin \mathcal{C}(w_{n}; S)$ does not hold, we consider the set \[
\mathcal{Q} := \Big\{i \in P_{n-1} : I(t_{i}) \notin \mathcal{C}\big( w_{i} \cdot W_{i}^{-1} \cdot W_{n}; S\big) \Big\}.
\] \begin{enumerate}
\item If $\mathcal{Q}$ is nonempty, we set $k := \max \mathcal{Q}$ and define \[
L_{n} := W_{k-1} r_{k} s_{k}t_{k} \big(S^{1} \setminus I(t_{k}) \big), \,\, P_{n} := P_{n-1} \cap \{1, \ldots, k\}.
\]
\item If $\mathcal{Q}$ is empty, then we set $L_{n} := W_{n}\mathcal{I}$, $P_{n} := \emptyset$.
\end{enumerate}
\end{enumerate}

\end{enumerate}
\end{dfn}

The following observation is immediate.
\begin{lem}\label{lem:pivotDependence}
In the setting of Definition \ref{dfn:pivot}, for each $n \in \mathbb{Z}_{>0}$, the outputs $P_{n}(\mathbf{r}, \mathbf{s}, \mathbf{t}; \mathbf{w})$ and $L_{n}(\mathbf{r}, \mathbf{s}, \mathbf{t}; \mathbf{w})$ depend only on the values of $(r_{i}, s_{i}, t_{i})_{i=1}^{n}, (w_{i})_{i=0}^{n}$ and not on the values of  $(r_{i}, s_{i}, t_{i}, w_{i})_{i=n+1}^{\infty}$.
\end{lem}

Next, we observe that the images of $\mathcal{I}$ at the pivotal times are nested. This follows from:
\begin{lem}\label{lem:pivotAlign1step}
In the setting of Definition \ref{dfn:pivot}, let $u \in \Z_{>0}$ and let $l<m$ be two consecutive elements in $P_{u}$, i..e, $l, m \in P_{u}$ and  $l =\max (P_{u} \cap \{1, \ldots, m-1\})$. Then we have \begin{equation}\label{disp:inclusionInd}
W_{l-1} r_{l} s_{l} \mathcal{I} \supseteq W_{l-1} r_{l} s_{l} t_{l} \big(S^{1} \setminus I(t_{l})\big) \supseteq W_{m-1} r_{m} \mathcal{I}.
\end{equation}
\end{lem}

\begin{proof}
Recall first the property of the median $\mathcal{I}$ of the Schottky set $S$: for every $t \in S$, we have $t^{-1} A \subseteq I(t)$ for every $A \subseteq S^{1} \setminus \mathcal{I}\subseteq S^{1} \setminus J(t)$. Consequently, $S^{1} \setminus I(t) \subseteq t^{-1} \mathcal{I}$ for every $t \in S$. This explains the first  inclusion in Display \ref{disp:inclusionInd}. For the second inclusion, we claim that:
\begin{claim}
The index $l$ must have been added when $P_{l}$ was constructed out of $P_{l-1}$. In other words, $P_{l} = P_{l-1} \cup \{l\}$ holds.
\end{claim}

Suppose to the contrary that $P_{l}$ is a subset of  $P_{l-1} \subseteq \{1, \ldots, l-1\}$. Then not only $P_{l}$, but all of $P_{l+1}$, $P_{l+2}$, $\ldots$ cannot contain $l$. This is because there is no mechanism so that $l$ can be added at the time of the construction $P_{l+1}$, $P_{l+2}$, $\ldots$. This contradicts the fact that $P_{u} \ni l$, and the claim follows.

For a similar reason, we have $m \in P_{m}$. Hence, scenario (2-A) must have held at step $n=l$ and $n=m$. Next, we assert that:

\begin{claim}
$P_{u} \cap \{1, \ldots, m-1\} = P_{m-1}$ holds.
\end{claim}
First, note that the elements of $P_{u}$ smaller than or equal to $m-1$ must have been acquired no later than step $m-1$, and then must have never been lost thereafter. Hence, they all belong to $P_{m-1}$. Meanwhile, all elements in $P_{m-1} $ should have remained till step $u$ for the following reason. If an element of  $P_{m-1}\subseteq \{1, \ldots, m-1\}$ was lost at some step $n\in \{m,m+1, \ldots, u\}$, it would mean that scenario (2-B) was the case at step $n$, with $k=\max \mathcal{Q}$ being smaller than $m-1$. This means that $P_{n}$ lost not only $P_{m-1}$ but also $m$, which contradicts $P_{u} \ni m$. Hence the claim follows.

We now finish the proof by dividing into two cases.\begin{enumerate}
\item $l = m-1$: this means that scenario (2-A) was the case at both step $l$ and step $m=l+1$.  Hence, $J(r_{l+1}) \subseteq W_{l}^{-1} L_{l} :=w_{l}^{-1} \big(S^{1} \setminus I(t_{l})\big) $ must hold. This implies  \[
r_{l+1} \mathcal{I} \subseteq J(r_{l+1}) \subseteq w_{l}^{-1}\big(S^{1} \setminus I(t_{l})\big) ,\,\,
W_{l} r_{l+1} \mathcal{I} \subseteq W_{l}  w_{l}^{-1}\big(S^{1} \setminus I(t_{l}) \big)= W_{l-1} r_{l} s_{l} t_{l}\big(S^{1} \setminus I(t_{l})\big)  \]
as desired.
\item $l < m-1$: in this case, $P_{m-1} = P_{u} \cap \{1, \ldots, m-1\} \subseteq \{1, \ldots, l\}$ does not contain $m-1$ so scenario (2-B) must have been the case at step $n=m-1$. Still, $P_{m-1} = P_{u} \cap \{1, \ldots, m-1\}$ contains an element $l$ so scenario (2-B-ii) is ruled out. Thus, scenario (2-B-i) was the case and $l$ must have been the maximum element of $\mathcal{Q}$. This leads to  $L_{m-1} := W_{l-1} r_{l} s_{l}t_{l}\big( S^{1} \setminus I(t_{l}) \big)$. We now know that scenario (2-A) was the case at step $n=m$, which implies $J(r_{m}) \subseteq W_{m-1}^{-1} L_{m-1}$. Hence, we conclude\[
W_{m-1}r_{m}\mathcal{I}\subseteq W_{m-1} J(r_{m})\subseteq L_{m-1} =  W_{l-1} r_{l} s_{l} t_{l} \big(S^{1} \setminus I(t_{l}) \big). \qedhere
\]
\end{enumerate}
\end{proof}

Recall once again that $\mathcal{I} \supseteq s\mathcal{I}$ for every $s \in S$. This combined with Lemma \ref{lem:pivotAlign1step} implies:
\begin{cor}
\label{cor:pivotAlign}
In the setting of Definition \ref{dfn:pivot}, let $P_{n} := \{i(1) < i(2) < \ldots < i(\#P_{n})\}$. Then we have  \[
\begin{aligned}
W_{i(1)-1} r_{i(1)} \mathcal{I} \supseteq W_{i(1)-1} r_{i(1)} s_{i(1)} \mathcal{I} \supseteq \ldots \supseteq W_{i(\#P_{n})-1} r_{i(\#P_{n})} \mathcal{I} \supseteq W_{i(\#P_{n})-1} r_{i(\#P_{n})}s_{i(\#P_{n})} \mathcal{I}
\end{aligned}
\]
\end{cor}

Next, we will observe that scenario (2-A) in Definition \ref{dfn:pivot} has high chance, when $\mathbf{r}, \mathbf{s}, \mathbf{t}$ are drawn based on a Schottky-uniform measure.

\begin{lem}\label{lem:pivotGain}
Let $S$ be a Schottky set with resolution $N$, with multiplicity $\zeta$ and with a median $\mathcal{I}$, and let $n \in \Z_{>0}$. Fix a sequence $\mathbf{w} = (w_{i})_{i=0}^{\infty}$ in $\Homeo(S^{1})$ and  a sequence  $\mathbf{s} = (s_{i})_{i=1}^{\infty}$ in $S$. Further, fix two sequences   $\mathbf{r} = (r_{i})_{i=1}^{\infty}$, $\mathbf{t} = (t_{i})_{i=1}^{\infty}$ in $S$ \emph{except the $n$-th entries}. Then for any Schottky-uniform probability measure on  $S$, we have \[
\Prob_{r_{n}, t_{n}\textrm{: i.i.d. $\sim \mu$}}\big( \#P_{n}(\mathbf{r}, \mathbf{s}, \mathbf{t} ; \mathbf{w}) = \# P_{n-1} (\mathbf{r}, \mathbf{s}, \mathbf{t}; \mathbf{w}) + 1 \big) \ge 1 - 3\zeta/\sqrt{N}.
\]
\end{lem}

\begin{proof}
Let $S = \cup_{i=1}^{N} \mathfrak{S}(I_{i}, J_{i})$ for essentially disjoint subsets $\{I_{i}, J_{i}\}_{i}$. Note that the set $P_{n-1}$ and the interval $L_{n-1}$ are determined from the fixed inputs. Now at step $n-1$ of the pivotal set construction, three possibilities arise: \begin{enumerate}
\item scenario (2-A) holds: Then we have $I(t_{n-1}) \notin \mathcal{C}(w_{n-1};S)$ and $L_{n-1} =W_{n-2} r_{n-1} s_{n-1} t_{n-1} \big(S^{1} \setminus I(t_{n-1})\big)$.
\item scenario (2-B-i) holds: Then $I(t_{k}) \notin \mathcal{C}(w_{k} W_{k}^{-1} W_{n-1}; S)$ and $L_{n-1} := W_{k-1} r_{k} s_{k} t_{k} \big(S^{1} \setminus I(t_{k})\big)$ holds for $k=\max P_{n-1}$.
\item scenario (2-B-ii) holds: Then $L_{n-1} := W_{n-1}\mathcal{I}$ contains every $W_{n-1}J_{i}$.
\end{enumerate}

The event under consideration is equivalent to saying that scenario (2-A) holds at step $n$.  First, Lemma \ref{lem:RepellingChoice} asserts that \[
\Prob_{t_{n} \sim \mu} \big( I(t_{n}) \in \mathcal{C}(w_{n}; S) \big) \le \frac{\zeta}{N} \cdot 2\sqrt{N} = \frac{2\zeta}{\sqrt{N}}.
\]
Let us now observe the condition $J(r_{n}) \subseteq W_{n-1}^{-1} L_{n-1}$ in each of the three cases at step $n$.\begin{enumerate}
\item scenario (2-A) holds: using Lemma \ref{lem:RepellingChoice} and the fact that $I(t_{n-1}) \notin \mathcal{C}(w_{n-1}; S)$, we realize that  $\mathcal{R}\big(w_{n-1}^{-1} I(t_{n-1}); S\big)$ has at most  $\zeta\sqrt{N}$ elements. Moreover, when $J(r_{n}) \notin \mathcal{R}\big(w_{n-1}^{-1} I(t_{n-1});S\big)$ holds true, \[
J(r_{n}) \subseteq S^{1} \setminus  w_{n-1}^{-1} I(t_{n-1})= W_{n-1}^{-1} W_{n-2} r_{n-1} s_{n-1} t_{n-1}  \big(S^{1} \setminus I(t_{n-1})\big) = W_{n-1}^{-1} L_{n-1}
\]
also follows. In view of this, we conclude \[
\Prob_{r_{n} \sim \mu} \big(J(r_{n}) \subseteq W_{n-1}^{-1} L_{n-1}\big) \ge \Prob_{r_{n} \sim \mu} \big(J(r_{n}) \notin \mathcal{R}\big(w_{n-1}^{-1} I(t_{n-1}); S\big)\big) \ge 1 - \zeta /\sqrt{N}.
\]
\item scenario (2-B-i) holds: using Lemma \ref{lem:RepellingChoice} and the $I(t_{k}) \notin \mathcal{C}(w_{k}W_{k}^{-1} W_{n-1}; S)$, we deduce that $\mathcal{R}\big(W_{n-1}^{-1} W_{k} w_{k}^{-1} I(t_{k}); S\big)$ has at most  $\zeta \sqrt{N}$ elements. Moreover, when $J(r_{n}) \notin \mathcal{R}\big(W_{n-1}^{-1} W_{k} w_{k}^{-1} I(t_{k});S\big)$ holds true, \[
J(r_{n}) \subseteq S^{1} \setminus W_{n-1}^{-1} W_{k} w_{k}^{-1}I(t_{k})  = W_{n-1}^{-1} W_{k-1} r_{k} s_{k} t_{k} \big(S^{1} \setminus I(t_{k}\big)) = W_{n-1}^{-1} L_{n-1}
\]
follows. Now a calculation analogous to the one in Item (1) tells us that  $J(r_{n}) \subseteq W_{n-1}^{-1} L_{n-1}$ happens for probability at least $1-\zeta/\sqrt{N}$.
 \item scenario (2-B-ii) holds: Then whatever $J(r_{n})$ is among $J_{1}, \ldots, J_{N}$, $J(r_{n}) \in W_{n-1}^{-1} L_{n-1} = \mathcal{I}$ holds.
\end{enumerate}
Based on our estimates for the probabilities for $I(t_{n}) \notin \mathcal{C}(w_{n};S)$ and $J(r_{n}) \subseteq W_{n-1}^{-1} L_{n-1}$ in the above three cases, we can conclude that $\#P_{n+1} = \#P_{n}+1$ happens for probability at least $1-3\zeta /\sqrt{N}$.
\end{proof}

We now prove a crucial lemma. Roughly speaking, it asserts that changing choices for $\mathbf{s}$ at the pivotal times does not change the set of pivotal times.

\begin{lem}\label{lem:pivotEquiv}
Let $S$ be a Schottky set with a median $\mathcal{I}$, let $n \in \Z_{>0}$ and let $\mathbf{w} = (w_{i})_{i=0}^{n}$ be a sequence in $\Homeo(S^{1})$. Let  $\mathbf{r} = (r_{i})_{i=1}^{\infty}, \mathbf{s}=(s_{i})_{i=1}^{\infty}, \bar{\mathbf{s}}=(\bar{s}_{i})_{i=1}^{\infty}, \mathbf{t}=(t_{i})_{i=1}^{\infty} $ be sequences in $S$. If  we have:\[
\textrm{$s_{i} = \bar{s}_{i}$ for each $i \in \{1, \ldots, n\} \setminus P_{n}(\mathbf{r}, \mathbf{s}, \mathbf{t} ; \mathbf{w})$},
\]
then $P_{l}(\mathbf{r}, \mathbf{s}, \mathbf{t} ; \mathbf{w}) = P_{l}(\mathbf{r}, \bar{\mathbf{s}}, \mathbf{t}; \mathbf{w})$ holds for each $1 \le l \le n$.
\end{lem}

\begin{proof}
As an elementary version of this lemma, let us consider:
\begin{claim}
In the setting as above, let $k \in P_{n}(\mathbf{r}, \mathbf{s}, \mathbf{t}; \mathbf{w})$ be an arbitrary pivotal time. If $s_{l} = \bar{s}_{l}$ holds for all $l \neq k$, then $P_{l}(\mathbf{r}, \mathbf{s}, \mathbf{t} ; \mathbf{w}) = P_{l}(\mathbf{r}, \bar{\mathbf{s}}, \mathbf{t}; \mathbf{w})$ holds for all $1 \le l \le n$.
\end{claim}

Put in other words, changing the choice at a \emph{single} pivotal time does not change the set of pivotal times (at step $1, \ldots, n$). Assuming this claim, in the setting of lemma, we can move from $\mathbf{s}$ to $\bar{\mathbf{s}}$ by changing the choices at the pivotal times, one per each time; then $P_{l}$'s remain unchanged, and the desired statement holds.

It remains to prove the claim. We will omit $\mathbf{w}, \mathbf{r}, \mathbf{t}$ in the sequel as they are fixed forever. When $l$ is smaller than $k$, $P_{l}(\mathbf{s})$ only depends on $s_{1}, \ldots, s_{k-1}$ and $\mathbf{w}, \mathbf{r}, \mathbf{t}$, so it coincides with $P_{l}(\bar{\mathbf{s}})$. Similarly, the value of $L_{l}$ should coincide for the two inputs.

At step $l=k$, we note that $k \in P_{n}(\mathbf{s})$. Hence, scenario (2-A) must have held. Here, note that the two conditions \[
J(r_{k}) \subseteq W_{k-1}^{-1} L_{k-1}, \,\, I(t_{k}) \notin \mathcal{C}(w_{k}; S)
\]
only depend on $s_{1}, \ldots, s_{k-1}$ (and other fixed inputs $\mathbf{w}, \mathbf{r}, \mathbf{t})$. Hence, these conditions are unchanged after switching $s_{k}$ to $\bar{s}_{k}$, and we have  \[
P_{k}(\bar{\mathbf{s}}) = P_{k-1}(\bar{\mathbf{s}}) \cup \{k\} = P_{k-1}(\mathbf{s}) \cup \{k\} = P_{k}(\mathbf{s}).
\]
At this moment, note the relation \[\begin{aligned}
L_{k}(\mathbf{s}) = W_{k-1} r_{k} s_{k} t_{k}\big(S^{1} \setminus I(t_{k})\big), \,\,L_{k}(\bar{\mathbf{s}}) = W_{k-1} r_{k} \bar{s}_{k} t_{k}\big(S^{1} \setminus I(t_{k})\big) = g \cdot W_{k-1} r_{k} s_{k}  t_{k}\big(S^{1} \setminus I(t_{k}) \big) \\
\big(\textrm{$g := W_{k-1} r_{k} \bar{s}_{k} s_{k}^{-1} r_{k}^{-1} W_{k-1}^{-1}$}\big).
\end{aligned}
\]
and $W_{l}(\bar{\mathbf{s}}) = g \cdot W_{l}(\mathbf{s})$ for each $k \le l \le n$.

Now, we inductively prove the following for $k< l \le n$: \begin{enumerate}
\item If scenario (2-A) holds at step $l$ for the input $\mathbf{s}$, the same is true for the input $\bar{\mathbf{s}}$.
\item If scenario (2-B-i) holds at step $l$ for the input $\mathbf{s}$, the same is true for the input $\bar{\mathbf{s}}$
\item scenario (2-B-ii) does not happen at step $l$.
\item In every case, $P_{l}(\mathbf{s}) = P_{l}(\bar{\mathbf{s}})$ and $L_{l}(\bar{\mathbf{s}}) =g L_{l}(\mathbf{s})$ hold.
\end{enumerate}
As the base case, we have observed Item (4) for  $l=k$. For general $k<l \le n$, we will start by assuming Item(4) for $l-1$. Recall the conditions for scenario (2-A) at step $l$, for the input $\mathbf{s}$:\[
J(r_{l}) \subseteq W_{l-1}(\mathbf{s})^{-1} L_{l-1}(\mathbf{s}), \,\, I(t_{l}) \notin \mathcal{C}(w_{l}; S).
\]
The latter one is clearly independent of the inputs $\mathbf{s}$. Furthermore, the inductive hypothesis tells us that  \[
\big[ J(r_{l}) \subseteq W_{l-1} (\mathbf{s})^{-1} L_{l-1}(\mathbf{s}) \big] \Leftrightarrow \big[
J(r_{l}) \subseteq W_{l-1}(\mathbf{s})^{-1} g^{-1} \cdot g L_{l-1}(\mathbf{s}) =W_{l-1}(\bar{\mathbf{s}})^{-1}  L_{l-1}(\bar{\mathbf{s}})\big].
\]
In summary,  scenario (2-A) at step $l$ for the input $\mathbf{s}$ is equivalent to the one for $\bar{\mathbf{s}}$. Furthermore, when these equivalent conditions hold true,  \[
P_{l}(\bar{\mathbf{s}}) = P_{l-1}(\bar{\mathbf{s}}) \cup \{l\} = P_{l-1}(\mathbf{s}) \cup \{l\} = P_{l}(\mathbf{s})
\]
and \[
L_{l} (\bar{\mathbf{s}}) :=  W_{l}(\bar{\mathbf{s}}) \cdot w_{l}^{-1}\big(S^{1}\setminus I(t_{l})\big)= g  W_{l}(\mathbf{s}) \cdot w_{l}^{-1}\big(S^{1}\setminus I(t_{l})\big)=g  L_{l} (\mathbf{s})
\]
also holds.

If scenario (2-B) holds for the input  $\mathbf{s}$, the same is true for $\bar{\mathbf{s}}$ because of the observation just before. We then focus on the set \[
\mathcal{Q}(\mathbf{s}) =\mathcal{Q}(\mathbf{s};l):= \Big\{ i \in P_{l-1} : I(t_{i}) \notin \mathcal{C}\big(w_{i} \cdot W_{i}(\mathbf{s})^{-1} W_{l} (\mathbf{s}); S\big) \Big\} 
\]
Here, recall that $W_{i}(\bar{\mathbf{s}}) = g W_{i}(\mathbf{s})$ for $i\ge k$. This implies that \[
\mathcal{Q}(\mathbf{s}; l) \cap \{k, k+1, \ldots, l-1\} = \mathcal{Q}(\bar{\mathbf{s}}; l) \cap \{k, k+1, \ldots, l-1\}.
\]
 Meanwhile, we know that $k$ is alive in  $P_{n}(\mathbf{s})$. This means that $k$ must not have been lost at step $l$. In other words, when scenario (2-B) holds at step $l$, $\mathcal{Q}(\mathbf{s};l)$ must contain an element greater than or equal to $k$. Hence, scenario (2-B-ii) is ruled out.

For this reason, $\mathcal{Q}(\bar{\mathbf{s}}; l) \cap \{k, k+1, \ldots, l-1\} = \mathcal{Q}(\mathbf{s}; l) \cap \{k, k+1, \ldots, l-1\}$ is nonempty. Because the maximum elements of $\mathcal{Q}(\mathbf{s})$ and $\mathcal{Q}(\bar{\mathbf{s}})$ are taken in these upper sections, we conclude that the two sets have the same maximum  $u \ge k$. We then conclude \[
P_{l}(\bar{\mathbf{s}}) = P_{l-1}(\bar{\mathbf{s}}) \cap \{1,\ldots, u\}= P_{l-1}(\mathbf{s})  \cap \{1,\ldots, u\} = P_{l}(\mathbf{s})
\]
and \[
L_{l} (\bar{\mathbf{s}}) := W_{u}(\bar{\mathbf{s}}) \cdot  w_{u}^{-1} \big(S^{1} \setminus I(t_{u}) \big) = g W_{u}(\mathbf{s}) w_{u}^{-1}\big(S^{1} \setminus I(t_{u}) \big)  =g  L_{l} (\mathbf{s})
\]
Here we used the fact that $u \ge k$. This ends the proof.
\end{proof}

Thanks to the previous lemma, we can now declare an equivalence relation based on the change of choices at the pivotal times, or in short, \emph{pivoting}.

\begin{dfn}\label{dfn:pivotEquivClass}
Let $S$ be a Schottky set with a median and let $\mathbf{w}$ be a sequence in $\Homeo(S^{1})$, as in the setting of Definition \ref{dfn:pivot}. We fix an integer $n \in \Z_{>0}$. Now, on the ambient set $S^{\Z_{>0}} \times S^{\Z_{>0}} \times S^{\Z_{>0}}$ parametrized by coordinates $(\mathbf{r}, \mathbf{s}, \mathbf{t})$, we declare the following equivalence relation: \[
\big[(\mathbf{r}, \mathbf{s}, \mathbf{t}) \sim_{n} (\bar{\mathbf{r}}, \bar{\mathbf{s}}, \bar{\mathbf{t}}) \big] \Leftrightarrow \left[\begin{array}{c}\textrm{$r_{i} =\bar{r}_{i}$ and $t_{i} = \bar{t}_{i}$ for each $i \in \Z_{>0} \setminus \{n+1\}$ and} \\ \textrm{$\bar{s}_{i} = s_{i}$ for each $i \in \Z_{>0} \setminus P_{n}(\mathbf{r}, \mathbf{s}, \mathbf{t} ; \mathbf{w})$}\end{array}\right].
\]
\end{dfn}
This is indeed an equivalence relation thanks to Lemma \ref{lem:pivotDependence}  and Lemma  \ref{lem:pivotEquiv}. Note that this equivalence relation crucially depends on the value of $n$.

In defining $\sim_{n}$ above, the coordinates $r_{n+1}$ and $t_{n+1}$ are free and are independent of any other coordinates. This is because the data $P_{n}(\mathbf{r}, \mathbf{s}, \mathbf{t}; \mathbf{w})$ depends only on $w_{i}, s_{i}, r_{i}, t_{i}$ for $i \le n$, and is completely independent of the choice of $r_{n+1}$ or $t_{n+1}$.

By abuse of notation, we will use $(\mathbf{r}, \mathbf{s}, \mathbf{t})$ for the coordinate functions on  $S^{\Z_{>0}} \times S^{\Z_{>0}} \times S^{\Z_{>0}}$; each element will be characterized by its value of $r_{1}, r_{2}, \ldots, s_{1}, s_{2}, \ldots, t_{1}, t_{2}, \ldots$. Now consider an arbitrary equivalence class  $\mathcal{E} \subseteq S^{\Z_{>0}} \times S^{\Z_{>0}} \times S^{\Z_{>0}}$ for  $\sim_{n}$. Then every element of $\mathcal{E}$ has a common ($n$-th step) set of pivotal times $P_{n}$, which we denote by $P_{n}(\mathcal{E})$. On $\mathcal{E}$, $r_{i}$ and $t_{i}$ can take arbitrary values in $S$ for $i=n+1$ and are fixed for $i \neq n+1$. On $\mathcal{E}$, $s_{i}$ can take arbitrary values in $S$ for $i \in P_{n}(\mathcal{E})$ and is fixed for $i \notin P_{n}(\mathcal{E})$.

When $S$ is endowed with a probability measure $\mu$, the ambient space $S^{\Z_{>0}} \times S^{\Z_{>0}} \times S^{\Z_{>0}}$ also becomes a probability space (with the product measure of $\mu$'s). Here, $r_{i}, s_{i}, t_{i}$'s become $\mu$-i.i.d.s. Now if we restrict ourselves on $\mathcal{E}$--an arbitrary equivalence class, $\{r_{i}, t_{i} : i\neq n+1\}$, $\{s_{i} : i \notin P_{n}(\mathcal{E})\}$ are all fixed constants and  $\{s_{i} : i \in P_{n}(\mathcal{E})\}$, $\{r_{n+1}, t_{n+1} \}$ are $\mu$-i.i.d.s.

\begin{prop}\label{prop:pivotDistbn}
Let $S$ be a Schottky set with a median $\mathcal{I}$ and with resolution $N$, and let $\mu$ be a Schottky-uniform measure on $S$. Fix a sequence $\mathbf{w}$  in $\Homeo(S^{1})$ and fix $n \in \Z_{>0}$. Let $\mathcal{E}$ be an equivalence class for $\sim_{n}$ given on $S^{\Z_{>0}} \times S^{\Z_{>0}}\times S^{\Z_{>0}}$. Then for each $j\ge0$, we have  \[
\Prob_{\textrm{$\{r_{i}, s_{i}, t_{i} : i > 0\}$: $\mu$-i.i.d.s}}\Big(\#P_{n+1} (\mathbf{r}, \mathbf{s}, \mathbf{t} ; \mathbf{w}) < \# P_{n}(\mathbf{r}, \mathbf{s}, \mathbf{t} ; \mathbf{w})-j \, \Big| \, \mathcal{E} \Big) \le (3\zeta/\sqrt{N})^{j+1}
\]
\end{prop}

\begin{proof}
For notational convenience, we denote $P_{n}(\mathcal{E})$, the common $n$-th step set of pivotal times, by $\{i(M) < i(M-1) < \ldots < i(2) < i(1)\}$ (with $M=\# P_{n}(\mathcal{E})$). Here, $M$ and $i(1), i(2), \ldots, i(M)$ are fixed information across $\mathcal{E}$, as well as  $\{w_{i} : i >0\}, \{r_{i}, t_{i} : i \neq n+1\}, \{s_{i} : i \notin P_{n}(\mathcal{E})\}$. In other words, elements in  $\mathcal{E}$ are determined  by the values of  $(s_{i(M)}, \ldots, s_{i(1)}, r_{n+1}, t_{n+1})$ which are $\mu$-i.i.d.s.

We will now define sets \[\begin{aligned}
A_{0} &\subseteq S \times S, \\
A_{1} (r_{n+1}, t_{n+1}) &\subseteq S, \\
A_{2} (s_{i(1)}, r_{n+1}, t_{n+1}) &\subseteq S, \\
\cdots, \\
A_{M} (s_{i(M-1)}, \ldots, s_{i(1)}, r_{n+1}, t_{n+1}) &\subseteq S
\end{aligned}
\]
and prove:
\begin{claim}\label{claim:AiInductive}
\begin{enumerate}
\item $\Prob_{\mu \times \mu}(A_{0}) \ge 1 - 3\zeta/\sqrt{N}$.
\item For every $(s_{i(M)}, \ldots, s_{i(1)}, r_{n+1}, t_{n+1}) \in S^{M+2}$, if $(r_{n+1}, t_{n+1}) \in A_{0}$ holds, then we have \[
\# P_{n+1} (s_{i(M)}, \ldots, s_{i(1)}, r_{n+1}, t_{n+1}) =\# P_{n}(s_{i(M)}, \ldots, s_{i(1)})+1.
\]
\item For every $1 \le l \le M$ and for every $(s_{i(l-1)}, \ldots, s_{i(1)}, r_{n+1}, t_{n+1}) \in S^{l+1}$ we have \[\Prob_{s_{i(l)} \sim \mu}\big(s_{i(l)} \in A_{l}(s_{i(l-1)}, \ldots, s_{i(1)}, r_{n+1}, t_{n+1})\big) \ge 1-3\zeta/\sqrt{N}.
\]
\item For every $(s_{i(M)}, \ldots, s_{i(1)}, r_{n+1}, t_{n+1}) \in S^{M+2}$ and  $1 \le l \le M$, whenever  $s_{i(l)}$ belongs to $A_{l} (s_{i(l-1)}, \ldots, s_{i(1)}, r_{n+1}, t_{n+1})$, we have\[
\# P_{n+1} (s_{i(M)}, \ldots, s_{i(1)}, r_{n+1}, t_{n+1}) \ge \# P_{n}(s_{i(M)}, \ldots, s_{i(1)}) - l
\]
\end{enumerate}
\end{claim}

Let us now prove the proposition from this claim. We let  \[
B_{0} := \big\{ (\mathbf{r}, \mathbf{s}, \mathbf{t}) \in \mathcal{E} : (r_{n+1}, t_{n+1}) \notin A_{0} \big\}
\]and inductively define  \[
B_{l} := \big\{ (\mathbf{r}, \mathbf{s}, \mathbf{t}) \in B_{l-1} : s_{i(l)} \notin A_{l} (\mathbf{r}, \mathbf{s}, \mathbf{t}; \mathbf{w})\big\}
\]
for $l=1, \ldots, M$. Then by Claim \ref{claim:AiInductive}(3), \[
\begin{aligned}
\Prob_{\mathcal{E}}  \big( B_{l} \big) &= \int_{(\mathbf{r}, \mathbf{s}, \mathbf{t}) \in B_{l-1}} \Prob_{s_{i(l)} \sim \mu} \big( s_{i(l)} \notin A_{l} \, \big| \, s_{i(l-1)}, \ldots, s_{i(1)}, r_{n+1}, t_{n+1} \big) \, d\mu(s_{i(l-1)}) \cdots d\mu(s_{i(1)}) \,d\mu(r_{n+1}) \,d\mu(t_{n+1}) \\
&\le \frac{3\zeta}{\sqrt{N}} \cdot \Prob_{\mathcal{E}}  \big(B_{l-1}) 
\end{aligned}
\]
holds. Moreover, Claim \ref{claim:AiInductive}(1) implies $\Prob_{\mathcal{E}} (B_{0}) \le 3\zeta/\sqrt{N}$. Combined together, we  observe $\Prob_{\mathcal{E}} (B_{l}) \le(3\zeta/\sqrt{N})^{l+1}$ for $l=0, \ldots, M$.

Next, \[
(\mathbf{r}, \mathbf{s}, \mathbf{t}) \in \mathcal{E} \setminus B_{0} \Rightarrow \#P_{n+1}(\mathbf{r}, \mathbf{s}, \mathbf{t}; \mathbf{w}) \ge \# P_{n}(\mathbf{r}, \mathbf{s}, \mathbf{t}; \mathbf{w})
\]
holds true; we also have    \[
(\mathbf{r}, \mathbf{s}, \mathbf{t}) \in B_{l-1} \setminus B_{l} \Rightarrow \#P_{n+1}(\mathbf{r}, \mathbf{s}, \mathbf{t}; \mathbf{w}) \ge \# P_{n}(\mathbf{r}, \mathbf{s}, \mathbf{t}; \mathbf{w}) - l
\]
for $l=1, \ldots, M$. In other words, we have  \[
(\mathbf{r}, \mathbf{s}, \mathbf{t}) \in \mathcal{E} \setminus B_{l} \Rightarrow \#P_{n+1}(\mathbf{r}, \mathbf{s}, \mathbf{t}; \mathbf{w}) \ge \# P_{n}(\mathbf{r}, \mathbf{s}, \mathbf{t}; \mathbf{w}) - l
\]
for each $l$. Since the probability of  $B_{l}$ is at most  $(3\zeta/\sqrt{N})^{l+1}$, the proposition follows.

It remains to prove the claim. The claim regarding $A_{0}$ was already established in Lemma \ref{lem:pivotGain}. That means, regardless of the values of $(s_{i(M)}, \ldots, s_{i(1)})$, we proved that the probability for $(r_{n+1}, t_{n+1})$ to satisfy $\#P_{n+1} = \# P_{n} + 1$ is at least  $1-3\zeta /\sqrt{N}$.  We will prove something more: we claim that the candidates for $r_{n+1}, t_{n+1}$ that make $\#P_{n+1} = \#P_{n} +1$  are independent of  $(s_{i(M)}, \ldots, s_{i(1)})$. When restricted to $\mathcal{E}$, $\# P_{n+1} (\mathbf{r}, \mathbf{s}, \mathbf{t}; \mathbf{w}) = \# P_{n} (\mathbf{r}, \mathbf{s}, \mathbf{t}; \mathbf{w}) +1$ holds if and only if  $I(t_{n+1}) \in \mathcal{C}(w_{n+1}; S)$ and $J(r_{n+1}) \subseteq W_{n}^{-1} L_{n} $. Here, \[
W_{n}^{-1} L_{n}  = \left\{ \begin{array}{cc}  \begin{array}{c} W_{n}^{-1} W_{\max P_{n}} w_{\max P_{n}}^{-1} \big(S^{1} \setminus I(t_{\max P_{n}})\big) \\= \big( w_{i(1)} r_{i(1)+1} s_{i(1) + 1} t_{i(1) + 1} w_{i(1) + 1} \cdots r_{n} s_{n} t_{n} w_{n}\big)^{-1} \big(S^{1} \setminus I(t_{i(1)})\big)\end{array} & (\textrm{when} \,\,P_{n}(\mathcal{E}) \neq \emptyset)  \\
W_{n}^{-1} W_{n-1} \mathcal{I} = \big(r_{n}s_{n}t_{n} w_{n}\big)^{-1} \mathcal{I} & (\textrm{when}\,\,P_{n}(\mathcal{E}) = \emptyset) \end{array}\right.
\]
are fixed throughout $\mathcal{E}$. This is why $\#P_{n+1} = \#P_{n} + 1$ depends on the choice of $r_{n+1}$ and $t_{n+1}$, regardless of the values of  $s_{i(1)}, \ldots, s_{i(M)}$. This settles Claim \ref{claim:AiInductive}(1), (2) and also the construction of $A_{0}$.

Now for each $l \in \{1, \ldots, M\}$ and for each choices $(s_{i(l-1)}, \ldots, s_{i(1)}, r_{n+1}, t_{n+1}) \in S^{l+1}$, we define \[
\begin{aligned}
A_{l} &:= \Big\{ s \in S : I(s) \notin \mathcal{C}\big( t_{i(l)} w_{i(l)} \cdot W_{i(l)}^{-1} W_{n+1} ; S\big) \Big\} \\
&= \Big\{ s \in S : I(s) \notin \mathcal{C}\big( t_{i(l)} w_{i(l)} \cdot (r_{i(l)+1} s_{i(l) + 1} t_{i(l)+1} w_{i(l) + 1}) \cdots (r_{n} s_{n} t_{n} w_{n}) \cdot (r_{n+1} s_{n+1} t_{n+1} w_{n+1}) ;S\big) \Big\}.
\end{aligned}
\]
Recall that $\{r_{i}, t_{i} : i \neq n+1\}$ and $\{s_{i} : i \notin P_{n}(\mathcal{E})\}$ are all fixed maps; hence, this $A_{l}$ depends only on the choices of  $s_{i(l-1)}, \ldots, s_{i(1)}$ and $r_{n+1}, t_{n+1}$. Furthermore, Lemma \ref{lem:RepellingChoice} tells us that $\Prob_{\mu}(A_{l}) \ge 1-2\zeta/\sqrt{N}$.

Now for an arbitrary $(s_{i(M)}, \ldots, s_{i(1)}, r_{n+1}, t_{n+1}) \in S^{M+2}$, suppose that $s_{i(l)} \in A_{l}(s_{i(l-1)}, \ldots, s_{i(1)}, r_{n+1}, t_{n+1})$. Then by definition we have \begin{equation}\label{eqn:estIJ}
\#\big\{ j : \bar{I}(s_{i(l)}) \cap t_{i(l)} w_{i(l)} \cdot W_{i(l)}^{-1} W_{n+1}\bar{J}_{j} \neq \emptyset\big\} \le  \zeta\sqrt{N}.
\end{equation} Meanwhile, Lemma \ref{lem:pivotAlign1step} tells us that \[
W_{i(l+1)-1} r_{i(l+1)} s_{i(l+1)}t_{i(l+1)} \big(S^{1} \setminus I(t_{i(l+1)})\big)\supseteq W_{i(l) - 1} r_{i(l)} \mathcal{I}.
\]
Finally, by the property of $\mathcal{I}$ as a median for $S$, we have $s_{i(l)}I(s_{i(l)}) \supseteq S^{1} \setminus \mathcal{I}$. Combining these two facts yields \[
W_{i(l+1)} w_{i(l+1)}^{-1} \bar{I}(t_{i(l+1)}) \subseteq \inte \big(  S^{1} \setminus W_{i(l)-1} r_{i(l)}  \mathcal{I}\big) \subseteq W_{i(l)-1} r_{i(l)} s_{i(l)}  \bar{I}(s_{i(l)}).
\]
Using Inequality \ref{eqn:estIJ}, we observe \[
\#\big\{ j : \bar{I}(t_{i(l+1)}) \cap (W_{i(l+1)} w_{i(l+1)}^{-1})^{-1} \cdot (W_{i(l) - 1} r_{i(l)} s_{i(l)}) \cdot t_{i(l)} w_{i(l)} \cdot W_{i(l)}^{-1} W_{n+1}\bar{J}_{j} \neq \emptyset\big\} \le \zeta\sqrt{N}.
\]
In other words,  $I(t_{i(l+1)}) \notin \mathcal{C}\big(w_{i(l+1)} \cdot (W_{i(l+1)})^{-1} \cdot W_{n+1}; S\big)$ holds true. This implies that the set $\mathcal{Q}$ in scenario (2-B) at step $n+1$ contains  $i(l+1)$. Hence, $P_{n+1}(\mathbf{r}, \mathbf{s}, \mathbf{t})$ contains $P_{n}(\mathcal{E}) \cap \{1,\ldots, i(l+1)\} = \{i(M) < \ldots < i(l+1)\}$ at least, which leads to the inequality $\#P_{n+1} \ge \#P_{n} - l$. This concludes Claim \ref{claim:AiInductive}(3), (4) and the entire proof.
\end{proof}

\begin{cor}\label{cor:pivotDistbn}
Let $S$ be a Schottky set with a median $\mathcal{I}$, with resolution $N$ and with multiplicity $\zeta$, and let $\mu$ be a Schottky-uniform measure on $S$. Fix a sequence $\mathbf{w}$  in $\Homeo(S^{1})$. When $S^{\Z_{>0}}\times  S^{\Z_{>0}}\times S^{\Z_{>0}}$ is endowed with the product measure of  $\mu$, we have the following for each integer $j, k, n \ge 0$:  \begin{equation}\label{eqn:condpivotDistbn}
\Prob \Big( \# P_{n+1}(\mathbf{r}, \mathbf{s}, \mathbf{t}; \mathbf{w}) < k - j \, \Big | \, \# P_{n}(\mathbf{r}, \mathbf{s}, \mathbf{t}; \mathbf{w}) = k \Big) \le (3\zeta/\sqrt{N})^{j+1}.
\end{equation}
\end{cor} 

\begin{proof}
First fix $n$ and give the equivalence relation $\sim_{n}$ on $(S^{\Z_{>0}})^{3}$. On each equivalence class, the $n$-th step set of pivotal times $P_{n}$ is fixed so its cardinality is also constant. Considering this, in order to prove Inequality \ref{eqn:condpivotDistbn} when conditioned on the size of $P_{n}$, it suffices to observe it on each equivalence class. This is then reduced to Proposition \ref{prop:pivotDistbn}.
\end{proof}

\begin{cor}\label{cor:pivotDistbn2}
Let $S$ be a Schottky set with resolution $N\ge 16\zeta^{2}$, with multiplicity $\zeta$ and with a median $\mathcal{I}$, and let $\mu$ be a Schottky-uniform measure on $S$. Fix a sequence $\mathbf{w}$  in $\Homeo(S^{1})$. Let $X_{1}, X_{2}, \ldots$ be i.i.d.s with distribution \begin{equation}\label{eqn:expRV}
\Prob(X_{i}=j) = \left\{\begin{array}{cc} 1 - 3\zeta/\sqrt{N} & \textrm{if}\,\, j=1,\\  \left(\frac{3\zeta}{\sqrt{N}}\right)^{-j}\left( 1 - \frac{3\zeta}{\sqrt{N}}\right)& \textrm{if}\,\, j < 0, \\ 0 & \textrm{otherwise.}\end{array}\right.
\end{equation}
When $S^{\Z_{>0}}\times  S^{\Z_{>0}}\times S^{\Z_{>0}}$ is endowed with the product measure of  $\mu$, $\#P_{n}$ dominates $X_{1} + \ldots + X_{n}$ in distribution for each $n$. That means,  \[
\Prob(\#P_{n}(\mathbf{r}, \mathbf{s}, \mathbf{t}; \mathbf{w}) \ge i) \ge \Prob( X_{1} + \ldots + X_{n} \ge i) \quad (\forall i \in \mathbb{Z}_{\ge 0}).
\]
\end{cor} 

\begin{proof}
Let $X_{i}$ be the RVs as in \ref{eqn:expRV}; we can require them to be independent from $S^{\Z_{>0}}\times  S^{\Z_{>0}}\times S^{\Z_{>0}}$, the ambient probability space on which $P_{1}, P_{2}, \ldots$ are define. Now Lemma  \ref{lem:pivotGain} and Corollary \ref{cor:pivotDistbn} tells us the following for each  $0 \le k \le n$, $i \ge 0$ and $j \in \{1, -1, -2, \ldots\}$:\begin{equation}\label{eqn:probDistbn}
\Prob\left( \# P_{k+1}(\mathbf{r}, \mathbf{s}, \mathbf{t};\mathbf{w})\ge i + j \, \Big| \, \#P_{k}(\mathbf{r}, \mathbf{s}, \mathbf{t};\mathbf{w}) = i\right) \ge  \left\{\begin{array}{cc} 1 - \frac{3\zeta}{\sqrt{N}} & \textrm{if}\,\, j=1,\\ 1 - \left(\frac{3\zeta}{\sqrt{N}}\right)^{-j+1} & \textrm{if}\,\, j < 0.\end{array}\right.
\end{equation}

Let us prove that for each $k=1, \ldots, n$ and for each $i \in \Z_{\ge 0}$ we have  $\Prob(\#P_{k} \ge i) \ge \Prob(X_{1} + \ldots + X_{k} \ge i)$. For $k=1$, the claim follows from Inequality \ref{eqn:probDistbn} because $\#P_{k-1} = 0$ always. Now, assuming the statement for $k$ as an induction hypothesis, we observe  \[\begin{aligned}
\Prob(\#P_{k+1}\ge i) &= \sum_{j \in \Z} \Prob\big(\#P_{k+1}\ge i \,\big| \#P_{k} = j\big) \cdot \Prob(\#P_{k} = j) \\
&\ge \sum_{j \in \Z} \Prob(X_{k+1}  \ge i-j) \cdot \Prob(\#P_{k} = j) \\
&= \sum_{u, v \in \Z, u+v \ge i} \Prob(X_{k+1} = u) \Prob(\#P_{k} = v) \\
&=\sum_{j\in \Z} \Prob(\#P_{k}\ge j) \Prob(X_{k+1} = i - j)\\
 & \ge \sum_{j\in \Z} \Prob(X_{1} + \cdots + X_{k} \ge j) \Prob(X_{k+1} = i-j) \\
 &= \Prob(X_{1} + \cdots + X_{k} + X_{k+1} \ge i). \qedhere
 \end{aligned}
\]
\end{proof}

\begin{cor}\label{cor:pivotDistbn3}
Let $S$ be a Schottky set  with multiplicity $\zeta$, with resolution $N \ge 9\cdot 10^{6} \zeta^{4}$ and with a median $\mathcal{I}$, and let $\mu$ be a Schottky-uniform measure on $S$. Fix a sequence $\mathbf{w}$  in $\Homeo(S^{1})$. When $S^{\Z_{>0}}\times  S^{\Z_{>0}}\times S^{\Z_{>0}}$ is endowed with the product measure of  $\mu$, we have\[
\Prob\big( \# P_{n}(\mathbf{r},\mathbf{s}, \mathbf{t}; \mathbf{w}) \le n/2 \big) \le   \bigg(3 \sqrt[4]{3\zeta /\sqrt{N}}\bigg)^{n} \le 0.6^{n}
\]
for each $n \in \Z_{>0}$.
\end{cor}

\begin{proof}
For convenience, we denote $3\zeta/\sqrt{N}$ by $a$. We will employ  Chebyshev's inequality. First recall the $X_{i}$'s in Display \ref{eqn:expRV}. We have\[\begin{aligned}
\E \left[ \sqrt{a}^{X_{i}} \right] &= \left(1-a\right) \cdot \left[\sqrt{a }+ \sum_{j=1}^{\infty} \sqrt{a}^{-j} \cdot a ^{j} \right] \\
&= \left(1-a\right) \sqrt{a} \left( 1 + \frac{1}{1-\sqrt{a}} \right) \\
&= 2 \sqrt{a} +  a -  \sqrt{a}^{3} \le 3 \sqrt{a}.
\end{aligned}
\]
Here, the last inequality used the fact that $\sqrt{a} \le 1$. Now Corollary \ref{cor:pivotDistbn2} and the independence of  $X_{i}$'s imply that \[
\E \left[ \sqrt{a}^{\#P_{n}(\mathbf{s})} \right] \le \E \left[ \sqrt{a}^{\sum_{i=1}^{n} X_{i}}\right] = \prod_{i=1}^{n} \E \left[ \sqrt{a}^{X_{i}} \right] \le \big(3 \sqrt{a}\big)^{n}.
\]
Now Chebyshev's inequality tells us that  \[
\E \left[ \sqrt{a}^{\#P_{n}(\mathbf{s})} \right] \ge\Prob(\# P_{n}(\mathbf{s}) \le n/2) \cdot \sqrt{a}^{n/2}.
\]
The conclusion follows by combining the two inequalities.
\end{proof}

We now finally prove Proposition \ref{prop:pivotingPrep}.

\begin{proof}[Proof of Proposition \ref{prop:pivotingPrep}]
In view of Lemma \ref{lem:pivotingPrep1}, it suffices to prove the following.

\begin{claim}
Let $S$ be a Schottky set with multiplicity $\zeta$, with resolution $N \ge 9\cdot 10^{6}\zeta^{2}$ and with a median $\mathcal{I}$, and let $\mu$ be a Schottky-uniform measure on $S$. Fix an integer $n \in \Z_{>0}$ and a sequence $\mathbf{w}$  in $\Homeo(S^{1})$. Let $\Omega =S^{\Z_{>0}}\times  S^{\Z_{>0}}\times S^{\Z_{>0}}$ be the probability space endowed with the product measure of  $\mu$. Then there exists a measurable subset $A$ of $\Omega$, a measurable partition  $\mathcal{P} = \{\mathcal{E}_{\alpha}\}_{\alpha}$ of $A$, and $\Homeo(S^{1})$-valued random variables $\{w_{i}'\}_{i=0, \ldots, \lfloor n/2 \rfloor}, \{s_{i}'\}_{i=1, \ldots, \lfloor n/2 \rfloor}$ such that the following hold:

\begin{enumerate}
\item $\Prob(A) \ge 1 - 0.6^{n}$.
\item When restricted on each equivalence class $\mathcal{E} \in \mathcal{P}$, the $w_{0}', \ldots, w_{\lfloor n/2 \rfloor}'$ are each fixed maps and the $s_{i}'$'s are $\mu$-i.i.d.s.
\item $w_{i}' \mathcal{I} \subseteq \mathcal{I}$ for each $i=1, \ldots, \lfloor n/2 \rfloor - 1$.
\item On $A$, the following equality holds: \[
w_{0} r_{1} s_{1} t_{1} w_{1} \ldots r_{n} s_{n} t_{n} w_{n} = w_{0}' s_{1}' w_{1}' \cdots s_{\lfloor n/2 \rfloor}' w_{\lfloor n/2 \rfloor}'.
\]
\end{enumerate}
\end{claim}

Corollary \ref{cor:pivotDistbn3} tells us that \[
\Prob \Big(A := \big\{ (\mathbf{r}, \mathbf{s}, \mathbf{t}) \in (S^{\Z_{>0}})^{3} : \#P_{n}(\mathbf{r}, \mathbf{s}, \mathbf{t}; \mathbf{w}) > n/2 \big\} \Big)\ge 1-0.6^{n}.
\]
 Next, we declare an equivalence relation on $(S^{\Z_{>0}})^{3}$ as follows: \[
\big[(\mathbf{r}, \mathbf{s}, \mathbf{t}) \sim_{n}' (\bar{\mathbf{r}}, \bar{\mathbf{s}}, \bar{\mathbf{t}}) \big] \Leftrightarrow \left[\begin{array}{c}\textrm{$r_{i} =\bar{r}_{i}$ and $t_{i} = \bar{t}_{i}$ for each $i \in \Z_{>0}$}, \\ \textrm{ $\bar{s}_{i} = s_{i}$ unless $i$ is among the $n/2$ smallest pivotal times of $P_{n}(\mathbf{r}, \mathbf{s}, \mathbf{t} ; \mathbf{w})$}\end{array}\right]
\]
As observed in Lemma  \ref{lem:pivotEquiv}, changing the coordinate of $\mathbf{s}$ at a pivotal time does not change the set of pivotal times, and hence does not change the ``$n/2$ smallest pivotal times". Therefore, $\sim_{n}'$ is indeed an equivalence relation. Note that the cardinality of the set of pivotal times is constant across each equivalence class, so every equivalence class is either contained in $A$ or disjoint from $A$. In other words, $A$ is a (disjoint) union of some equivalence classes and $\sim_{n}'$ restricts to an equivalence relation on $A$.

Next, fix a $\sim_{n}'$-equivalence class $\mathcal{E}$ contained in $A$. All its elements share the $n$-th step set of pivotal times $P_{n}(\mathcal{E})$, which we denote by $\{i(1) < i(2) < \ldots)\}$.
Since we are assuming $\mathcal{E} \subseteq A$, there are at least $n/2$ elements of $P_{n}(\mathcal{E})$. We then construct \[
\begin{aligned}
w_{0}' &:= W_{i(1) - 1} r_{i(1)} = w_{0} \cdot r_{1}s_{1}t_{1}w_{1} \cdots r_{i(1) - 1}s_{i(1)-1}t_{i(1)-1}w_{i(1)-1} r_{i(1)}, \\
w_{l}'&:= t_{i(l)} w_{i(l)} W_{i(l)}^{-1} W_{i(l+1) - 1} r_{i(l+1)} \\
&= t_{i(l)} w_{i(l)} \cdot r_{i(l)+1} s_{i(l)+1}t_{i(l)+1}w_{i(l)+1} \cdots r_{i(l+1) - 1}s_{i(l+1)-1}t_{i(l+1)-1}w_{i(l+1)-1} r_{i(l+1)}, & (l=1, \ldots,  \lfloor n/2 \rfloor) \\
w_{ \lfloor n/2 \rfloor}'&:= t_{i( \lfloor n/2  \rfloor)} w_{i( \lfloor n/2 \rfloor)} W_{i( \lfloor n/2 \rfloor)}^{-1} W_{n}\\
&= t_{i(\lfloor n/2  \rfloor)} w_{i(\lfloor n/2  \rfloor)} \cdot r_{i(\lfloor n/2  \rfloor)+1} s_{i(\lfloor n/2  \rfloor)+1}t_{i(\lfloor n/2  \rfloor)+1}w_{i(\lfloor n/2  \rfloor)+1} \cdots r_{n}s_{n}t_{n}w_{n} .
\end{aligned}
\]	
The definition of $\sim_{n}'$ tells us that the maps  $w_{0}', w_{1}', \ldots, w_{\lfloor n/2 \rfloor}'$ are fixed throughout $\mathcal{E}$. Moreover,  we observed in Lemma \ref{lem:pivotAlign1step} that  $w_{l}' \mathcal{I} \subseteq \mathcal{I}$ holds for $l=1, \ldots, \lfloor n/2 \rfloor - 1$. Furthermore, the $s_{l}' := s_{i(l)}$'s are $\mu$-i.i.d.s when restricted on $\mathcal{E}$. The equality \[
w_{0}' s_{1}' w_{1}'\cdots  s_{\lfloor n/2 \rfloor}' w_{\lfloor n/2 \rfloor}' = w_{0} r_{1}s_{1}t_{1} w_{1} \cdots r_{n}s_{n}t_{n}w_{n}
\]
is clear on $\mathcal{E}$. This ends the proof.
\end{proof}

\section*{Conflict of Interest}

The author has no conflict of interest.

\section*{Data availability}

The paper is self-contained.

\medskip
\bibliographystyle{alpha}
\bibliography{Tits}

@article{kleptsyn2004convergence,
	author = {Kleptsyn, V. A. and Nalskii, M. B.},
	doi = {10.1007/s10688-005-0005-9},
	fjournal = {Funktsional\cprime ny\u{\i} Analiz i ego Prilozheniya},
	issn = {0374-1990},
	journal = {Funktsional. Anal. i Prilozhen.},
	mrclass = {37H99 (37A15 37A50 37E10 60G50)},
	mrnumber = {2117507},
	mrreviewer = {Vadim A. Ka\u{\i}manovich},
	number = {4},
	pages = {36--54, 95--96},
	title = {Convergence of orbits in random dynamical systems on a circle},
	url = {https://doi.org/10.1007/s10688-005-0005-9},
	volume = {38},
	year = {2004}}

@book{navas2011groups,
	author = {Navas, Andr\'{e}s},
	doi = {10.7208/chicago/9780226569505.001.0001},
	edition = {Spanish},
	isbn = {978-0-226-56951-2; 0-226-56951-9},
	mrclass = {37E10 (37C05 37C15 37E45 57M60)},
	mrnumber = {2809110},
	pages = {xviii+290},
	publisher = {University of Chicago Press, Chicago, IL},
	series = {Chicago Lectures in Mathematics},
	title = {Groups of circle diffeomorphisms},
	url = {https://doi.org/10.7208/chicago/9780226569505.001.0001},
	year = {2011}}

@article{deroin2007sur-la-dynamique,
	author = {Deroin, Bertrand and Kleptsyn, Victor and Navas, Andr\'{e}s},
	doi = {10.1007/s11511-007-0020-1},
	fjournal = {Acta Mathematica},
	issn = {0001-5962},
	journal = {Acta Math.},
	mrclass = {37C85 (34D08 37E10)},
	mrnumber = {2358052},
	mrreviewer = {Shigenori Matsumoto},
	number = {2},
	pages = {199--262},
	title = {Sur la dynamique unidimensionnelle en r\'{e}gularit\'{e} interm\'{e}diaire},
	url = {https://doi.org/10.1007/s11511-007-0020-1},
	volume = {199},
	year = {2007}}

@article{MR413217,
	author = {Siebert, Eberhard},
	doi = {10.1214/aop/1176996091},
	fjournal = {The Annals of Probability},
	issn = {0091-1798},
	journal = {Ann. Probability},
	mrclass = {60B15},
	mrnumber = {413217},
	mrreviewer = {Kalyanapuram R. Parthasarathy},
	number = {3},
	pages = {433--443},
	title = {Convergence and convolutions of probability measures on a topological group},
	url = {https://doi.org/10.1214/aop/1176996091},
	volume = {4},
	year = {1976}}

@article{MR756386,
	author = {Antonov, V. A.},
	fjournal = {Vestnik Leningradskogo Universiteta. Matematika, Mekhanika, Astronomiya},
	issn = {0024-0850},
	journal = {Vestnik Leningrad. Univ. Mat. Mekh. Astronom.},
	mrclass = {85A05 (70F15 86A15)},
	mrnumber = {756386},
	number = {vyp. 2},
	pages = {67--76},
	title = {Modeling of processes of cyclic evolution type. {S}ynchronization by a random signal},
	url = {https://mathscinet.ams.org/mathscinet-getitem?mr=756386},
	year = {1984}}

@article{ghys1987sur-un-groupe,
	author = {Ghys, \'{E}tienne and Sergiescu, Vlad},
	doi = {10.1007/BF02564445},
	fjournal = {Commentarii Mathematici Helvetici},
	issn = {0010-2571},
	journal = {Comment. Math. Helv.},
	mrclass = {57S30 (57R30 57R32 57R50 58D05 58F18)},
	mrnumber = {896095},
	mrreviewer = {Takashi Tsuboi},
	number = {2},
	pages = {185--239},
	title = {Sur un groupe remarquable de diff\'{e}omorphismes du cercle},
	url = {https://doi.org/10.1007/BF02564445},
	volume = {62},
	year = {1987}}

@article{neumann1954groups,
	author = {Neumann, B. H.},
	doi = {10.1112/jlms/s1-29.2.236},
	fjournal = {The Journal of the London Mathematical Society},
	issn = {0024-6107},
	journal = {J. London Math. Soc.},
	mrclass = {20.0X},
	mrnumber = {62122},
	mrreviewer = {Don Higman},
	pages = {236--248},
	title = {Groups covered by permutable subsets},
	url = {https://doi.org/10.1112/jlms/s1-29.2.236},
	volume = {29},
	year = {1954}}

@article{peneau2025limit,
	author = {P{\'e}neau, Axel},
	bdsk-color = {7},
	journal = {arXiv preprint arXiv:2402.05751},
	title = {Limit theorems for a strongly irreducible product of independent random matrices under optimal moment assumptions},
	year = {2025}}

@article{chawla2022the-poisson,
	author = {Chawla, Kunal and Forghani, Behrang and Frisch, Joshua and Tiozzo, Giulio},
	bdsk-color = {7},
	journal = {arXiv preprint arXiv:2209.02114},
	title = {The Poisson boundary of hyperbolic groups without moment condition},
	year = {2022}}

@article{tits1972free,
	author = {Tits, J.},
	doi = {10.1016/0021-8693(72)90058-0},
	fjournal = {Journal of Algebra},
	issn = {0021-8693},
	journal = {J. Algebra},
	mrclass = {20.75},
	mrnumber = {286898},
	mrreviewer = {B. A. F. Wehrfritz},
	pages = {250--270},
	title = {Free subgroups in linear groups},
	url = {https://doi.org/10.1016/0021-8693(72)90058-0},
	volume = {20},
	year = {1972}}

@article{beklaryan2002on-analogues,
	author = {Beklaryan, L. A.},
	doi = {10.1023/A:1014890606194},
	fjournal = {Matematicheskie Zametki},
	issn = {0025-567X},
	journal = {Mat. Zametki},
	mrclass = {22F50 (37A99 37E10)},
	mrnumber = {1913605},
	number = {3},
	pages = {334--347},
	title = {On analogues of the {T}its alternative for groups of homeomorphisms of the circle and the line},
	url = {https://doi.org/10.1023/A:1014890606194},
	volume = {71},
	year = {2002}}

@article{gouezel2022exponential,
	author = {Gou\"{e}zel, S\'{e}bastien},
	doi = {10.2140/tunis.2022.4.635},
	fjournal = {Tunisian Journal of Mathematics},
	issn = {2576-7658},
	journal = {Tunis. J. Math.},
	mrclass = {60B15 (20F67 20P05 60F10)},
	mrnumber = {4533553},
	number = {4},
	pages = {635--671},
	title = {Exponential bounds for random walks on hyperbolic spaces without moment conditions},
	url = {https://doi.org/10.2140/tunis.2022.4.635},
	volume = {4},
	year = {2022}}

@article{choi2022random1,
	author = {Choi, Inhyeok},
	journal = {arXiv preprint arXiv:2207.06597v2},
	title = {Random walks and contracting elements {I}: Deviation inequality and limit laws},
	year = {2022}}

@article{margulis2000free,
	author = {Margulis, Gregory},
	doi = {10.1016/S0764-4442(00)01694-3},
	fjournal = {Comptes Rendus de l'Acad\'{e}mie des Sciences. S\'{e}rie I. Math\'{e}matique},
	issn = {0764-4442},
	journal = {C. R. Acad. Sci. Paris S\'{e}r. I Math.},
	mrclass = {37C85 (22E40 37E10 57S30 58D05)},
	mrnumber = {1797749},
	mrreviewer = {B. Sury},
	number = {9},
	pages = {669--674},
	title = {Free subgroups of the homeomorphism group of the circle},
	url = {https://doi.org/10.1016/S0764-4442(00)01694-3},
	volume = {331},
	year = {2000}}

@article{ghys2001groups,
	author = {Ghys, \'{E}tienne},
	fjournal = {L'Enseignement Math\'{e}matique. Revue Internationale. 2e S\'{e}rie},
	issn = {0013-8584},
	journal = {Enseign. Math. (2)},
	mrclass = {37C85 (37E10 57S05)},
	mrnumber = {1876932},
	mrreviewer = {Grant Cairns},
	number = {3-4},
	pages = {329--407},
	title = {Groups acting on the circle},
	url = {https://mathscinet.ams.org/mathscinet-getitem?mr=1876932},
	volume = {47},
	year = {2001}}

@article{malicet2017random,
	author = {Malicet, Dominique},
	doi = {10.1007/s00220-017-2996-5},
	fjournal = {Communications in Mathematical Physics},
	issn = {0010-3616},
	journal = {Comm. Math. Phys.},
	mrclass = {60B15 (60B10 60G50)},
	mrnumber = {3719548},
	mrreviewer = {Michael Voit},
	number = {3},
	pages = {1083--1116},
	title = {Random walks on {${\rm Homeo}(S^1)$}},
	url = {https://doi.org/10.1007/s00220-017-2996-5},
	volume = {356},
	year = {2017}}

@article{gilabert-vio2024probabilistic,
	author = {Gilabert Vio, Mart{\'\i}n},
	bdsk-color = {7},
	journal = {arXiv preprint arXiv:2412.08779},
	title = {Probabilistic Tits alternative for circle diffeomorphisms},
	year = {2024}}

\end{document}